\documentclass[11pt]{article}

\usepackage{amsmath,amsthm,amssymb}
\usepackage[usenames,dvipsnames]{xcolor}
\usepackage{enumerate}
\usepackage{graphicx}
\usepackage{cite}
\usepackage{comment}
\usepackage{bbm}
\usepackage[margin=1in]{geometry}

\usepackage[pdftitle={Gaussian curvature on random planar maps and Liouville quantum gravity},
  pdfauthor={Andres Contreras-Hip and Ewain Gwynne},
colorlinks=true,linkcolor=NavyBlue,urlcolor=RoyalBlue,citecolor=PineGreen,bookmarks=true,bookmarksopen=true,bookmarksopenlevel=2,unicode=true,linktocpage]{hyperref}

\theoremstyle{plain}
\newtheorem{thm}{Theorem}[section]
\newtheorem{theorem}[thm]{Theorem}
\newtheorem{cor}[thm]{Corollary}
\newtheorem{lemma}[thm]{Lemma}
\newtheorem{prop}[thm]{Proposition}
\newtheorem{conj}[thm]{Conjecture}

\def\@rst #1 #2other{#1}
\newcommand\MR[1]{\relax\ifhmode\unskip\spacefactor3000 \space\fi
  \MRhref{\expandafter\@rst #1 other}{#1}}
\newcommand{\MRhref}[2]{\href{http://www.ams.org/mathscinet-getitem?mr=#1}{MR#2}}

\theoremstyle{definition}
\newtheorem{defn}[thm]{Definition}
\newtheorem{remark}[thm]{Remark}

\numberwithin{equation}{section} 

\newcommand{\dsb}{\begin{adjustwidth}{2.5em}{0pt}
\begin{footnotesize}}
\newcommand{\dse}{\end{footnotesize}
\end{adjustwidth}}

\newcommand{\ssb}{\begin{adjustwidth}{2.5em}{0pt}}
\newcommand{\sse}{\end{adjustwidth}}

\newcommand{\aryb}{\begin{eqnarray*}}
\newcommand{\arye}{\end{eqnarray*}}
\def\alb#1\ale{\begin{align*}#1\end{align*}}
\def\allb#1\alle{\begin{align}#1\end{align}}
\newcommand{\eqb}{\begin{equation}}
\newcommand{\eqe}{\end{equation}}
\newcommand{\eqbn}{\begin{equation*}}
\newcommand{\eqen}{\end{equation*}}

\newcommand{\BB}{\mathbb}

\newcommand{\eqD}{\overset{d}{=}}
\newcommand{\ep}{\varepsilon}

\newcommand{\wt}{\widetilde}

\newcommand{\mcl}{\mathcal}

\newcommand{\rng}{\mathring}

\newcommand\p{\partial}
\newcommand\e{\varepsilon}

\newcommand\R{\mathbb{R}}
\newcommand\Z{\mathbb{Z}}

\newcommand\norm[1]{\lVert#1\rVert}

\let\originalleft\left
\let\originalright\right
\renewcommand{\left}{\mathopen{}\mathclose\bgroup\originalleft}
\renewcommand{\right}{\aftergroup\egroup\originalright}

\title{Gaussian curvature on random planar maps and Liouville quantum gravity}
 \date{ }
 \author{ 
\begin{tabular}{c} Andres A. Contreras Hip and Ewain Gwynne\\ \small University of Chicago \end{tabular}  
}

\begin{document}

\maketitle

\newcommand{\Cupper}{{\hyperref[eqn-bilip-def]{\mathfrak C_*}}}

\begin{abstract}
%Let $\mathcal{G}_\e$ denote the $\e$-mated CRT map whose vertex set is a Poisson point process, and let $(\Phi,D_{\Phi},\mu_{\Phi})$ be a Liouville quantum cone. Many results are concerned with the geometric properties of LQG surfaces such as the quantum cone, however many are limited to first order information such as geodesics due to the low regularity of LQG and the technical issues that arise. One natural question is how many geometric properties are shared between LQG and its smooth counterparts. The geometric property by excellence is the Gaussian curvature.
We investigate the notion of curvature in the context of Liouville quantum gravity (LQG) surfaces. We define the Gaussian curvature for LQG, which we conjecture is the scaling limit of discrete curvature on random planar maps. Motivated by this, we study asymptotics for the discrete curvature of $\epsilon$-mated CRT maps. More precisely, we prove that the discrete curvature integrated against a $C_c^2$ test function is of order $\epsilon^{o(1)},$ which is consistent with our scaling limit conjecture. On the other hand, we prove the total discrete curvature on a fixed space-filling SLE segment scaled by $\epsilon^{\frac{1}{4}}$ converges in distribution to an explicit random variable.
\end{abstract}

\tableofcontents

%\begin{figure}[ht!]
%\begin{center}
%\includegraphics[width=0.75\textwidth]{Pic1b.pdf} 
%\caption{\label{fig1} Type caption here
%}
%\end{center}
%\vspace{-3ex}
%\end{figure}

\bigskip
\noindent\textbf{Acknowledgements.}
We thank R\'emi Rhodes for helpful discussions. E.G.\ was partially supported by NSF grant DMS-2245832.

\section{Introduction}
Liouville quantum gravity is a canonical 1-parameter family of models of a random surface first introduced by \cite{polyakov-qg1}. LQG is conjectured to be the scaling limit of various random planar map models, which are planar graphs embedded in $\mathbb{C}$ in such a way that no two edges cross, viewed modulo orientation-preserving homomorphisms. There are cases where the convergence to LQG has already been proven: for example, uniform planar maps (uniform triangulations, uniform quadrangulations etc.) have been proven to converge in the Gromov-Hausdorff sense to $\sqrt{8/3}$-LQG (\cite{legall-uniqueness,miermont-brownian-map,lqg-tbm1,lqg-tbm2,hs-cardy-embedding}). Other random planar map models are expected to converge to LQG, see e.g. \cite{ghs-mating-survey}.

Let $\gamma \in (0,2),$ and let $Q$ be defined by $Q=\frac{\gamma}{2}+\frac{2}{\gamma}.$ Suppose that $\Phi$ is a GFF or a variant such as the quantum cone. Formally, the LQG metric tensor associated with $\Phi$ is defined as $e^{\gamma \Phi}(dx^2+dy^2).$ A distance metric $D_h$ and measure $\mu_h$ can be defined via approximation of $\Phi$ by appropriate mollifications (see \cite{kahane,rhodes-vargas-review,berestycki-gmt-elementary,shef-kpz,dddf-lfpp,dg-uniqueness}.

There is a vast literature on LQG: see \cite{gwynne-ams-survey,sheffield-icm,bp-lqg-notes,gkr-cft-survey} for surveys on the main techniques and results in the area. Many geometric properties about LQG surfaces have been studied, such as the confluence property of LQG geodesics established in \cite{gm-confluence}, the KPZ formula for Hausdorff dimensions of LQG subsets \cite{gp-kpz,shef-kpz,aru-kpz,ghs-dist-exponent,ghm-kpz,grv-kpz,gwynne-miller-char,bjrv-gmt-duality,benjamini-schramm-cascades,wedges,shef-renormalization}, and more general properties about the LQG metric's behavior \cite{ddg-metric-survey}. %Other examples include \cite{gwynne-ball-bdy,shef-zipper,grv-lbm}.
However, given the non smooth nature of LQG surfaces, most of the study of their geometry is restricted to first order information  which does not require derivatives of the underlying field. Moreover, LQG does not always behave as its smooth counterpart (e.g. LQG surfaces have Hausdorff dimension greater than $2,$ and the LQG measure and metric scale differently). Given this, it is natural to ask which geometric properties are shared with smooth 2d Riemannian manifolds, and to what extent notions from Riemannian geometry can be extended for LQG.

A central concept in Riemannian geometry is curvature. In this regard, a first question is whether there is a way of defining Gaussian curvature for LQG surfaces so that it extends the usual definition, and at the same time is compatible with the known LQG theory, in the sense that it arises naturally as a limit of the corresponding discrete curvature for all reasonable random planar map approximations. In this paper, we define the Gaussian curvature for LQG surfaces \eqref{contcurv}, and we conjecture that it arises as the scaling limit of the discrete curvature defined in \eqref{curvdef}. To this end, the first question to study is whether the discrete curvature converges in a reasonable random planar maps model, and if so, under what scaling. The particular random planar map model we look at in this paper is the Poisson mated CRT map defined in Section \ref{prelim}.  This implies that we need to study the asymptotics of this discrete curvature as the mesh size of the planar map tends to $0,$ which is the main goal of this paper and is the purpose of Theorem \ref{main}. After this, we study the total curvature on a ``canonical" set, namely space filling SLE segments. In Theorem \ref{mainCRTthm} we show that the scaling limit of this total curvature converges in law to an explicit random variable. We note that surprisingly the scaling for this quantity is different from that of the discrete Gaussian curvature when summed against a $C_c^2$ test function.

%we will focus on a variant of the mated CRT map, which has been discretized from an independently sampled Poisson point process, and consider the discrete curvature corresponding to this random planar map model. To present our results, let $\mathcal{G}_\e$ denote the $\e$-mated CRT map whose vertex set is a Poisson point process, and let $(\Phi,D_{\Phi},\mu_{\Phi})$ be a Liouville quantum cone.

\subsection{Curvature for LQG surfaces}
\begin{defn}\label{gff}
The whole plane Gaussian free field (GFF) $\Phi$ is the centered Gaussian random generalized function with covariance structure given by
\[
\mathrm{Cov}(\Phi(z),\Phi(w)) = G(v,w) = \log \left(\frac{\max\{\vert z\vert,1\},\max\{\vert w\vert,1\}}{\vert z-w\vert}\right).
\]
\end{defn}
Note that this correlation function corresponds to a GFF normalized so that the average over the unit circle is $0$~\cite[Section 2.1.1]{vargas-dozz-notes}.

\begin{defn} \label{def:lqg}
An LQG surface is an equivalence class defined as follows. We say pairs $(D_1,\Phi_1),(D_2,\Phi_2)$ consisting of a domain in $\mathbb C$ and a generalized function on the domain are equivalent if there exists a conformal map $f:D_1\to D_2$ with
\[
\Phi_2=\Phi_1\circ f^{-1}+Q \log\vert (f^{-1})'\vert , 
\]
where here we recall that $Q = 2/\gamma+\gamma/2$.
Then any equivalence class is called $\gamma$-LQG surface.
\end{defn}

If $\Phi$ is a GFF, or more generally if $\Phi = \Phi^0 + f$, where $\Phi^0$ is a GFF and $f$ is a possibly random continuous function, then for $\gamma \in (0,2)$ we can define the \emph{LQG area measure} $\mu_\Phi$ as a limit of regularized versions of $e^{\gamma \Phi} \,dz$, where $dz$ denotes two-dimensional Lebesgue measure~\cite{kahane,rhodes-vargas-review,berestycki-gmt-elementary,shef-kpz}. This measure is compatible with changes of coordinate, in the sense that if $(D_1,\Phi_1)$ and $(D_2,\Phi_2)$ are related as in Definition~\ref{def:lqg}, then the conformal map $f$ pushes forward $\mu_{\Phi_1}$ to $\mu_{\Phi_2}$~\cite[Proposition 2.1]{shef-kpz}.

One can intuitively define the Gaussian curvature associated with LQG as follows. If we consider a general metric in isothermal coordinates $e^{\phi}(dx^2+dy^2),$ then the Gaussian curvature is given by
\[
\frac{\Delta \phi}{2}e^{-\phi}.
\]
Given this, if $\Phi$ is the underlying field, after a formal computation one obtains that the Gaussian curvature on an LQG surface $K_\Phi$ is given by
\[
\int_{\mathbb{C}} K_\Phi(z)f(z) d\mu_\Phi(z) = \int_{\mathbb{C}} \frac{\gamma}{2} \Delta \Phi e^{-\gamma \Phi} f(z) e^{\gamma \Phi} dz = \int_{\mathbb{C}} \frac{\gamma}{2} \Delta \Phi(z) f(z) dz = \int_{\mathbb{C}} \frac{\gamma}{2} \Phi(z) \Delta f(z) dz
\]
where $f$ is any $C_c^2$ test function, and so it is natural to define the Gaussian curvature weakly by
\begin{equation}\label{contcurv}
\int_{\mathbb{C}} K_{\Phi}(z) f(z) d\mu_{\Phi}(z):= \int_{\mathbb{C}}\frac{\gamma}{2} \Phi(z) \Delta f(z) dz.
\end{equation}
Note that this is invariant under LQG coordinate change: indeed, suppose that $U$ is a bounded open set and that $f$ is a $C_c^2$ test function supported on $U.$ If $\psi:U \to V$ is a conformal map, then for ${\bar{\Phi}} = \Phi\circ \psi+Q\log \vert \psi'\vert$ we have
\[
\int_{U}K_{\bar{\Phi}}(z)f(z) d\mu_{\mu_{\bar{\Phi}}}(z) =\int_{U} \Delta f \frac{\gamma}{2} (\Phi\circ \psi)dz + \int_{U} \Delta f Q \log \vert \psi'\vert dz.
\]
Since $\log \vert \psi'\vert$ is harmonic and the Dirichlet energy is invariant under conformal transformations, after integrating by parts we obtain
\begin{eqnarray*}
\int_{U}K_{\Phi}(z)f(z) d\mu_{\Phi}(z) &=&\int_{U} \frac{\gamma}{2}\Phi\circ \psi (z) \Delta f(z)dz = \int_V\frac{\gamma}{2} \Phi(z) (\Delta f) \circ \psi^{-1}(z) \vert \psi'(z)\vert^{-2}dz\\
&=& \int_V \frac{\gamma}{2} \Phi (z) \Delta(f\circ \psi^{-1})(z) dz = \int_V K_\Phi (z) f\circ \psi^{-1}(z) d\mu_\Phi(z).
\end{eqnarray*}
\begin{comment}
\begin{eqnarray*}
\int_{\mathbb{C}}K_{\Phi}(z)f(z) d\mu_{\Phi}(z) &=& \int_{\mathbb{C}}\frac{\gamma}{2}\Phi(z) \Delta f(z) d\mu_{\Phi}\\
&=&\int_{\mathbb{C}}\frac{\gamma}{2} \Phi\circ \psi(z)\Delta f(z)d\mu_{\Phi}+ \\
&=&\int_{\mathbb{C}} \frac{\gamma}{2}\Phi(z) \Delta f(z)dz.
\end{eqnarray*}
\end{comment}

\subsection{Discrete curvature}

Suppose that $M$ is a planar map embedded in $\mathbb{C}.$ To define the curvature, suppose we set all angles at each corner of a face $F$ to be $\frac{\vert F\vert-2}{\vert F\vert}\pi,$ where $\vert F\vert$ is the number of vertices incident to $F.$ Inspired by the Gauss-Bonnet formula, it is natural to define the discrete curvature at a vertex $v$ as 
\[
K_{M}(v) := 2\pi - \sum_{F} \frac{\vert F\vert-2}{\vert F\vert}\pi
\]
where the sum is taken over all possible faces $F$ incident to $v.$ Note that in the case the graph $\mathcal{G}$ is a triangulation, this is equivalent to
\begin{equation}\label{curvdef}
K_{M}(v) = \frac{\pi}{3}\left(6-\deg (v)\right).
\end{equation}

We conjecture that the discrete curvature $K_M$ converges to the curvature $K_\Phi$ for any planar map model in the $\gamma$-LQG universality class.
\begin{conj}\label{conj}
Let $\{M_n\}_{n\geq 1}$ be a sequence of infinite random planar maps believed to be in the universality class of $\gamma$-LQG (e.g. uniform infinite triangulations for $\gamma=\sqrt{8/3}$). Suppose we embed the map in the plane via any ``reasonable embedding", such as the circle packing embedding or the Tutte embedding. Then for an appropriate scaling factor $c_n,$ we have that for any $C_c^2$ test function $f:\mathbb{C}\to \R,$
\[
\frac{1}{c_n}\sum_{v \in \mathcal{V}M_n} f(v)K_{M_n}(v) \xrightarrow{d} \int_{\mathbb{C}} f(z)K_\Phi(z) d\mu_\Phi(z)
\]
where $\Phi$ is an appropriate variant of the GFF. Moreover, the convergence in distribution holds jointly for any finite collection of $C_c^2$ test functions.
\end{conj}

\subsection{Mated CRT maps}\label{matedcrtprelims}

Mated CRT maps are a family of random planar map models which are in the universality class of $\gamma$-LQG for $\gamma \in (0,2)$. These maps are particularly natural since they are closely connected both to LQG and to other random planar map models. Indeed, as we discuss below, mated CRT maps have a direct connection to SLE and LQG thanks to the results of~\cite{wedges}. On the other hand, mated CRT maps can be seen as ``course grained approximations" to other types of random planar maps (e.g., uniform triangulations, spanning tree decorated maps, bipolar oriented maps) due to discrete mating of trees bijections. We refer to~\cite{ghs-mating-survey} for a survey of various results related to mated-CRT maps.

In other literature, the CRT map is defined by discretizing time using $\e\mathbb{Z}.$ However, in this paper we consider a variant which uses a Poisson point process of intensity $\e^{-1}$ instead. The two resulting graphs are not equivalent; however, they should both be in the same universality class. One can call the variant considered in this paper ``Poisson mated CRT map", but for the sake of brevity we will usually just refer to it as the mated CRT map (there is no risk of ambiguity since we will not consider any other variant). The advantage of the Poisson mated CRT map is that the space filling SLE segments depend locally on the underlying field (see Lemma \ref{locality}).

\begin{defn}\label{CRTdef}
Let $\e>0.$ We let $\Lambda^\e = \{y_j\}_{j \in \Z}$ denote a Poisson point process chosen on $\R$ with intensity $\e^{-1}$ defined so that $\{y_j\}\subset \R$ is an increasing sequence. Suppose that $(L,R):\R \to \R^2$ is a pair of correlated two sided standard linear Brownian motions, normalized such that $L_0=R_0=0$ and such that $\mathrm{corr}(L_t,R_t) = -\cos\left(\frac{\pi\gamma^2}{4}\right)$ for $t \neq 0.$ The mated CRT map is defined to be the random planar map $\tilde{\mathcal{G}}_\e$ obtained by mating (gluing together) discretized versions of the continuum random trees constructed from $L$ and $R,$ that is, two vertices $y_j,y_k \in \Lambda^\e$ such that $j<k$ are connected if either
\begin{equation}\label{CRTeqdef}
\left(\inf_{t \in [y_{j-1},y_j]}L_t\right) \vee \left(\inf_{t \in [y_{k-1},y_k]}L_t\right) \leq \left(\inf_{t \in [y_j,y_{k-1}]}L_t\right)
\end{equation}
or the same holds with $L$ replaced by $R.$ If $\vert j-k\vert \geq 2$ and \eqref{CRTeqdef} holds for both $L$ and $R,$ then $y_j,y_k$ are connected by two edges (see Figure \ref{MatedCRT}, left).
\end{defn}

One can draw the graph $\tilde{\mathcal{G}}_\e$ in the plane by connecting two vertices $y_j,y_k \in \Lambda^\e$ by an arc above or below the real line according to whether \eqref{CRTeqdef} holds for $L$ or $R$ respectively (see Figure \ref{MatedCRT}, right). This makes the mated-CRT map into a planar map.

\begin{figure}
\begin{center}
\includegraphics[width=0.4\textwidth]{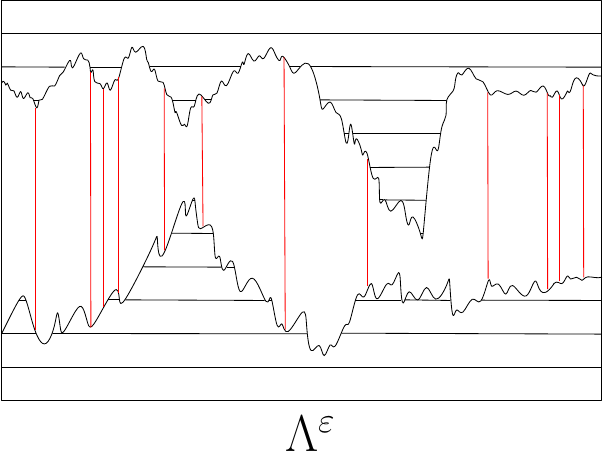}  \hspace{0.1\textwidth} 
\includegraphics[width=0.4\textwidth]{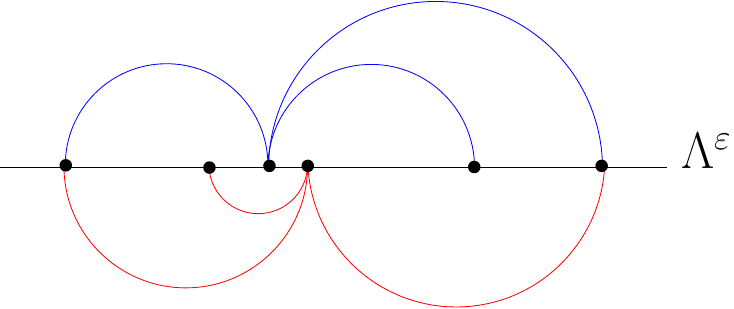}  
\caption{\label{MatedCRT}\textbf{Left:} Poisson mated CRT map. The bottom graph corresponds to the graph of $L,$ while the top graph corresponds to the graph of $C-R$ for some large enough constant $C.$ The red lines are the vertical identifications (hence the ``mating"), while the black horizontal lines correspond to the identifications when \eqref{CRTeqdef} holds.
\textbf{Right:} A portion of the Poisson mated CRT map. The top blue arcs correspond to  when \eqref{CRTeqdef} holds for $R,$ while the bottom red ones correspond to when \eqref{CRTeqdef} holds for $L.$ The vertex set is given by the Poison point process $\Lambda^\e.$}
\end{center}
\end{figure}

In fact, the mated-CRT map is a triangulation. Indeed, with the above embedding of the planar graph, each face of $\tilde{\mathcal{G}}_\e$ corresponds to a horizontal strip below the graph of $L$ or above the graph of $R,$ which is bounded by two horizontal segments on the real line between points in $\Lambda^\e$ and two segments of either the graph of $L$ or of $R.$ Almost surely, neither $L$ nor $R$ attains a local minimum at any point in $\Lambda^\e$ and neither $L$ nor $R$ has two local minima where it attains the same value. Hence a.s. one of the horizontal segments is a single point. This means that a.s. each face of $\tilde{\mathcal{G}}_\e$ has exactly three vertices on its boundary, and thus $\tilde{\mathcal{G}}_\e$ is a triangulation.

Note that the law of a mated CRT map does not depend on the parameter $\e,$ however the parameter $\e$ is relevant when taking scaling limits.

The mated CRT maps constructed above have a well established relation to SLE decorated LQG. Let $\tilde{\Phi}$ be the $\gamma$-quantum cone field (a particular variant of the GFF, see subsection \ref{qcone}), and let $\tilde{\eta}$ be a whole-plane space-filling $\mathrm{SLE}_{\kappa}$ from $\infty$ to $\infty$ sampled independently from $\tilde{\Phi},$ then parameterized by $\gamma$-quantum mass with respect to $\tilde{\Phi},$ where $\kappa= 16/\gamma^2>4.$ That is, $\mu_{\tilde\Phi}(\tilde\eta[a,b]) = b-a$ for each $a<b\in\BB R$, and $\tilde\eta(0) = 0$.

Then we can define the graph $\tilde{\mathcal{G}}_\e$ with vertex set $\Lambda^\e$ where two vertices $y_i, y_j$ are connected if the space filling SLE segments $\tilde{\eta}([y_{i-1},y_i])$ and $\tilde{\eta}([y_{j-1},y_j])$ (called \emph{cells}) share a nontrivial boundary arc, with two edges occurring between the two segments if they share nontrivial boundary arcs in both the left and right boundaries. By \cite{wedges}, the resulting graph agrees in law with the mated CRT map defined above. This construction yields a natural embedding of the mated CRT map into $\mathbb C$, called the \emph{SLE/LQG} embedding, by assigning each vertex $v \in \Lambda^\e$ to a point in the corresponding mated CRT map cell. It is convenient to take this point to be $\tilde\eta(v)$. See Section \ref{SLEstuff} for more details. We will use this fact throughout the proof of the main Theorems \ref{main} and \ref{mainCRTthm} below.

\subsection{Main results}

Henceforth when we talk about the mated CRT map $\tilde{\mathcal{G}}_\e,$ we will always assume that it is coupled with $(L,R,\Lambda^\e)$ as described at the end of Section \ref{matedcrtprelims}. In particular, $\Lambda^\ep \ni v \mapsto \tilde{\eta}(v)$ is the SLE/LQG embedding of the mated CRT map.
To show a result such as Conjecture \ref{conj} for $\tilde{\mathcal G}_\e$, one would not only need a definition of the Gaussian curvature that is compatible with the discrete setting, but also precise control over the discrete curvature in order to find the correct scaling factor, and later prove its scaling limit matches our notion of LQG curvature. This is also related to the problem of finding an observable on a random planar map converging to the underlying field: indeed, if $f$ is a $C_c^2$ test function and we assume that
\[
\frac{1}{a_\e}\sum_{v \in \mathcal{V}\tilde{\mathcal{G}}_\e}f(\tilde{\eta}(v))K_{\tilde{\mathcal{G}}_\e}(v) \to \int_{\mathbb{C}}f(z)K_{\tilde{\Phi}}(z)d\mu_{\tilde{\Phi}} (z) = \int_{\mathbb{C}} \Delta f(z) \tilde{\Phi}(z) dz,
\]
for some appropriate scaling factor $a_\e,$ then we have
\[
\frac{1}{a_\e} \sum_{v \in \mathcal{V}\tilde{\mathcal{G}}_\e}f(\tilde{\eta}(v))\Delta^{-1}K_{\tilde{\mathcal{G}}_\e}(v) \to \int_{\mathbb{C}} f(z) \tilde{\Phi}(z)dz,
\]
and so, at least formally, we should have that $\Delta^{-1}K_{\tilde{\mathcal{G}}_\e} \to \tilde{\Phi}$ in some weak sense.

Our main theorem is a first step towards proving the convergence of the discrete curvature. If we let $f$ be a $C_c^2$ test function, and we look at the expression $\sum_{v \in \mathcal{V}{\tilde{\mathcal{G}}}_\e} f(\tilde{\eta}(v))K_{\tilde{\mathcal{G}}_\e}(v),$ at first glance one could expect it to be of order $\e^{-\frac{1}{2}}$ since it is a sum over $\approx \e^{-1}$ terms of random signs times values of $f$ at the graph's vertices, and so if we expect square root cancellation one would obtain the order $\e^{-\frac{1}{2}}.$ However, our first main theorem tells us something much stronger: the sum $\sum_{v \in \mathcal{V}{\tilde{\mathcal{G}}}_\e} f(\tilde{\eta}(v))K_{\tilde{\mathcal{G}}_\e}(v)$ is in fact of a much smaller order.

Now we will state our first main theorem.

\begin{theorem}\label{main}
Let $\tilde{\mathcal{G}}_\e$ be the $\e$ mated CRT map defined as in Definition \ref{CRTdef}, with $\mathcal{V}\tilde{\mathcal{G}}_\e =\Lambda^\e$ defined to be its vertex set. Let $K_{\tilde{\mathcal{G}}_\e}$ denote the discrete curvature as defined in \eqref{curvdef}. Suppose that $f \in C_c^2(\mathbb{C})$ is any compactly supported twice differentiable function on $\mathbb{C}.$ Then with probability going to $1$ as $\e \to 0,$ we have that
\begin{equation}\label{weaksum}
\sum_{v \in \mathcal{V}\tilde{\mathcal{G}}_\e} f(\tilde{\eta}(v))K_{\tilde{\mathcal{G}}_\e}(v) = \e^{o(1)}.
\end{equation}
\end{theorem}
The $\e^{o(1)}$ asymptotics in the statement of Theorem \ref{main} is necessary: the next theorem tells us that if an appropriate normalization factor for the curvature summed against a $C_c^2$ test function exists, then it must be slow varying.

\begin{prop} \label{lem-scaling-asymp} 
Let $\tilde{\Phi}$ be the embedding of the $\gamma$-quantum cone used to define $\tilde{\mathcal{G}}_\e$. 
Assume that there exists a twice continuously differentiable, compactly supported function $f : \BB C \to \BB R$, which is not identically zero, and constants $\{a_\ep\}_{\ep > 0}$ such that we have the convergence in probability 
\eqb  \label{eqn-scaling-asymp}
  \frac{1}{a_\ep} \sum_{x \in \mcl V \tilde{\mathcal{G}}_\e }   K_{\tilde{\mathcal{G}}_\e}(x) f( \tilde{\eta}(x) )  \to \int_{\BB C} f(z)  K_{\tilde{\Phi}}(z) \,d \mu_{\tilde{\Phi}}(z) 
\eqe 
with respect to the weak topology on generalized functions. Then $\{  a_\ep\}_{\ep > 0}$ is slowly varying as $\ep\to 0$, i.e., 
\eqb \label{eqn-slow-var}
\lim_{\ep \to 0} \frac{a_{C\ep}}{a_\ep} = 1 ,\quad\forall C  > 0 . 
\eqe
\end{prop}

The proof is given in Section \ref{scaling-asymp-proof}.

One could ask whether the correct scaling factor should in fact be of constant order. We believe the answer is no. A heuristic argument for this is that if one makes a local change in $\tilde{\mathcal{G}}_\e$ such as artificially increasing the degree of a vertex by $1,$ then while the sum $\sum_{v \in \mathcal{V}\tilde{\mathcal{G}}_\e}K_{\tilde{\mathcal{G}}_\e}(v)f(\tilde{\eta}(v))$ is affected by a fixed amount for any $\e>0,$ the scaling limit should not be, hence the scaling factor must make such local changes negligible.

 Another interesting quantity is the total discrete curvature in a small region, such as a space filling SLE segment. More precisely, let $C=\tilde{\eta}([0,1]).$ We define the discrete total curvature on $C$ to be the sum of the discrete curvature on every vertex in $C,$
\begin{equation}\label{kgepsC}
K_{\tilde{\mathcal{G}}_\e}(C) = \sum_{v \in \mathcal{V}\tilde{\mathcal{G}}_\e \cap C} K_{\tilde{\mathcal{G}}_\e}(v).
\end{equation}
This represents the total curvature in a space filling SLE segment.

Our second main theorem is the following.
\begin{theorem}\label{mainCRTthm}
There exists a deterministic constant $\alpha>0$ such that
\[
\frac{K_{\tilde{\mathcal{G}}_\e}(C)}{\e^{-\frac{1}{4}}} \to \alpha\mathcal{B},
\]
in law, where the law of $\mathcal{B}$ can be described as follows. Sample $\mathcal{L}$ according to the law of the boundary length of the mated CRT map cell $C.$ Now let $\Theta$ be sampled according to a Gaussian distribution with mean $0$ and variance $\mathcal{L}.$ Then $\mathcal{B}$ and $\Theta$ have the same law.
\end{theorem}

Note that Theorem \ref{main} tells us that if there is a scaling limit of the discrete curvature summed against a $C_c^2$ test function, then the scaling used must be of order $\e^{o(1)}.$ However, Theorem \ref{mainCRTthm} tells us the correct scaling for the total curvature on a mated CRT map cell is $\e^{\frac{1}{4}}.$ This discrepancy suggests that it is impossible to define the Gaussian curvature together with the curvature on rough or lower dimensional sets, which is consistent with the fact that it is not so straightforward to define the total Gaussian curvature on a space-filling SLE segment since $\int_C \Delta \tilde{\Phi}(z) dz$ is not defined due to the poor regularity of $\Delta \tilde{\Phi}.$ Moreover, it seems that defining the Gaussian curvature integrated against $C_c^2$ test functions and the geodesic curvature on a mated CRT map cell in a consistent manner is impossible. Indeed, if one is to define the geodesic curvature on the boundary of space filling SLE segments, it must be defined in a manner that is compatible with the Gaussian curvature in the sense of the Gauss-Bonnet theorem, such that when ``integrated" over the boundary gives us the total curvature over a space-filling SLE segment. Since the total Gaussian curvature uses a different scaling factor than the Gaussian curvature, this suggests that not all notions of curvature (at least on boundaries of space-filling SLE segment) are simultaneously compatible. 

One of the main ingredients in the proof of Theorem \ref{main} is a cancellation when computing $\sum_{v \in \mathcal{V}\mathcal{G}_\e}f(\eta(v))K_{\mathcal{G}_\e}(v):$ one can rewrite this sum as a sum of discrete gradients of $f$ along appropriately oriented edges in $\mathcal{G}_\e.$ This is proved in Section \ref{secsec} via purely combinatorial arguments. After this, we rewrite the sum as an integral of a function which is constant in each mated CRT map cell. To estimate this integral, we estimate its variance and split the double integral into two parts: the diagonal and off diagonal parts (see Propositions \ref{ondiag} and \ref{offdiag} respectively). For the diagonal, we use the fact that the measure of the diagonal region is small. For the off diagonal region, we use the long range independence properties of the GFF (see \eqref{xepsindepprop}). These estimates are proved in subsections \ref{ondiagsubsec} and \ref{offdiagsubsec} respectively. After this, we transfer all our results back to the case of the quantum cone in Section \ref{quantumcone}. Part of our proof is inspired by techniques in \cite{gms-harmonic}, Section 3, however substantial new ideas are needed to get a very exact cancellation. In some sense, our argument pushes the technique in \cite{gms-harmonic} as far as possible.

For the proof of Theorem \ref{mainCRTthm}, the main ingredients are a discrete analogue of Gauss-Bonnet about sums degrees on triangulations (see Lemma \ref{cancellations}), together with Donsker's theorem to ensure the convergence of an average on the boundary of a mated CRT map cell.

\section{Preliminaries}\label{prelim}

Throughout the paper we will need the definitions of discrete gradient and Laplacian. We define them now. Suppose $\mathcal{G}$ is a triangulation of the plane. Suppose that this graph is embedded in $\mathbb{C}$ and suppose that $f:\mathcal{V}\mathcal{G} \to \mathbb{C},$ where $\mathcal{V}\mathcal{G}$ is the set of vertices in $\mathcal{G}.$ We define the discrete Laplacian $\Delta_{\mathcal{G}} f$ by
\[
\Delta_{\mathcal{G}} f(x) = \sum_{y \sim x}(f(y)-f(x)).
\]
Similarly, for any oriented edge $\overrightarrow{e}$ with starting vertex $v_1$ and ending vertex $v_2,$ we define
\begin{equation}\label{discgrad}
\mathcal{D} f(\overrightarrow{e}) = f(v_2)-f(v_1).
\end{equation}
We will slightly abuse notation, and define $\mathcal{D}f$ the same way for functions $f:\mathbb{C}\to\mathbb{C}$ by identifying vertices $v \in \mathcal{G}$ with their counterparts $\tilde{\eta}(v) \in \mathbb{C}.$

\subsection{Quantum cones}\label{qcone}
First we start with the definition of the quantum cone, an important field (random generalized function) which we will be using throughout the proof of Theorems \ref{main} and \ref{mainCRTthm}.
\begin{defn}
Let $\alpha < Q.$ An $\alpha$-quantum cone is a doubly marked quantum surface. Its law can be described in the following way. Let $A_s$ be the process defined by
\[
A_s =
\left\{
\begin{matrix}
B_s+\alpha s, & s>0\\
\hat{B}_{-s}+\alpha s, & s<0,
\end{matrix}
\right.
\]
where $B_s$ is a Brownian motion with $B_0=0,$ and $\hat{B}$ is a Brownian motion independent of $B$ conditioned so that $\hat{B}_t+(Q-\alpha)t > 0$ for all $t>0.$ We will define the spaces $\mathcal{H}_1(\mathbb{C}), \mathcal{H}_2(\mathbb{C})$ as follows: we let $\mathcal{H}_1(\mathbb{C})$ denote the set of radially symmetric functions in the Sobolev space $H^1(\mathbb{C}),$ and let $\mathcal{H}_2(\mathbb{C})$ denote the set of function in $H^1(\mathbb{C})$ with common mean on all circles centered at the origin. Then we define the $\alpha$-quantum cone field $\tilde{\Phi}$ as the field with projection onto $\mathcal{H}_1(\mathbb{C})$ given by the function with common value on $\p B(0,e^{-s})$ equal to $A_s$ for all $s \in \R,$ and whose projection onto $\mathcal{H}_2(\mathbb{C})$ is given by the projection onto $\mathcal{H}_2(\mathbb{C})$ of a whole plane GFF as defined in Definition \ref{gff}. Finally, the additive constant is fixed so that the circle average on $\p B(0,1)$ vanishes.
\end{defn}
\begin{remark}
Note that the restriction of an $\alpha$-quantum cone field to the unit disk has the same law as $\left. \left(h-\alpha \log \vert \cdot\vert\right) \right\vert_{B(0,1)}$ where $h$ is a whole plane GFF.
\end{remark}

\subsection{Geometric properties of space-filling SLE segments}\label{SLEstuff}

The Schramm-Loewner evolution $SLE_\kappa$ for $\kappa>0$ is a one parameter family of fractal curves, first defined in \cite{schramm0}. These curves exhibit different properties depending on $\kappa.$ More precisely, $SLE_\kappa$ is a simple curve for $\kappa \in (0,4],$ self-touching but not space filling or self crossing for $\kappa \in (4,8),$ and becomes space-filling for $\kappa \geq 8$~\cite{schramm-sle}.

Space-filling $SLE_{\kappa}$ for $\kappa > 4$ is a variant of $SLE_{\kappa}$ introduced in~\cite{ig4}. We will recall the construction of the version of space-filling $SLE_{\kappa}$ from $\infty$ to $\infty$ in $\mathbb C$ here. To define it, let
\[
\chi^{\mathrm{IG}}:= \frac{2}{\sqrt{\kappa}} - \frac{\sqrt{\kappa}}{2}.
\]
We define a whole-plane GFF modulo a multiple of $2\pi \chi_{\mathrm{IG}}$ to be an equivalence class of random distributions as follows. First we sample $\Psi$ from the law of a whole-plane GFF normalized so that $\Psi_1(0) = 0.$ Suppose that we fix $z \in \mathbb{C}$ and $\theta \in [0,2\pi].$ Then by \cite[Theorem 1.1]{ig4} we can define the flow line of $\Psi$ at $z$ with angle $\theta.$ To define the space filling $SLE_\kappa$ we will need the flow lines started at points in $z\in \mathbb{Q}^2$ with angles $-\frac{\pi}{2}$ or $\frac{\pi}{2},$ which will be denoted $\tilde{\eta}_z^L,\tilde{\eta}_z^R$ respectively.

For distinct $z,w \in \mathbb{Q}^2,$ the flow lines $\tilde{\eta}_z^L,\tilde{\eta}_w^L$ merge a.s. upon intersecting, and similarly for $\tilde{\eta}_z^R,\tilde{\eta}_w^R.$ The two flow lines $\tilde{\eta}_z^L,\tilde{\eta}_w^R$ a.s.\ do not cross, but these flow lines bounce off each other without crossing if and only if $\kappa \in (4,8)$ \cite[Theorem 1.7]{ig4}.

Given this, we can define a total ordering on $\mathbb{Q}^2$ by declaring that $z$ comes before $w$ if and only if $w$ lies in a connected component of $\mathbb{C}\setminus (\tilde{\eta}_z^L \cup \tilde{\eta}_z^R)$ lying to the right of $\tilde{\eta}_z^L.$ Using the whole plane analogue of Theorem 4.12 in \cite{ig4} we see that there is a well defined curve $\tilde{\eta}$ passing through $\mathbb{Q}^2$ in the order above, such that $\tilde{\eta}^{-1}(\mathbb{Q}^2)$ is a dense set, and is continuous when parametrized by Lebesgue measure, that is in such a way such that $\mathrm{area}(\tilde{\eta}([a,b]))=b-a$ whenever $a<b.$ The curve $\tilde{\eta}$ is defined to be the whole-plane space-filling $SLE_{\kappa}$ from $\infty$ to $\infty.$

We will need a couple of geometric facts about space-filling SLE segments.
\begin{lemma}\cite[Proposition 3.4]{ghm-kpz}
Let $\tilde{\eta}$ be a space filling $SLE_{\kappa}$ from $\infty$ to $\infty$ in $\mathbb{C}.$ For $r\in (0,1),$ $R>0,$ $\delta_0>0,$ let $\mathcal{E}_{\delta_0}=\mathcal{E}_{\delta_0}(R,r)$ be the event that the following is true. For all $0 <\delta \leq \delta_0,$ and each $a<b\in \R$ such that $\tilde{\eta}([a,b])\subset B_R(0)$ and $\mathrm{diam}(\tilde{\eta}([a,b])) \geq \delta^{1-r},$ the set $\tilde{\eta}([a,b])$ contains a ball of radius at least $\delta.$ Then
\[
\lim_{\delta_0\to 0}\mathbb{P}(\mathcal{E}_{\delta_0}) = 1.
\]
\end{lemma}
In particular, we have the following corollary:
\begin{cor}\label{CRTdiamsqrdarea}
Let $\tilde{\eta}$ be a space filling $SLE_{\kappa}$ from $\infty$ to $\infty$ in $\mathbb{C}.$ Let $r\in (0,1),$ and let $R>0.$ For $\delta_0>0,$ let $\tilde{\mathcal{E}}_{\delta_0}=\tilde{\mathcal{E}}_{\delta_0}(R,r)$ be the event that the following is true. For any $a<b \in \R$ such that $\mathrm{diam}(\tilde{\eta}([a,b])) \leq \delta_0,$ we have
\[
\mathrm{Area}(\tilde{\eta}([a,b])) \geq \mathrm{diam}(\tilde{\eta}([a,b]))^{\frac{2}{1-r}}.
\]
Then
\[
\lim_{\delta_0\to 0}\mathbb{P}(\tilde{\mathcal{E}}_{\delta_0}) = 1.
\]
\end{cor}

\subsection{Mated CRT maps with Gaussian free fields}
\label{matedcrtgff}

Recall the SLE/LQG embedding of the mated CRT map $\tilde{\mathcal G}_\e$ as defined at the end of Section~\ref{matedcrtprelims}. The main results of this paper concern $\tilde{\mathcal G}_\e$, but it is often convenient to work with a variant of the mated CRT map $\tilde{\mathcal G}_\e,$ where the underlying field is a whole-plane GFF $\Phi$ instead of the quantum cone $\tilde{\Phi}.$ This map, which we call $\mathcal{G}_\e$ is defined as follows. As in the construction of $\tilde{\mathcal G}_\e$, let $\Lambda^\e = \{y_j\}_{j \in \Z}$ denote a Poisson point process chosen on $\R$ with intensity $\e^{-1}$ defined so that $\{y_j\}\subset \R$ is an increasing sequence. We also let $\eta$ be a space-filling $\mathrm{SLE}_{\kappa}$ curve with $\kappa = 16/\gamma^2,$ reparametrized according to unit $\mu_\Phi$ measure. We let the vertex set of the graph $\mathcal{G}_\e$ be $\Lambda^\e,$ and we let two vertices $y_i, y_j$ be connected if the space filling SLE segments $\eta([y_{i-1},y_i])$ and $\eta([y_{j-1},y_j])$ share a nontrivial boundary arc, with two edges occurring between the two segments if they share nontrivial boundary arcs in both the left and right boundaries. Moreover, we can naturally embed this graph in $\mathbb{C}$ by assigning each vertex to a point in the corresponding mated CRT map cell.
We note that unlike in the case of $\tilde{\mathcal G}_\e$, the random planar map $\mathcal G_\e$ does not admit a simple description in terms of Brownian motion.

\subsection{Proof of Proposition \ref{lem-scaling-asymp}}
\label{scaling-asymp-proof}

\begin{proof}[Proof of Proposition \ref{lem-scaling-asymp}]
The proof is via a scaling argument where we study what happens when we replace $\tilde{\Phi}$ by $\tilde{\Phi}+\log(1/C)$. 
Fix $C > 0$. Write $\tilde{\eta}_C$ for the space-filling SLE curve parametrized by $\mu_{\tilde{\Phi}+\log(1/C)}$-mass instead of $\mu_{\tilde{\Phi}}$-mass. Let $\Lambda^\ep_C$ be a Poisson point process on $\BB R$ with intensity $1/\ep$, sampled independently from $\tilde{\Phi}$ and $\tilde{\eta}$. Let $\tilde{\mathcal{G}}_\e^C$ and its associated curvature $K_{\tilde{\mathcal{G}}_\e}^C$ be defined as in \eqref{curvdef}, but using the field $\tilde{\Phi} + \log (1/C)$ and $\Lambda_C^\e$ instead of $\tilde{\Phi}^\e$ and $\Lambda^\e.$ 

By the definition~\eqref{contcurv} of $K_{\tilde{\Phi}},$ we have $ K_{{\tilde{\Phi}} + \log(1/C)}  =  K_{\tilde{\Phi}}$.  By the local absolute continuity of the laws of $\tilde{\Phi}$ and $\tilde{\Phi} + \log(1/C)$ (see, e.g., \cite[Proposition 1.5.1]{bp-lqg-notes}), \eqref{eqn-scaling-asymp} implies the convergence in probability 
\eqb  \label{eqn-change-field-conv}
  \frac{1}{a_\ep} \sum_{v \in \mcl V {\tilde{\mathcal{G}}_\e^C} }   K_{\tilde{\mathcal{G}}_\e}^C(v) f( \tilde{\eta}_C(v) ) 
  \to \int_{\BB C} f(z)  K_{\tilde{\Phi}+\log(1/C)}(z) \,d z 
  = \int_{\BB C} f(z)  K_{\tilde{\Phi} }(z) \,d z .   
\eqe 

On the other hand, since LQG area measure satisfies $\mu_{\tilde{\Phi}}+ \log(1/C)  = C^{-1} \mu_{\tilde{\Phi}}$, we have that $\mu_{\tilde{\Phi}}(\tilde{\eta}_C[a,b]) = C (b-a)$ for each $b > a$. Thus $\tilde{\eta}_C(\cdot) = \tilde{\eta}(C \cdot )$. %$\mu_\Phi(\tilde{\eta}[0,C t]) = C t = \mu_\Phi(\tilde{\eta}_C[0,t])$
The point process $C  \Lambda^\ep_C $ is a Poisson point process on $\BB R$ with intensity $1/(C\ep)$. %\#[( C \Lambda_C^\ep) \cap [0,1]] = \#[ \Lambda_C^\ep \cap [0,1/C] ]$, which has expectation $1/(C \ep)$. 
From this, we deduce that $\tilde{\eta}_C(\Lambda^\ep_C) \eqD \tilde{\eta} ( \Lambda^{C \ep} ) $ and the joint law of $(\tilde{\Phi} , \tilde{\eta} , C \Lambda_C^\ep)$ is the same as the joint law of $(\tilde{\Phi},\tilde{\eta}, \Lambda^{C\ep})$. Recalling that $\mcl V\tilde{\mathcal{G}}_{C\ep} = \Lambda^{C\ep}$, we obtain the equality of joint laws 
\eqbn
\left( \tilde{\Phi}, \tilde{\eta}, \sum_{v \in \mcl V \tilde{\mathcal{G}}_\ep^C } K_{\tilde{\mathcal{G}}_\e}^C(v) f(\tilde{\eta}_C(v)) \right)  
\eqD \left( \tilde{\Phi}, \tilde{\eta} ,  \sum_{v\in \mcl V\tilde{\mathcal{G}}_{C\ep}}  K_{\tilde{\mathcal{G}}_{C\e}}(v)f(\tilde{\eta} (v)) \right) . 
\eqen
Therefore, \eqref{eqn-scaling-asymp} implies the convergence in probability
\eqb  \label{eqn-change-eps-conv}
  \frac{1}{a_{ C \ep } } \sum_{v \in \mcl V \tilde{\mathcal{G}}_\ep^C }  K_{\tilde{\mathcal{G}}_\e}^C(v) f( \tilde{\eta}_C(v)) \to \int_{\BB C} f(z)  K_{\tilde{\Phi} }(z) \,d z
\eqe 
By combining \eqref{eqn-change-field-conv} and \eqref{eqn-change-eps-conv} and using the fact that $\int_{\BB C} f(z) K_{\tilde{\Phi}}(z) \;dz$ is a.s. non-zero, we get that $\lim_{\ep\to 0} a_{C \ep } / a_\ep = 1$, as required.
\end{proof}

\section{Combinatorial relation for discrete curvature}\label{secsec}

Let $\eta$ be a whole-plane space filling SLE with parameter $\kappa = \frac{16}{\gamma^2}$ reparametrized according to unit $\mu_{\Phi}$ measure, where $\Phi$ is a whole-plane GFF, and let $\Psi$ denote the associated imaginary geometry field. In the rest of the paper, we will frequently switch between $\Phi$ and the quantum cone field, $\tilde{\Phi}.$

Recall that $\mathcal{G}_\e$ denotes the mated CRT map whose underlying field is the GFF $\Phi$ (Section~\ref{matedcrtgff}). We let $\mathcal{V}{\mathcal{G}}_\e$ and $\mathcal{E}{\mathcal{G}}_\e$ denote the vertex set and set of (unoriented) edges of $\mathcal{G}_\e$ respectively (and we define $\mathcal{V}\tilde{\mathcal{G}}_\e, \mathcal{E}\tilde{\mathcal{G}}_\e$ similarly). We also let $\mathcal{F}\mathcal{G}_\e$ denote the set of faces in $\mathcal{G}_\e.$

In this section we will only be considering the graph $\mathcal{G}_\e$ corresponding to the whole plane GFF $\Phi$. We will later transfer all results to the case where the underlying field is a $\gamma$-quantum cone (this is done in Section \ref{quantumcone}).

Suppose that $f \in C_c^2(\mathbb{C})$ is a compactly supported function. The main expression we will estimate is
\[
\sum_{v \in \mathcal{V}{\mathcal{G}}_\e} K_{\mathcal{G}_\e}(v) f(\eta(v)).
\]
The first step will be to rewrite this sum in a more convenient way. To that end, we recall the definition of the discrete gradient: for any oriented edge $e=(v_1,v_2)$ in $\mathcal{G}_\e$ we let $\mathcal{D} f(e)$ be the discrete gradient, $f(v_2)-f(v_1),$ where we abuse notation and identify $v$ and $\eta(v)$ in the case that $f : \mathbb{C} \to \mathbb{C}.$

 We will need to split up the set of edges on $\mathcal{G}_\e$: we set $\mathcal{E}_L\mathcal{G}_\e,\mathcal{E}_R\mathcal{G}_\e$ to be the sets of vertices corresponding to pairs $y_j,y_k$ for which \eqref{CRTeqdef} holds with $L$ and $R$ respectively, and $\vert j-k\vert>1.$ We also define $\mathcal{E}_M \mathcal{G}_\e$ to be the set of edges along the real line, that is edges between $y_j,y_k$ with $\vert j-k\vert=1$. Then 
 \[
 \mathcal{E}\mathcal{G}_\e = \mathcal{E}_M \mathcal{G}_\e \cup \mathcal{E}_L \mathcal{G}_\e \cup \mathcal{E}_R \mathcal{G}_\e.
 \]
We will need the following notation throughout.
\begin{defn}\label{Hvdef}
For any vertex $v$ let $H_{v}^\e$ denote the mated CRT map cell containing $v,$ that is
\[
H_v^\e = \eta([y_{j-1},y_j]),
\]
where $j$ is chosen so that $v \in \eta([y_{j-1},y_j]).$
\end{defn}
We have the following identity.
\begin{prop}\label{goodorientation}
We have two assignments $ \mathcal{E}_L\mathcal{G}_\e\cup \mathcal{E}_M\mathcal{G}_\e \ni e\mapsto \overrightarrow{e_L},$ $\mathcal{E}_R\mathcal{G}_\e \cup \mathcal{E}_M\mathcal{G}_\e \ni e\mapsto \overrightarrow{e_R}$ assigning to each edge an orientation, such that
\begin{equation}\label{identity}
\sum_{v \in \mathcal{V}\mathcal{G}_{\e}} f(\eta(v)) K_{\mathcal{G}_\e}(v) = \frac{\pi}{3}\sum_{e \in \mathcal{E}_L\mathcal{G}_\e\cup \mathcal{E}_M\mathcal{G}_\e} \mathcal{D} f(\overrightarrow{e_L}) + \frac{\pi}{3}\sum_{e \in \mathcal{E}_R\mathcal{G}_\e\cup \mathcal{E}_M\mathcal{G}_\e} \mathcal{D} f(\overrightarrow{e_R}).
\end{equation}
\end{prop}
\begin{proof}
Let $\mathcal{F}\mathcal{G}_\e$ denote the set of faces of $\mathcal{G}_\e,$ and let $\mathcal{F}_L\mathcal{G}_\e,\mathcal{F}_R\mathcal{G}_\e$ denote the sets of faces consisting of edges in $\mathcal{E}_L\mathcal{G}_\e\cup\mathcal{E}_M\mathcal{G}_\e$ and $\mathcal{E}_R\mathcal{G}_\e\cup\mathcal{E}_M\mathcal{G}_\e$ respectively. Then for any face $t \in \mathcal{F}\mathcal{G}_\e,$ there are three vertices on its boundary corresponding to three distinct space filling SLE segments. Ordering the cells as $C_1, C_2, C_3$ in chronological order according to $\eta,$ let $v(t)$ be the vertex corresponding to $C_2.$ We claim that $v(\cdot)\vert_{\mathcal{F}_L\mathcal{G}_\e}:\mathcal{F}_L\mathcal{G}_\e\to \mathcal{E}_L\mathcal{G}_\e$ and $v(\cdot)\vert_{\mathcal{F}_R\mathcal{G}_\e}:\mathcal{F}_R\mathcal{G}_\e\to \mathcal{E}_R\mathcal{G}_\e$ are both bijective maps. Indeed, if $H_v^\e = \eta([y_j,y_{j+1}]),$ then there exist four $\eta$ segments $H_{w_1}^\e,H_{w_2}^\e,H_{w_3}^\e,H_{w_4}^\e$ sharing a nontrivial boundary arc with each of $\p H_v^\e$ and $\eta((-\infty,y_j])\cap \eta([y_{j+1},\infty)).$ These four $\eta$ segments correspond to four vertices, denoted $v_r^\pm$ and $v_{\ell}^\pm$ ($v_r^{\pm}$ correspond to the past and future rightmost $\eta$ cells, while $v_{\ell}^{\pm}$ correspond to the past and future leftmost $\eta$ cells). Then the faces $\{v_r^-, v_r^+, v\}$ and $\{v_{\ell}^-, v_{\ell}^+, v\}$ are both such that $v(t)=v$ and the face $\{v_{\ell}^-, v_{\ell}^+, v\}$ is in $\mathcal{F}_L\mathcal{G}_\e,$ while the face $\{v_r^-, v_r^+, v\}$ is in $\mathcal{F}_R\mathcal{G}_\e$ (see Figure \ref{2to1}). This shows the surjectivity of $v(\cdot)\vert_{\mathcal{F}_L\mathcal{G}_\e}$ and $v(\cdot)\vert_{\mathcal{F}_R\mathcal{G}_\e}.$ For the injectivity, suppose without loss of generality that $v(t_1)=v(t_2)$ with $t_1,t_2 \in \mathcal{F}_L\mathcal{G}_\e.$ Suppose that the vertices in $t_1$ are $y_{j_1},y_{j_2},y_{j_3} \in \Lambda^\e$ where $j_1<j_2<j_3.$ Then $v(t_1)=y_{j_2}.$ Since the arc between $y_{j_1},y_{j_3}$ separates the upper half plane into two connected components, the vertices of $t_2$ must be between $y_{j_1}$ and $y_{t_3}.$ If the vertices of $t$ are given by $y_{\tilde{j}_1},y_{\tilde{j}_3}$ and $y_{\tilde{j}_2}=v(t_2),$ then this implies that
\[
j_1 < \tilde{j}_1 < j_2=\tilde{j}_2< \tilde{j}_3 < j_3.
\]
This is impossible, since the arcs between $y_{j_1}, y_{t_2}$ and $y_{j_3}, y_{t_2}$ separate $y_{\tilde{j}_1},y_{\tilde{j}_3}$ in different connected components of the upper half plane, hence it is impossible for an arc between $y_{\tilde{j}_1},y_{\tilde{j}_3}$ to exist (see Figure \ref{topo}).
\begin{figure}
\begin{center}
\includegraphics[width=0.4\textwidth]{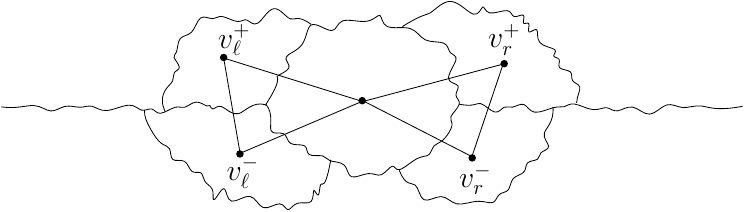}  
\caption{\label{2to1} The two indicated triangles are mapped to the same vertex.
}
\end{center}
\end{figure}
\begin{figure}
\begin{center}
\includegraphics[width=0.4\textwidth]{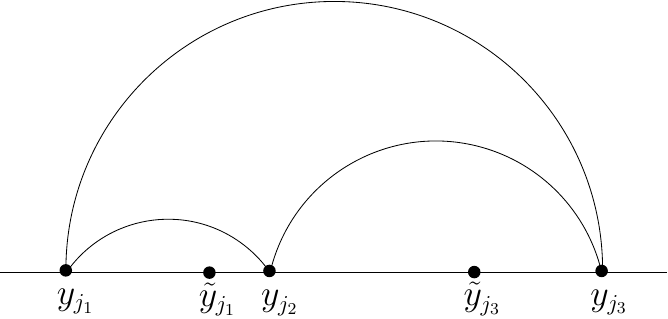}  
\caption{\label{topo} It is impossible for two triangles that both belong to $\mathcal{F}_L\mathcal{G}_\e$ or both belong to $\mathcal{F}_r\mathcal{G}_\e$ $t_1,t_2$ to be mapped to the same vertex under $v.$
}
\end{center}
\end{figure}
Then
\[
\sum_{t\in\mathcal{F}_L\mathcal{G}_\e} f(\eta(v(t))=\sum_{v\in \mathcal{V}\mathcal{G}_\e} f(\eta(v))=\sum_{t\in\mathcal{F}_R\mathcal{G}_\e} f(\eta(v(t))),
\]
while
\[
\sum_{v\in \mathcal{V}\mathcal{G}_\e} f(\eta(v))\mathrm{deg}(v) = \sum_{t \in \mathcal{F}\mathcal{G}_\e}\sum_{v \in \p t} f(\eta(v)),
\]
where $\p t$ denotes the set of vertices on the boundary of the face $t.$ Putting these together we obtain that
\[
\sum_{v \in \mathcal{V}\mathcal{G}_\e} f(\eta(v)) K_{\mathcal{G}_\e}(v) = \frac{\pi}{3}\left(\sum_{t \in \mathcal{F}_R\mathcal{G}_\e} \left(3f(\eta(v(t)))-\sum_{v \in \p t}f(\eta(v))\right) + \sum_{t \in \mathcal{F}_L\mathcal{G}_\e} \left(3f(\eta(v(t)))-\sum_{v \in \p t}f(\eta(v))\right)\right).
\]
Finally, noting that
\[
3f(\eta(v(t)))-\sum_{v \in \p t} f(\eta(v)) = \sum_{v \in \p t\setminus \{v(t)\}} f(\eta(v(t)))-f(\eta(v)).
\]
This implies that \eqref{identity} holds. To define the orientation explicitly, let $e$ be any edge with endpoints $y_{j_1},y_{j_2},$ with $j_1<j_2.$ If $\vert j_1-j_2\vert>1,$ then the arc between $y_{j_1},y_{j_2}$ divides the upper half plane into two connected regions. In this case there are two faces incident to $e$: one whose additional vertex is $y_{j_3}$ with $j_1<j_3<j_2,$ and the other whose additional vertex is $y_{j_4}$ with $j_4<j_1$ or $j_2<j_4.$ Let $t$ be the latter face. Then $v(t) \in \{y_{j_1},y_{j_2}\}$ ($v(t)=y_{j_1}$ if $j_4<j_1$ and $v(t)=y_{j_2}$ if $j_2<j_4$).  We orient the edge $e$ so that the ending vertex is $v(t).$ In the case that $j_2=j_1+1,$ the edge $e$ is incident to a face $t,$ with the vertices $y_{j_1},y_{j_2},y_{j_3}$ where $j_3<j_1$ or $j_3>j_2.$ In the first case, we have $v(t)=y_{j_1}, $ while in the second we have $v(t)=y_{j_2}.$ In either case, we orient the edge so that $v(t)$ is the ending vertex. This completes the proof.
\end{proof}

We will now define convenient mappings between vertices of $\mathcal{G}_{\e}$ and oriented edges.
\begin{defn}\label{vertmap}
Let $\tilde{\mathcal{E}}_\e$ denote the set of oriented edges in the graph $\mathcal{G}_\e.$ We will define four functions $\mathcal{V}\mathcal{G}_\e \ni v \mapsto \overrightarrow{e}_v^{PL},\overrightarrow{e}_v^{PR},\overrightarrow{e}_v^{FL},\overrightarrow{e}_v^{FR} \in \tilde{\mathcal{E}}_\e.$ For each vertex $v\in \mathcal{V}\mathcal{G}_\e,$ we define $\overrightarrow{e}_v^{PL}$ to be the leftmost outgoing edge contained in the past of $\eta,$ $\overrightarrow{e}_v^{PR}$ the rightmost outgoing edge contained in the past of $\eta,$ $\overrightarrow{e}_v^{FL}$ the leftmost outgoing edge contained in the future of $\eta,$ and $\overrightarrow{e}_v^{FR}$ the rightmost outgoing edge contained in the future of $\eta$ (see Figure \ref{Evpic}).
\end{defn}

%%%%%%%%%%%%%%%%%%%%%%%%
\begin{figure}
\begin{center}
\includegraphics[width=0.4\textwidth]{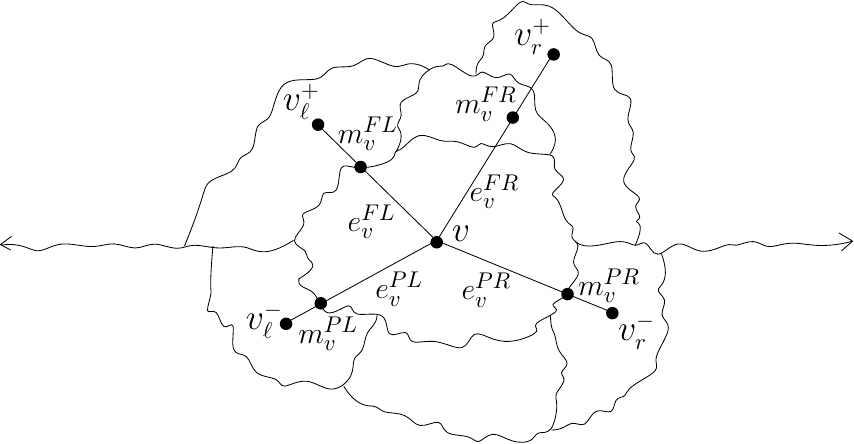}  
\caption{\label{Evpic} The edges $e_v^{PL}, e_v^{PR},e_v^{FL},e_v^{FR},$ and the points $m_v^{PL},m_v^{PR},m_v^{FL},m_v^{FR}.$
}
\end{center}
\end{figure} 
 %%%%%%%%%%%%%%%%%%%%%%%

If $\overrightarrow{e}$ is an oriented edge, let $e$ denote the unoriented edge corresponding to $\overrightarrow{e}.$ Then the maps $v \mapsto e_v^{PL}, v \mapsto e_v^{PR}, v \mapsto e_v^{FL}, v \mapsto e_v^{FR}$ are injective:
\begin{lemma}\label{injedge}
The maps $v \mapsto  e_v^{PL}, e_v^{PR}, e_v^{FL}, e_v^{FR},$ where $\overrightarrow{e}_v^{PL}, \overrightarrow{e}_v^{PR}, \overrightarrow{e}_v^{FL}, \overrightarrow{e}_v^{FR}$ are as defined in Definition \ref{vertmap}, are injective.
\end{lemma}
\begin{proof}
We will prove this for the map $v \mapsto e_v^{PL},$ the proofs for the other maps being identical. We will also denote $e_v^{PL}$ by $e_v$ for the remainder of the proof to lighten notation. Suppose that $v_1, v_2 \in \mathcal{V}\mathcal{G}_\e$ are such that $e_{v_1}=e_{v_2},$ but $v_1\neq v_2.$ Since $e_{v_1} = e_{v_2}$ if and only if their endpoints agree, this must mean that $e_{v_1}$ and $e_{v_2}$ are the same edge with different orientations. This implies that $v_1, v_2$ are the two endpoints of $e_{v_1} = e_{v_2},$ given that $v_1\neq v_2.$ Since $\overrightarrow{v_1v_2} = e_{v_1},$ this implies that $v_2$ is in the $v_1$'s past. However, given that $\overrightarrow{v_2v_1} = e_{v_2},$ this implies that $v_1$ is in $v_2$'s past, which yields a contradiction.
\end{proof}
Since these maps are injective, we have the following corollary:
\begin{cor}\label{rewritenabla}
For each vertex $v,$ there exist oriented edges $v \mapsto \overrightarrow{e_v}^{PL},$ $v \mapsto \overrightarrow{e_v}^{PR},$ $v \mapsto \overrightarrow{e_v}^{FL},$ and $v \mapsto \overrightarrow{e_v}^{FR},$ with starting point $v$ such that
\[
\sum_{v \in \mathcal{V}\mathcal{G}_\e}f(\eta(v))K_{\mathcal{G}_\e}(v) = \frac{\pi}{3}\sum_{v \in \mathcal{V}\mathcal{G}_\e} \left(\mathcal{D} f(e_v^{PL})+\mathcal{D} f(e_v^{PR})+\mathcal{D} f(e_v^{FL})+\mathcal{D} f(e_v^{FR})\right).
\]
\end{cor}
\begin{proof}
We claim that $\{e_v^{FR}\}_{v\in \mathcal{V}\mathcal{G}_\e}\dot{\cup} \{e_v^{PR}\}_{v\in \mathcal{V}\mathcal{G}_\e} = \mathcal{E}_R\mathcal{G}_\e \cup \mathcal{E}_M\mathcal{G}_\e,$ and similarly $\{e_v^{FL}\}_{v\in \mathcal{V}\mathcal{G}_\e}\dot{\cup} \{e_v^{PL}\}_{v\in \mathcal{V}\mathcal{G}_\e} = \mathcal{E}_L\mathcal{G}_\e \cup \mathcal{E}_M\mathcal{G}_\e.$ Indeed, to prove the first statement (the second being proved identically), note that if $e$ is an edge with endpoints $y_{j_1},y_{j_2}$ with $\vert j_1- j_2\vert$ (that is, $e \in  \mathcal{E}_R\mathcal{G}_\e$), there exist $j_3,j_4$ such that $\{y_{j_1},y_{j_2},y_{j_3}\}$ and $\{y_{j_1},y_{j_2},y_{j_4}\}$ form two faces in $\mathcal{F}\mathcal{G}_\e$ using edges in $\mathcal{E}_R\mathcal{G}_\e\cup \mathcal{E}_M\mathcal{G}_\e.$ Since the arc connecting $y_{j_1},y_{j_2}$ separates the upper half plane into two connected components, we can assume without loss of generality that $j_1<j_4 < j_2,$ and that $j_3 \in (-\infty,j_1] \cup [j_2,\infty).$ Then note that if $j_3 < j_1,$ then $e_{y_{j_1}}^{FR} = e$ while if $j_3>j_2$ then $e_{y_{j_2}}^{PR} = e$ (see Figures \ref{biject1} and \ref{biject2}). In the case that $j_2=j_1+1$ (that is, $e \in \mathcal{E}_M\mathcal{G}_\e$), there is a face incident to $e$ with vertices $y_{j_1},y_{j_2},y_{j_3}$ for some $j_3\in (-\infty,j_1-1]\cup[j_2+1,\infty),$ whose boundary consists of edges in $\mathcal{E}_R\mathcal{G}_\e \cup \mathcal{E}_M \mathcal{G}_\e.$ Then as in the previous case, $e_{y_{j_1}}^{FR}=e$ if $j_3<j_1,$ while $e_{y_{j_2}}^{PR}=e$ if $j_3>j_2.$ Moreover, if $e= e_{v_1}^{FR}=e_{v_2}^{PR}, $ then $v_1 = y_{j_1}$ and $v_2 = y_{j_2}.$ Since $e_{y_{j_1}}^{FR}=e,$ there exists $j_3<j_1$ and $j_4>j_2$ such that $\{y_{j_1},y_{j_3},y_{j_2}\}$ and $\{y_{j_1},y_{j_4},y_{j_2}\}$ form triangles. Noting that $j_3=j_4$ by orientation, we obtain a contradiction.
\begin{figure}
\centering
\begin{minipage}{0.4\textwidth}
  \centering
  \includegraphics[width=0.9\textwidth]{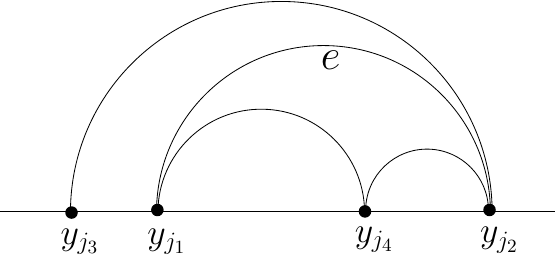}
  \caption{\label{biject1} The case $j_3<j_1<j_4<j_2.$ Here we have $e_{y_{j_1}}^{FR}=e.$}
\end{minipage}
\begin{minipage}{0.4\textwidth}
  \centering
  \includegraphics[width=0.9\linewidth]{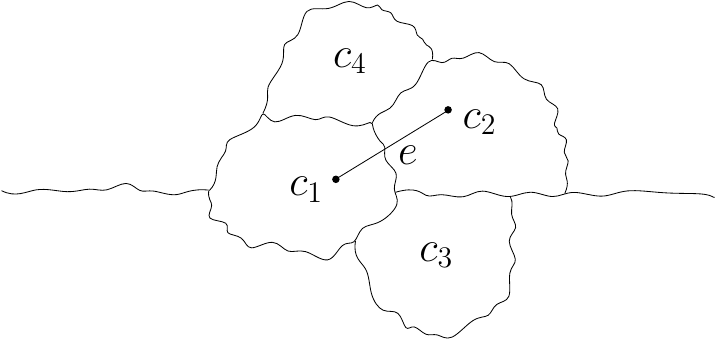}
  \caption{\label{biject2} The cells $c_1,c_2,c_3,c_4$ correspond to the vertices $y_{j_1},y_{j_2},y_{j_3},y_{j_4}$ respectively.}
\end{minipage}
\end{figure}The result is now a direct consequence of Proposition \ref{goodorientation} and Lemma \ref{injedge}.
\end{proof}

We rewrite this sum even further. For any point $z \in \mathbb{C},$ let $H_z^\e$ denote the mated CRT cell containing $z,$ and let $x_z^\e \in \mathcal{V}\mathcal{G}_\e$ be the corresponding vertex. Also, let $v_z^{PL},v_z^{PR},v_z^{FL},v_z^{FR} \in \mathbb{C}$ denote the endpoints of $e_v^{PL},e_v^{PR},e_v^{FL},e_v^{FR}$ that are not $v,$ where $v=x_z^\e.$ Then there are points $m_z^{PL},m_z^{PR},m_z^{FL},m_z^{FR} \in \mathbb{C}$ (not necessarily vertices in $\mathcal{G}_\e$) such that
\begin{equation}\label{mzdef}
\vert z-m_z^{PL}\vert \leq \mathrm{diam}(H_z^\e),\quad \vert m_z^{PL}-v_z\vert \leq \mathrm{diam}(H_{v_z}^\e)
\end{equation}
and the analogous statements for $m_z^{PR},m_z^{FL},m_z^{FR}$ (see Figure \ref{mz}).
\begin{figure}
\begin{center}
\includegraphics[width=0.4\textwidth]{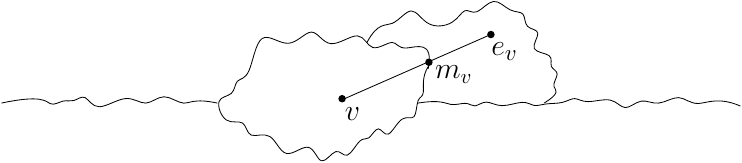}  
\caption{\label{mz} The definition of the points $m_v \in \{m_v^{PL},m_v^{PR},m_v^{FL},m_v^{FR}\}$
}
\end{center}
\end{figure}
Then we can decompose $\mathcal{D} f(e_z)$ as
\[
\mathcal{D} f(e_z) = f(v_z)-f(\eta(x_z^\e)) = (f(v_z)-f(m_z))+(f(m_z)-f(\eta(x_z^\e)))
\]
(here $e_z\in \{e_z^{PL},e_z^{PR},e_z^{FL},e_z^{FR}\}$ and $m_z \in \{m_z^{PL},m_z^{PR},m_z^{FL},m_z^{FR}\}$ is the corresponding point). Let
\[
G_f^{PL}(z) := f(m_z^{PL})-f(\eta(x_z^\e)) + \sum_{y:v_y^{PL}=z} (f(\eta(x_z^\e))-f(m_y^{PL}))
\]
and define $G_f^{PR}(z), G_f^{FL}(z),G_f^{FR}(z)$ similarly. Then we let
\begin{equation}\label{Gdef}
G_f(z):= G_f^{PL}(z)+ G_f^{PR}(z) + G_f^{FL}(z)+ G_f^{FR}(z),
\end{equation}
that is, $G_f$ is the sum of the discrete gradients over the red segments in Figure \ref{Gfnull}. We have the following lemma.
\begin{figure}
\begin{center}
\includegraphics[width=0.4\textwidth]{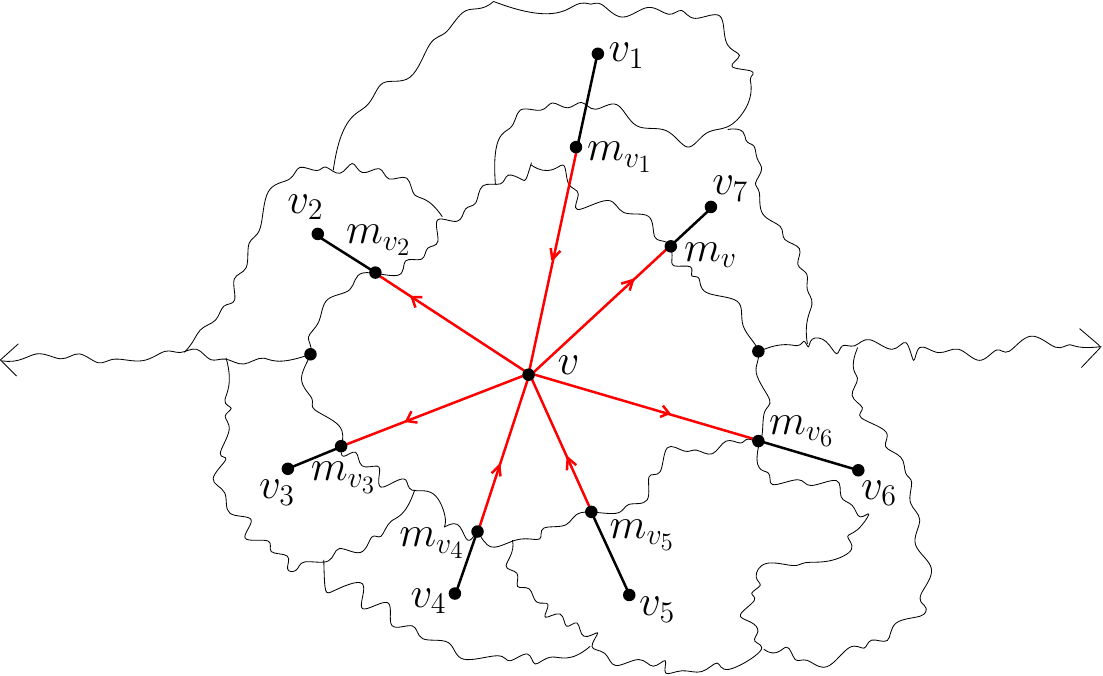}  
\caption{\label{Gfnull} $G_f$ is the sum of $f$'s discrete gradients along the red segments.
}
\end{center}
\end{figure}
\begin{lemma}\label{localsplit}
Let $G_f(z)$ be defined as in \eqref{Gdef}. Then for any compactly supported function $f \in C_c^2(\mathbb{C})$ we have
\begin{equation}\label{localspliteq}
\sum_{z \in \mathcal{V}\mathcal{G}_\e} f(\eta(z))K_{\mathcal{G}_\e}(v) = \frac{\pi}{3}\sum_{z \in \mathcal{V}\mathcal{G}_\e} G_f(z).
\end{equation}
\end{lemma}

The reason for rewriting the left hand side of \eqref{localspliteq} in terms of $G_f$ is that $G_f$ only depends on a neighborhood of $v$ and therefore satisfies useful independence properties (see Lemma \ref{Gfindep}).

\begin{proof}
Recall that $v_z^{PL},v_z^{PR},v_z^{FL},v_z^{FR} \in \mathbb{C}$ denote the endpoints of $e_v^{PL},e_v^{PR},e_v^{FL},e_v^{FR}$ that are not $v.$ Then we have
\begin{eqnarray}\label{splitaux1}
&&\sum_{e \in \mathcal{E}_L\mathcal{G}_\e \cup \mathcal{E}_M\mathcal{G}_\e} D f(\overrightarrow{e_L})\nonumber\\
&&= \sum_{z \in \mathcal{V}\mathcal{G}_\e} \sum_{e\in \{e_z^{PL},e_z^{FL}\}} \mathcal{D} f(\overrightarrow{e_L})\nonumber\\
&&= \sum_{v \in \mathcal{V}\mathcal{G}_\e} \sum_{(m_z,v_z) \in \{(m_z^{PL},v_z^{PL}),(m_z^{FL},v_z^{FL})\}} (f(v_z)-f(m_z))+(f(m_z)-f(\eta(z)))\nonumber\\
&&=\sum_{z \in \mathcal{V}\mathcal{G}_\e} \sum_{m_z \in \{m_z^{PL},m_z^{FL}\}}(f(m_z)-f(\eta(z)))\nonumber\\
&&+ \sum_{z \in \mathcal{V}\mathcal{G}_\e}\sum_{y:v_y^{PL}=\eta(z)} f(\eta(z)) - f(m_y^{PL})+ \sum_{z \in \mathcal{V}\mathcal{G}_\e}\sum_{y:v_y^{FL}=\eta(z)} f(\eta(z)) - f(m_y^{FL})\nonumber\\
&=& \sum_{z \in \mathcal{V}\mathcal{G}_\e} G_f^{PL}(z)+ G_f^{FL}(z),
\end{eqnarray}
and similarly
\begin{equation}\label{splitaux2}
\sum_{e \in \mathcal{E}_R\mathcal{G}_\e \cup \mathcal{E}_M\mathcal{G}_\e} D f(\overrightarrow{e_L}) = \sum_{z \in \mathcal{G}_\e} G_f^{PR}(z)+ G_f^{FR}(z).
\end{equation}
\begin{comment}
\begin{eqnarray}\label{splitaux1}
&&\sum_{z \in \mathcal{V}_L\mathcal{G}_\e} \mathcal{D} f(e_z^{PL})+\mathcal{D} f(e_z^{PR})+\mathcal{D} f(e_z^{FL})+\mathcal{D} f(e_z^{FR})\nonumber\\
&&=\sum_{z \in \mathcal{V}_L\mathcal{G}_\e} \sum_{m_z \in \{m_z^{PL},m_z^{FL}\}}(f(m_z)-f(\eta(z))) + \sum_{y \in \mathcal{V}\mathcal{G}_\e}\sum_{m_y \in \{m_y^{PL},m_y^{FL}\}} \left(f(v_y)-f(m_y)\right)\nonumber\\
&&= \sum_{z \in \mathcal{V}_L\mathcal{G}_\e}\sum_{m_z \in \{m_z^{PL},m_z^{FL}\}} (f(m_z)-f(\eta(z))) + \sum_{z \in \mathcal{V}\mathcal{G}_\e}\sum_{y: v_y = z}\sum_{m_y \in \{m_y^{PL},m_y^{FL}\}} \left(f(v_y)-f(m_y)\right)\nonumber\\
&&= \sum_{z \in \mathcal{V}_L\mathcal{G}_\e} \left(\sum_{m_z \in \{m_z^{PL},m_z^{FL}\}}(f(m_z)-f(\eta(z))) + \sum_{y: v_y = z}\sum_{m_y \in \{m_y^{PL},m_y^{FL}\}} (f(z) - f(m_y))\right),\nonumber\\
\end{eqnarray}
and similarly,
\begin{eqnarray}\label{splitaux2}
&&\sum_{z \in \mathcal{V}_R\mathcal{G}_\e} \mathcal{D} f(e_z^{PL})+\mathcal{D} f(e_z^{PR})+\mathcal{D} f(e_z^{FL})+\mathcal{D} f(e_z^{FR})\nonumber\\
&&= \sum_{z \in \mathcal{V}_R\mathcal{G}_\e} \left(\sum_{m_z \in \{m_z^{PR},m_z^{FR}\}}(f(m_z)-f(\eta(z))) + \sum_{y: v_y = z}\sum_{m_y \in \{m_y^{PR},m_y^{FR}\}} (f(\eta(z)) - f(m_y))\right),\nonumber\\
\end{eqnarray}
\end{comment}
Now the proof follows immediately from Corollary \ref{rewritenabla} together with \eqref{splitaux1}, \eqref{splitaux2}.
\end{proof}
Recall the definition of $H_z^\e$ in Definition \ref{Hvdef}. Let $\widehat{H}_z^\e$ denote the mated CRT map cell corresponding to $z$ together with all the neighboring cells (see Figure \ref{hatH}), that is
\begin{equation} \label{eqn:hatH-def}
\widehat{H}_z^\e = H_z^\e \cup \bigcup_{y:z\sim y}H_y^\e.
\end{equation}
%%%%%%%%%%%%%%%%%%%%%%%
\begin{figure}
\begin{center}
\includegraphics[width=0.4\textwidth]{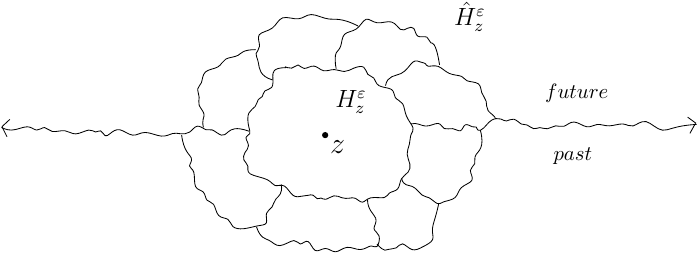}  
\caption{\label{hatH} The definition of $\widehat{H}_x^\e$
}
\end{center}
\end{figure}
\begin{lemma}\label{locality}
Let $U$ be a deterministic open set. Let $K_z^\e$ be defined so that $K_z^\e=H_z^\e$ if $H_z^\e \subseteq U,$ and empty otherwise. Then the collection $\{K_z^\e\}$ is determined by $(\eta(\mathcal{V}\mathcal{G}_\e) \cap U, \Phi\vert_U,\Psi\vert_U,\eta(\Lambda^\e)\cap U).$
\end{lemma}
\begin{proof}
We know by Lemma 2.4 in \cite{gms-harmonic} that the space filling SLE segments are determined by the restricted fields $\Phi\vert_U, \Psi\vert_U$ together with the $\nu_{\Phi}$- unit length parametrization. This parametrization is in turn locally determined by $\Lambda^\e, \Phi\vert_U$ and thus the space filling SLE segments are locally determined by $\eta(\mathcal{V}\mathcal{G}_\e) \cap U, \Phi\vert_U,\Psi\vert_U.$
\end{proof}
We note that the function $G_f$ defined as in \eqref{Gdef} only depends on the fields $\Psi, \Phi$ locally:
\begin{cor}\label{Gfindep}
Suppose $U$ is an open set and $f$ is a $C_c^2$ test function. Then for any $z,$ the value of $G_f(z)1_{\widehat{H}_z^\e \subseteq U}$ is uniquely determined by $(\Phi,\Psi)\vert_{U}$ and $\eta(\Lambda^\e)\cap U.$
\end{cor}
\begin{proof}
This is a direct consequence of Lemma \ref{locality}.
\end{proof}
In the following lemma and onwards, we will let $A\lesssim B$ mean that $A \leq CB$ for some absolute constant $C.$ We will prove the following lemma about $G_f$ which will be used later, see \eqref{eq8}.
\begin{lemma}\label{Gfnotfound}
Let $f\in C_c^2(\mathbb{C})$ be a $C^2$ compactly supported function, and define $G_f$ as in \eqref{Gdef}. Then for any vertex $z \in \mathcal{V}\mathcal{G}_\e,$
\[
\vert G_f(z)\vert\lesssim (\mathrm{deg}(z)+1) \norm{\nabla f}_\infty \mathrm{diam}(H_z^\e).
\]
\end{lemma}
\begin{proof}
By \eqref{Gdef} it suffices to show that
\[
\vert G_f^{PL}(z)\vert \lesssim (\mathrm{deg}(z)+1) \norm{\nabla f}_\infty \mathrm{diam}(H_z^\e),
\]
\[
\vert G_f^{PR}(z)\vert \lesssim (\mathrm{deg}(z)+1) \norm{\nabla f}_\infty \mathrm{diam}(H_z^\e),
\]
\[
\vert G_f^{FL}(z)\vert \lesssim (\mathrm{deg}(z)+1) \norm{\nabla f}_\infty \mathrm{diam}(H_z^\e),
\]
\[
\vert G_f^{FR}(z)\vert \lesssim (\mathrm{deg}(z)+1) \norm{\nabla f}_\infty \mathrm{diam}(H_z^\e).
\]
We will only prove the first statement, the proof of the other three being identical. We have that
\begin{eqnarray*}
\vert G_f^{PL}(z)\vert &=& \left\vert f(m_z^{PL})-f(\eta(z)) + \sum_{y\in \mathcal{V}\mathcal{G}_\e: v_y^{PL} = z} (f(\eta(z))-f(m_y^{PL})) \right\vert\\
&\leq &  \vert f(m_z^{PL})-f(\eta(z)) \vert + \sum_{y\in \mathcal{V}\mathcal{G}_\e: v_y^{PL} = z} \vert f(\eta(z))-f(m_y^{PL})\vert\\
&\leq & \norm{\nabla f}_\infty \vert m_z^{PL}-\eta(z) \vert + \sum_{y\in \mathcal{V}\mathcal{G}_\e: v_y^{PL} = z} \norm{\nabla f}_\infty\vert \eta(z)-m_y^{PL}\vert
\end{eqnarray*}
By the definition of $m_y^{PL}$ for a general vertex $y$ \eqref{mzdef}, we have that
\[
\vert m_z^{PL}-\eta(z)\vert,\vert \eta(z)-m_y^{PL}\vert \leq \mathrm{diam}(H_z^\e)
\]
if $v_y^{PL}=\eta(z).$ Hence
\[
\vert G_f^{PL}(z)\vert \leq \norm{\nabla f}_\infty \vert m_z^{PL}-\eta(z) \vert + \sum_{y\in \mathcal{V}\mathcal{G}_\e: v_y^{PL} = z} \norm{\nabla f}_\infty\vert \eta(z)-m_y^{PL}\vert \leq (\mathrm{deg}(z)+1) \norm{\nabla f}_\infty \mathrm{diam}(H_z^\e).
\]
This completes the proof.
\end{proof}

\section{Control of main sum in the whole-plane case}

First we will control the sum in \eqref{localspliteq}.

\begin{prop}\label{mainvar}
With probability going to $1$ as $\e\to 0,$ we have
\[
\sum_{z \in \mathcal{V}\mathcal{G}_\e} f(\eta(z)) K_{\mathcal{G}_\e}(z)  = \e^{o(1)}.
\]
\end{prop}

Once we have Proposition \ref{mainvar}, we obtain the following as a direct consequence.
\begin{prop}\label{mainbutno}
If we replace $\tilde{\eta}$ with $\eta$ where $\eta$ is a space filling $SLE_{\kappa}$ parametrized according LQG area with respect to a whole plane GFF $\Phi$ instead of a $\gamma$-quantum cone field, then the conclusion of Theorem \ref{main} holds.
\end{prop}
The proof of Theorem \ref{main} will be done in Section \ref{quantumcone}.
We will now focus on proving Proposition \ref{mainvar}.

The first step will be to split up the sum in \eqref{localsplit} into the parts corresponding to $\mathcal{E}_L\mathcal{G}_\e$ and $\mathcal{E}_R\mathcal{G}_\e.$ We have by Lemma \ref{localsplit}
\[
\sum_{z \in \mathcal{V}\mathcal{G}_\e} f(\eta(z)) K_{\mathcal{G}_\e}(z) = \frac{\pi}{3}\sum_{v \in \mathcal{V}\mathcal{G}_\e} G_f(v) = \frac{\pi}{3}\left(\sum_{v \in \mathcal{V}\mathcal{G}_\e} G_f^{PL}(v)+ G_f^{FL}(v)\right) + \frac{\pi}{3}\left(\sum_{v \in \mathcal{V}\mathcal{G}_\e} G_f^{PR}(v)+ G_f^{FR}(v)\right).
\]
Hence it suffices to show that
\begin{equation}\label{halfofneeded}
\sum_{v \in \mathcal{V}\mathcal{G}_\e} G_f^{PL}(v)+ G_f^{FL}(v) = \e^{o(1)},
\end{equation}
\begin{equation}\label{otherhalfofneeded}
\sum_{v \in \mathcal{V}\mathcal{G}_\e} G_f^{PR}(v)+ G_f^{FR}(v) = \e^{o(1)}.
\end{equation}
We will show the first identity, the second being proved identically. To this end, we will rewrite this last sum as an integral over $\mathbb{C}.$ Let $\Gamma_f(z):=G_f^{PL}(z)+G_f^{FL}(z).$ Recalling Definition \ref{Hvdef} and defining let $x_z^\e\in \mathcal{V}\mathcal{G}_\e$ be the vertex corresponding to $H_z^\e,$ note that
\begin{equation}\label{halfofneeded2}
\sum_{v\in \mathcal{V}\mathcal{G}_\e} \Gamma_f(v) = \int_{\mathbb{C}} \frac{\Gamma_f(x_z)}{\mathrm{Area}(H_z^\e)}dz.
\end{equation}
Let $X_z^\e$ denote the quotient
\begin{equation}\label{Xzeps}
X_z^\e = \frac{\Gamma_f(x_z^\e)}{\mathrm{area}(H_z^\e)}1_{z \in \mathrm{Supp}f}.
\end{equation}
We will follow the ideas of Section 4 in \cite{gms-harmonic}. The next step will be to define a sequence of sets $Z_j$ to conveniently split up this integral. This is done in the following subsection.

\subsection{Defining the truncation}

We will need the following technical lemmas on the size for degrees on $\tilde{\mathcal{G}}_\e$ and $\mathcal{G}_\e.$
\begin{lemma}\label{degunifboundnotnot}
Let $\tilde{\mathcal{G}}_\e$ be defined as in Definition \ref{CRTdef} (with the $\gamma$-quantum cone field $\tilde{\Phi}$), and let $K$ be a deterministic compact set. Let $v_0$ denote the vertex associated with the mated CRT map cell containing the origin. There are constants $c_0,c_1>0$ such that for any $n \in \mathbb{N}$ we have
\[
\mathbb{P}(\mathrm{deg}(v_0) >n) \leq c_0 e^{-c_1 n^{\frac{1}{2}}}.
\]
\end{lemma}
\begin{proof}
The proof is similar to the proof of Lemma 2.5 in \cite{gms-harmonic}. Recall that the Poisson point process is denoted $\Lambda^\e = \{y_j\}.$ First, let $j_0\in \mathbb{Z}$ denote the index such that $0 \in [y_{j_0},y_{j_0+1}),$ and let 
\[
S :=\left\{j \geq j_0: \left(\inf_{t \in [y_{j_0},y_{j_0+1}]}L_t\right) \vee \left(\inf_{t \in [y_j,y_{j+1}]}L_t\right)< \inf_{t \in [y_{j_0},y_j]}L_t\right\}
\]
To prove the lemma, it suffices to show that
\[
\mathbb{P}(\# S > n ) \leq c_0 e^{-c_1 n^{\frac{1}{2}}}
\]
for appropriate constants $c_0,c_1.$ Note that if $j \in S,$ then in particular we have that $L_t$ has not yet reached $\inf_{t \in [y_{j_0},y_{j_0+1}]}L_t$ at time $y_j,$ but
\[
\inf_{t \in [y_j,y_{j+1}]}L_t < \inf_{t \in [y_{j_0},y_j]}L_t \leq 0.
\]
Therefore if we define $\tilde{\mathcal{N}}_1$ as
\[
\tilde{\mathcal{N}}_1 = \#\{j \geq j_0: \inf_{t \in [y_{j_0},y_j]}L_t > \inf_{t \in [y_{j_0},y_{j_0+1}]}L_t, \inf_{t \in [y_j,y_{j+1}]}L_t < 0\},
\]
then $\# S \lesssim \tilde{\mathcal{N}}_1.$ It suffices to bound $\mathbb{P}(\tilde{\mathcal{N}}_1>n).$
Now we will define a sequence of stopping times. Let $A$ denote the infimum
\[
A=\inf_{t \in [y_{j_0},y_{j_0+1}]} L_t.
\]
We define the times $\sigma_n,\tau_n$ as follows: we let
\[
\sigma_n = \tau_n+A^2,
\]
and
\[
\tau_n = \min\{t: t \geq \sigma_{n-1} : L_t=0\}.
\]
Define
\[
\tilde{\mathcal{N}}_2 := \min\{n : \exists t \in [\tau_n,\sigma_n] \mbox{ s.t. } L_t \leq A \}.
\]
Then by definition of $\tilde{\mathcal{N}}_1$ we have that
\begin{equation}\label{ncomparison}
\tilde{\mathcal{N}}_1 \leq \sum_{n=1}^{\tilde{\mathcal{N}}_2} \#(\Lambda^\e \cap [\tau_n,\sigma_n]).
\end{equation}
First, we will bound $\mathbb{P}\left(\tilde{\mathcal{N}}_2 >n\right).$ After this, we will use a union bound to bound the probability that simultaneously, every interval $[\tau_n,\sigma_n]$ has a bounded number of points in $\Lambda^\e.$ 

By the Markov property of Brownian motion, we have that the expression
\[
\mathbb{P}\left(\exists t \in [\tau_n,\sigma_n] \mbox{ s.t. } L_t \leq M \vert \mathcal{F}_{\tau_n}\right)
\]
only depends on the constant $M,$ where $\mathcal{F}_{\tau_n}$ is the $\sigma$-algebra generated by $L\vert_{[y_{j_0},\tau_n]}.$

We will need the following estimate (which is a consequence of Bennett's inequality): if $X$ has a Poisson distribution of intensity $\lambda,$ then
\begin{equation}\label{poissontail}
\mathbb{P}(X>x+\lambda) \leq e^{-\frac{x^2}{2(\lambda+x)}}.
\end{equation}
In particular, using the fact that the number of points in $\Lambda^\e$ contained in a fixed interval follows a Poisson distribution, we have that if $\tilde{\mathcal{N}}_2 \leq n$ then we have that for some constant $c_1>0,$
\[
\mathbb{P}\left(\vert \Lambda^\e \cap [\tau_n,\sigma_n]\vert \leq A^2n \vert A\right) \geq 1-n e^{-c_1A^2 n^{\frac{1}{2}}}.
\]
Note that $A$ has stretch exponential decay: indeed,
\[
\mathbb{P}(A>M) \leq \int_0^\infty \mathbb{P}\left(\inf_{t\in[0,\ell]} L_t > M\right) \e e^{-\e \ell} d\ell \lesssim \int_0^\infty e^{-C\ell M} \e e^{-\e \ell} d\ell \lesssim e^{-CM}.
\]
Hence for some $c_2>0,$ we have
\[
\mathbb{P}\left(A^2>n\right) = \mathbb{P}\left(A>n^{\frac{1}{2}}\right) \leq e^{-c_2n^{\frac{1}{2}}}.
\]
Thus we obtain that with probability at most $e^{-c_1A^2n^{\frac{1}{2}}}n + e^{-c_2n^{\frac{1}{2}}},$ $\tilde{\mathcal{N}}_2 >n.$

Now, note that after conditioning on the intervals $\{[\tau_n,\sigma_n]\}_n,$ the quantities $\#(\Lambda^\e \cap [\tau_n,\sigma_n]),\#(\Lambda^\e \cap [\tau_m,\sigma_m])$ are independent. Since the interval $[\tau_n,\sigma_n]$ is of length $A^2,$ the law of $\#(\Lambda^\e \cap [\tau_n,\sigma_n])$ is that of an exponential distribution with mean $\frac{A^2}{\e}.$ Therefore using a union bound and \eqref{poissontail} we obtain that with probability at least $1- e^{-c_1A^2n^{\frac{1}{2}}}n - e^{-c_2n^{\frac{1}{2}}}- n e^{-c_3n^{\frac{1}{2}}},$ we have $\mathrm{deg}(v_0) \leq n^3.$ This completes the proof.
\end{proof}

\begin{lemma}\label{degunifboundnot}
Let $\theta>0.$ Let $\tilde{\mathcal{\mathcal{G}}}_\e$ be defined as in Definition \ref{CRTdef}. Then we have
\begin{equation}\label{eq13}
\sup_{z \in B_1(0)} \mathrm{deg} (x_z^\e) \leq (\log \e^{-1})^{2+\theta} 
\end{equation}
with probability tending to $1$ as $\e \to 0.$
\end{lemma}
\begin{proof}
By Proposition 6.2 in \cite{hs-euclidean} we have that for $M>0,$ $B_1(0) \subset \tilde{\eta}([-\e^{-M},\e^{-M}])$ except on an event of probability at most $C\e^{-M p^*(\gamma)}$ for some constant $C$ and power $p^*(\gamma).$ Now note that the law of the number of points in the Poisson point process $\Lambda^\e$ contained in $[-\e^{-M},\e^{-M}],$ is a Poisson distribution with mean $2\e^{-M-1},$ hence with probability going to $1$ as $\e \to 0,$ this number is at most $\e^{-M-2}.$ Thus by Lemma \ref{degunifboundnotnot} and a union bound, we conclude \eqref{eq13} holds with probability tending to $1$ as $\e \to 0.$
\end{proof}
Now we need to transfer this lemma to the case we have a whole plane GFF instead of the quantum cone field.
\begin{lemma}\label{degunifbound}
Let $K$ be a fixed, deterministic compact set, and let $\theta>0.$ Then after replacing the quantum cone field $\tilde{\Phi}$ by a whole plane GFF $\Phi,$ we have
\[
\sup_{z \in B_1(0)} \mathrm{deg} (x_z^\e) \leq (\log \e^{-1})^{2+\theta} 
\]
with probability tending to $1$ as $\e \to 0.$
\end{lemma}
\begin{proof}
Note that if we have a $\bar{\theta}>0$ and $f$ such that $B_{\bar{\theta}}(0) \cap K = \varnothing$ and $K \subseteq B_1(0),$ then using the absolute continuity of $\Phi\vert_{\mathbb{C} \setminus B_{\bar{\theta}}(0)}$ with respect to $\tilde{\Phi}\vert_{\mathbb{C} \setminus B_{\bar{\theta}}(0)}$ we would be done. Hence the conclusion holds as long as $K$ avoids a neighborhood of the origin.

For the general case, we proceed as follows. First, let $r>1.$ We claim that for any two points $z,w$ the laws of $\Phi(\cdot + z)\vert_{B_r(0)}$ and $\Phi(\cdot+w)\vert_{B_r(0)}$ are mutually absolutely continuous. Indeed, if we take $r'>r,$ then the conditional law of $\Phi(\cdot+z)\vert_{B_r(0)}$ given $\Phi(\cdot+z)\vert_{B_{r'}(0)^c}$ is mutually absolutely continuous with respect to the law of a zero boundary GFF restricted to $B_r(z)$ by the Cameron-Martin property for the GFF. Noting that the zero boundary GFF is translation invariant proves the claim. To lift the condition that $K$ avoids the origin, we cover $K$ in finitely many balls of radius $r,$ and use the claim together with a union bound, applying Lemma \ref{degunifboundnot} to each ball.
\end{proof}

To define the events $Z_j$ we will need the following lemma.
\begin{lemma}\label{lbCRTcell}
Let $K\subset \mathbb{C}$ be a compact set. There exists a $\beta>0$ such that the following holds. With probability going to $1$ as $\e \to 0,$ for any mated CRT map cell $C$ such that $C \cap K \neq \varnothing,$ we have
\[
\mathrm{area}(C) \geq \e^\beta.
\]
\end{lemma}

\begin{proof}%[Proof of Lemma \ref{lbCRTcell}]
The proof is similar to that of Lemma 2.7 in \cite{gms-harmonic}. Let $\tilde{K}$ be a large compact set whose interior contains $K.$ We note that by Proposition 6.2 in \cite{hs-euclidean}, we have with probability going to $1$ as $\e\to 0,$ that $\tilde{K} \subseteq \eta([-\e^{-1},\e^{-1}]).$ Note that the space-filling SLE segments are determined by the Poisson point process $\Lambda^\e,$ therefore using a union bound together with Lemma \ref{lbCRTcell}, we obtain that with probability going to $1$ as $\e\to 0 ,$ the $\mu_{\Phi}$ mass of each mated CRT map cell contained in $\tilde{K}$ is at least $\e^2,$ hence we can assume for the rest of the proof that these measures are simultaneously at least $\e^2.$ Let $\bar{\beta} \geq \frac{2}{(2-\gamma)^2}.$ By \cite{dg-lqg-dim}, Lemma 3.8, we have that with probability tending to $1$ as $\e \to 0,$
\begin{equation}\label{eq14}
\sup_{z \in \tilde{K}}\mu_{\Phi}(B_{\e^{\bar{\beta}}}(z)) \leq \e.
\end{equation}
Note that each segment of $\eta$ contained in $\tilde{K}$ intersecting $K$ with diameter at most $\e^\beta$ is contained in a Euclidean ball of radius at most $2\e^\beta$ contained in $B_1(0).$ Each space filling SLE segment has $\mu_{\Phi}$ mass $\e^2,$ therefore by \eqref{eq14} if we choose $\beta$ such that $2 \e^\beta \leq \e^{\bar{\beta}},$ we have that $\mathrm{diam}(C) \geq \e^\beta$ for any space filling SLE segment $C.$ Combining this with Lemma \ref{degunifbound} we obtain the result.
\end{proof}

We also have that space-filling SLE segments cannot have very large $\mu_\Phi$ measure.
\begin{lemma}\label{ubCRTcell}
Let $K\subset \mathbb{C}$ be a compact set. There exists a $\tilde{\beta}>0$ such that the following holds. With probability going to $1$ as $\e \to 0,$ for any mated CRT map cell $C$ such that $C \cap K \neq \varnothing,$ we have
\[
\mathrm{area}(C) \leq \e^{\tilde{\beta}}.
\]
\end{lemma}
\begin{proof}
We proceed as in Lemma 3.8 in \cite{dg-lqg-dim}. We have the standard moment bound
\[
\mathbb{E}\left(\mu_{\Phi}(B_\delta(z))^p\right) \leq \delta^{w(p)+o_\delta(1)}
\]
where $p \in \left[0,\frac{4}{\gamma^2}\right),$ $\delta\in (0,1),$ and $w(p)$ is defined by
\[
w(p):= \left(2+\frac{\gamma^2}{2}\right) p -\frac{\gamma^2}{2}p^2
\]
(See Theorem 2.14 in \cite{rhodes-vargas-review} and Lemma 5.2 in \cite{ghm-kpz}). Using Markov's inequality we obtain
\[
\mathbb{P}\left(\mu_{\Phi}(B_\delta(z))>\delta^{\tilde{\beta}^{-1}}\right) \leq \delta^{w(p)-\tilde{\beta}^{-1}p + o_\delta(1)}.
\]
Now taking $p = \frac{4+\gamma^2 - 2\tilde{\beta}^{-1}}{2\gamma^2}$ we obtain
\[
\mathbb{P}\left(\mu_{\Phi}(B_\delta(z))>\delta^{\tilde{\beta}^{-1}}\right)\leq \delta^{\frac{(4\gamma^2 -2\tilde{\beta}^{-1})^2}{8\gamma^2}+o_\delta(1)}
\]
after applying a union bound we conclude.
\end{proof}

Now we are ready to define the sets $Z_j.$ The idea is to split into sets of $z$ for which the associated mated CRT map cell cluster $\widehat{H}_z^\e$ (as defined in~\eqref{eqn:hatH-def}) has its diameter at a fixed scale. More precisely, we let $\zeta>0$ be a small positive number, we let $N$ be a large integer, and we fix $\alpha_1\leq \cdots  \leq \alpha_N$ such that
\begin{equation}\label{alphachoice}
\sup_{1\leq j \leq N-1} \vert \alpha_j - \alpha_{j+1}\vert < \zeta.
\end{equation}
Note that by choosing $N$ to be large enough, we can force
\[
\mathbb{P}\left( \sup_{z:\mathrm{Supp}f \cap H_z^\e \neq \varnothing} \mathrm{diam}\widehat{H}_z^\e \geq \e^{\alpha_N}\right), \mathbb{P}\left( \inf_{z:\mathrm{Supp}f \cap H_z^\e \neq \varnothing} \mathrm{diam}\widehat{H}_z^\e \leq \e^{\alpha_1}\right)
\]
to be arbitrarily small. Indeed, by Lemmas \ref{lbCRTcell} and \ref{ubCRTcell} together with Corollary \ref{CRTdiamsqrdarea} we see that the probability that $\widehat{H}_z^\e$ has a diameter outside of $\left[\e^{\alpha_N},\e^{\alpha_1}\right]$ tends to $0$ as $\alpha_1\to -\infty$ and $\alpha_N \to \infty.$ We let $Z_j$ be defined by
\begin{equation}\label{Ejdef}
Z_j :=\left\{z \in \mathrm{supp} \; f: \mathrm{diam}(\widehat{H}_z^\e) \in \left[\e^{\alpha_{j+1}},\e^{\alpha_j}\right], \; \mathrm{area}(H_z^\e) \geq \e^\beta, \deg x_z^\e \leq C(\log\e^{-1})^{3}, \e^{\zeta} \leq\frac{\mathrm{diam}\left(H_z^\e\right)^2}{\mathrm{area}\left(H_z^\e\right)} \leq \e^{-\zeta}\right\}
\end{equation}
for a large enough constant $C.$ The reason for the extra conditions on $\mathrm{area}\left(H_z^\e\right)$, $\deg x_z^\e$, and $\mathrm{diam}\left(\hat{H}_z^\e\right)$ is only technical, however we need these conditions for lower bounds in the future for certain expectations of expressions.
We will need the following lemma throughout this section.
\begin{lemma} \label{lem:Z-cover}
Let $\beta,\zeta>0$ be small fixed positive constants. Let $K \subset \mathbb{C}$ be a compact set. Then with probability tending to $1$ as $\e\to0,$ we have
\[
K \subseteq \bigcup_{j=1}^N Z_j.
\]
\end{lemma}
\begin{proof}
This lemma is a direct consequence of Lemmas \ref{degunifbound}, \ref{degunifbound} and Corollary \ref{CRTdiamsqrdarea}.
\end{proof}

\subsection{Restricting to a single set}

The integral we are interested in is
\[
\int_{\mathbb{C}} X_z^\e dz
\]
where $X_z^\e$ is defined in \eqref{Xzeps}. By \eqref{halfofneeded2} and \eqref{Xzeps}, proving \eqref{halfofneeded} is equivalent to
\[
\mathbb{E}\left(\left(\int_{\mathbb{C}} X_z^\e dz\right)^2\right) \leq \e^{-2\zeta + o(1)}.
\]
We claim we can restrict to a single set $Z_j.$
\begin{prop}\label{singleevent}
Let $1 \leq j \leq N-1.$ Then we have
\[
\mathbb{E}\left(\left(\int_{\mathbb{C}} 1_{Z_j}(z) X_z^\e dz\right)^2\right) \leq \e^{-15\zeta + o(1)}.
\]
\end{prop}
Assuming this, we will prove Proposition \ref{mainvar}.
\begin{proof}[Proof of Proposition \ref{mainvar}.]
By Lemma~\ref{lem:Z-cover}, there exists an event $F_\e$ such that $\mathbb{P}(F_\e) \to 1$ as $\e\to0,$ and for which we have
\begin{eqnarray*}
\mathbb{E}\left(1_{F_\e}\left(\int_{\mathbb{C}} X_z^\e dz \right)^2\right) & \lesssim & \mathbb{E}\left(\left(\sum_{j=1}^{N-1} \int_{\mathbb{C}} 1_{Z_j}(z) X_z^\e dz \right)^2\right)\\
&\leq & (N-1) \mathbb{E}\left( \sum_{j=1}^{N-1} \left(\int_{\mathbb{C}} 1_{Z_j}(z) X_z^\e dz\right)^2\right)\\
&=& (N-1)^2 \sup_{j=1, \ldots , N-1} \mathbb{E}\left( \left(\int_{\mathbb{C}} 1_{Z_j}(z) X_z^\e dz\right)^2\right).
\end{eqnarray*}
By Proposition \ref{singleevent}, the last expression is bounded above by $\e^{-15\zeta+o(1)}$. Since $\zeta > 0$ can be made arbitrarily small, this proves \eqref{halfofneeded}. Combining this with \eqref{otherhalfofneeded} (whose proof is identical), we obtain the result.
\end{proof}
From now on we will fix $j,$ and focus on proving Proposition \ref{singleevent}. We define $\bar{X}_z^\e$ by
\begin{equation}\label{barXzeps}
\bar{X}_z^\e := 1_{Z_j}(z) X_z^\e - \mathbb{E}\left( 1_{Z_j}(z) X_z^\e \vert\Phi_{\e^{\alpha_j-\zeta}}(z), \Psi_{\e^{\alpha_j-\zeta}}(z)\right).
\end{equation}
Since
\[
1_{Z_j}(z) X_z^\e = \bar{X}_z^\e + \mathbb{E}\left( 1_{Z_j}(z) X_z^\e \vert\Phi_{\e^{\alpha_j-\zeta}}(z), \Psi_{\e^{\alpha_j-\zeta}}(z)\right),
\]
we have that
\[
\int_{\mathbb{C}} 1_{Z_j}(z) X_z^\e dz = \int_{\mathbb{C}} \bar{X}_z^\e dz + \int_{\mathbb{C}} \mathbb{E}\left( 1_{Z_j}(z) X_z^\e \vert\Phi_{\e^{\alpha_j-\zeta}}(z), \Psi_{\e^{\alpha_j-\zeta}}(z)\right) dz
\]
and therefore
\begin{eqnarray}\label{varexpsplit}
\mathbb{E}\left(\left(\int_{\mathbb{C}} 1_{Z_j}(z) X_z^\e dz\right)^2\right) &=& \mathbb{E}\left(\left(\int_{\mathbb{C}}1_{Z_j}(z)\bar{X}_z^\e dz + \int_{\mathbb{C}} \mathbb{E}\left(1_{Z_j}(z)X_z^\e \vert\Phi_{\e^{\alpha_j-\zeta}}(z), \Psi_{\e^{\alpha_j-\zeta}}(z)\right)dz\right)^2\right)\nonumber\\
&\leq & 2 \mathbb{E}\left(\left(\int_{\mathbb{C}} 1_{Z_j}(z)\bar{X}_z^\e dz\right)^2\right) + 2 \mathbb{E} \left(\left(\int_{\mathbb{C}} \mathbb{E}(1_{Z_j}(z)X_z^\e\vert \Phi_{\e^{\alpha_j-\zeta}}(z),\Psi_{\e^{\alpha_j-\zeta}}(z))dz\right)^2\right).\nonumber\\
\end{eqnarray}
We claim each one of these expectations is small.
\begin{prop}\label{secondmom}
With probability going to $1$ as $\e\to 0,$ we have
\[
\mathbb{E}\left(\left(\int_{\mathbb{C}} 1_{Z_j}(z)\bar{X}_z^\e dz\right)^2\right) \leq \e^{o(1)-C\zeta}.
\]
for an absolute constant $C.$
\end{prop}
\begin{prop}\label{firstmom}
With probability going to $1$ as $\e \to 0,$ we have
\[
\mathbb{E} \left(\left(\int_{\mathbb{C}} \mathbb{E}(1_{Z_j(z)}X_z^\e\vert \Phi_{\e^{\alpha_j-\zeta}}(z),\Psi_{\e^{\alpha_j-\zeta}}(z)) dz\right)^2 \right) \leq \e^{o(1)}.
\]
\end{prop}

Now we will focus on proving Propositions \ref{secondmom} and \ref{firstmom}.
\subsection{Proof of Proposition \ref{secondmom}}
To prove Proposition \ref{secondmom}, we write 
\[
\mathbb{E}\left(\left(\int_{\mathbb{C}} 1_{Z_j}(z) \bar{X}_z^\e dz\right)^2\right) = \mathbb{E}\left(\int_{Z_j}\int_{Z_j} \bar{X}_z^\e \bar{X}_w^\e dzdw\right).
\]
Let $\zeta>0.$ We will decompose this double integral into the set of $(z,w)$ with $\vert z-w\vert \leq \e^{\alpha_{j}-\zeta}$ and where $\vert z-w\vert \geq \e^{\alpha_{j}-\zeta}.$ More precisely, we have
\begin{eqnarray}\label{ondiagoffdiag}
\mathbb{E}\left(\left(\int_{\mathbb{C}} 1_{Z_j}(z) \bar{X}_z^\e dz\right)^2\right) &=& \mathbb{E}\left(\int\int_{\vert z-w\vert < \e^{\alpha_{j}-\zeta}} 1_{Z_j}(z) 1_{Z_j}(w)\bar{X}_z^\e\bar{X}_w^\e dzdw\right)\nonumber\\
&+& \mathbb{E}\left(\int\int_{\vert z-w\vert \geq \e^{\alpha_{j}-\zeta}} 1_{Z_j}(z) 1_{Z_j}(w)\bar{X}_z^\e\bar{X}_w^\e dzdw\right).\nonumber\\
\end{eqnarray}
We claim each term is of order at most $\e^{o(1)-2\zeta}.$ We will prove this in subsections \ref{ondiagsubsec} and \ref{offdiagsubsec}. For the first, we prove the integral is small because of the fact that $\{z,w : \vert z-w\vert \leq \e^{\alpha_j-\zeta}\}$ has small measure. For the second integral, we use the fact that $X_z^\e$ satisfies a long range independence property (see \eqref{xepsindepprop}).

\subsection{Diagonal integral} \label{ondiagsubsec}
The main result of this section is the following.
\begin{prop}\label{ondiag}
We have
\[
\left\vert\mathbb{E}\left(\int\int_{\vert z-w\vert < \e^{\alpha_{j}-\zeta}} \bar{X}_z^\e\bar{X}_w^\e dzdw\right)\right\vert \leq \e^{o(1)-15\zeta}.
\]
\end{prop}

To prove Proposition \ref{ondiag}, we will need the following lemma.
\begin{lemma}\label{smallexp}
We have that
\begin{eqnarray*}
&&\mathbb{E}\left( 1_{Z_j}(z) X_z^\e  \vert\Phi_{\e^{\alpha_j-\zeta}}(z), \Psi_{\e^{\alpha_j-\zeta}}(z)\right)\\
&&\leq \norm{\nabla^2 f}_\infty \mathbb{E}\left(1_{Z_j\cap \mathrm{supp } f}(z)\left.(\mathrm{deg}(z)+1)\frac{\mathrm{diam}(H_z^\e)^2}{\mathrm{area}(H_z^\e)}\right\vert\Phi_{\e^{\alpha_j-\zeta}}(z), \Psi_{\e^{\alpha_j-\zeta}}(z)\right)+\e^{o(1)}.
\end{eqnarray*}
\end{lemma}
\begin{proof}
By Taylor expansion, we have that
\[
f(w) = f(z) + \mathcal{\nabla} f(z)\cdot (w-z) +\mathcal{N}_zf(w-z),
\]
where the Taylor expansion error $\mathcal{N}_z f$ satisfies
\begin{equation}\label{lowerorder}
\vert \mathcal{N}_zf(w-z)\vert \lesssim \norm{\nabla^2f}_\infty \vert z-w\vert^2.
\end{equation}
Let $\tilde{f}$ be the linear function defined by
\[
\tilde{f}_z(w) = f(z) + \nabla f(z) \cdot (w-z).
\]
Then
\[
\Gamma_f(w) = \Gamma_{\tilde{f}_z}(w) + \Gamma_{\mathcal{N}_zf(\cdot-z)}(w).
\]
Let $a,b \in\R,$ and let $\mathbb{P}_{a,b}$ be the law of $\Phi^0,\Psi^0,$ where the fields $\Phi^0,\Psi^0$ are the fields $\Phi, \Psi$ normalized so that $\Phi^0_{\e^{\alpha_j}-\zeta}=a$ and $\Psi^0_{\e^{\alpha_j}-\zeta} = b.$ By the invariance of the fields $\Phi^0, \Psi^0$ under rotations about $z,$ we have that
\begin{equation}\label{rotinv}
\mathbb{E}_{a,b}\left(\frac{\Gamma_{\tilde{f}_z}(z)}{\mathrm{area}\left(H_z^\e\right)}1_{Z_j}(z)\right) = 0.
\end{equation}
Let $M_{a,b}$  denote the Radon-Nikodym derivative of the conditional law of $(\Phi,\Psi)_{B_{\e^{\alpha_j}}(z)}$ given $\{\Phi_{\e^{\alpha_j-\zeta}}(z)=a,\Psi_{\e^{\alpha_j-\zeta}}(z)=b\}$ with respect to $\Phi^0,\Psi^0.$ Then
\[
\mathbb{E}\left( 1_{Z_j}(z) \frac{\Gamma_{\tilde{f}_z}(z)}{\mathrm{area}(H_z^\e)} \vert \Phi_{\e^{\alpha_j-\zeta}}=a, \Psi_{\e^{\alpha_j}-\zeta}(z)=b\right) = \mathbb{E}_{a,b}\left(\frac{\Gamma_{\tilde{f}_z}}{\mathrm{area}(H_z^\e)} M_{a,b}\right).
\]
Using \eqref{rotinv}, we have that
\begin{eqnarray*}
\mathbb{E}\left( \left. 1_{Z_j}(z) \frac{\Gamma_{\tilde{f}_z}(z)}{\mathrm{area}(H_z^\e)} \right\vert \Phi_{\e^{\alpha_j-\zeta}}=a, \Psi_{\e^{\alpha_j}-\zeta}(z)=b\right) &=& \mathbb{E}_{a,b}\left(1_{Z_j}(z)\frac{\Gamma_{\tilde{f}_z}}{\mathrm{area}(H_z^\e)} \left(M_{a,b}-1\right)\right) + \mathbb{E}_{a,b}\left(\frac{\Gamma_{\tilde{f}_z}}{\mathrm{area}(H_z^\e)}\right)\\
&=& \mathbb{E}_{a,b}\left(1_{Z_j}(z)\frac{\Gamma_{\tilde{f}_z}}{\mathrm{area}(H_z^\e)} \left(M_{a,b}-1\right)\right)\\
&\leq & \left(\mathbb{E}_{a,b}\left(1_{Z_j}(z)\left(\frac{\Gamma_{\tilde{f}_z}}{\mathrm{area}\left(H_z^\e\right)}\right)^2\right)\right)^{\frac{1}{2}} \left(\mathbb{E}_{a,b}\left((M_{a,b}-1)^2\right)\right)^{\frac{1}{2}}\\
&\leq & \left(\mathbb{E}_{a,b}\left(1_{Z_j}(z)\left(\frac{\Gamma_{\tilde{f}_z}}{\mathrm{area}\left(H_z^\e\right)}\right)^2\right)\right)^{\frac{1}{2}} \left(\mathbb{E}_{a,b}\left(M_{a,b}^2-1\right)\right)^{\frac{1}{2}},
\end{eqnarray*}
where in the last line we used that
\begin{eqnarray*}
\mathbb{E}_{a,b}\left((M_{a,b}-1)^2\right) &=& \mathbb{E}_{a,b}\left(M_{a,b}^2\right) + 1 - 2 \mathbb{E}_{a,b}(M_{a,b})\\
&=& \mathbb{E}_{a,b}\left(M_{a,b}^2\right) - 1
\end{eqnarray*}
since
\[
\mathbb{E}_{a,b}(M_{a,b}) =1.
\]
By the proof of Lemma A.1 in \cite{gms-harmonic}, we have that $M_{a,b} = M_a M_b,$ where $M_a,M_b$ are independent, and $M_a$ agrees in law with
\[
\frac{\sqrt{\log \tilde{R}}}{\sqrt{\log \tilde{R} + \log \left(1-\vert \tilde{z}\vert^2\right)}}\exp{\left(\frac{a^2}{2\log\tilde{R}}-\frac{\left(a-\Omega_{\e^{\alpha_j}-\zeta}(z)\right)^2}{2\left(\log\tilde{R}+\log\left(1-\vert \tilde{z}\vert^2\right)\right)}\right)},
\]
where
\[
\tilde{R} = \frac{\left(\e^{\alpha_j-\zeta}\right)^2-\vert z\vert^2}{\e^{\alpha_j-\zeta}},
\]
\[
\tilde{z} = -\frac{z}{\e^{\e^{\alpha_j}-\zeta}-\vert z\vert^2}
\]
and finally $\Omega_{\e^{\alpha_j}-\zeta}(z)$ is a centered Gaussian with variance $\log \left(\left(1-\vert \tilde{z}\vert^2\right)^{-1}\right) \lesssim 1.$ Therefore
\[
\mathbb{E}_{a,b}(M_{a,b}^2)- 1 = \mathbb{E}_{a,b}(M_a^2) \mathbb{E}_{a,b}(M_b^2) - 1= \left(\mathbb{E}_{a,b}(M_a^2)-1\right) \left(\mathbb{E}_{a,b}(M_b^2)-1\right) + \left(\mathbb{E}_{a,b}(M_a^2)-1\right) + \left(\mathbb{E}_{a,b}(M_b^2)-1\right).
\]
We have
\begin{eqnarray*}
\mathbb{E}_{a,b}(M_a^2)-1 &=& \mathbb{E}_{a,b} \left(\left(\frac{\log \tilde{R}}{\log \tilde{R} + \log \left(1-\vert \tilde{z}\vert^2\right)}\right)\exp{\left(\frac{a^2}{\log\tilde{R}}-\frac{\left(a-\Omega_{\e^{\alpha_j}-\zeta}(z)\right)^2}{\left(\log\tilde{R}+\log\left(1-\vert \tilde{z}\vert^2\right)\right)}\right)}\right)-1\\
&=& \left(\frac{\log \tilde{R}}{\log \tilde{R} + \log \left(1-\vert \tilde{z}\vert^2\right)}\right) \mathbb{E}_{a,b}\left(\exp{\left(\frac{a^2}{\log\tilde{R}}-\frac{\left(a-\Omega_{\e^{\alpha_j}-\zeta}(z)\right)^2}{\left(\log\tilde{R}+\log\left(1-\vert \tilde{z}\vert^2\right)\right)}\right)}\right) - 1.
\end{eqnarray*}
Now note that
\begin{eqnarray*}
&&\mathbb{E}_{a,b}\left(\exp{\left(\frac{a^2}{\log\tilde{R}}-\frac{\left(a-\Omega_{\e^{\alpha_j}-\zeta}(z)\right)^2}{\left(\log\tilde{R}+\log\left(1-\vert \tilde{z}\vert^2\right)\right)}\right)}\right) \\
&&= \int_\R \exp{\left(\frac{a^2}{\log\tilde{R}}-\frac{\left(a-\omega\right)^2}{\left(\log\tilde{R}+\log\left(1-\vert \tilde{z}\vert^2\right)\right)}\right)} \frac{1}{\sqrt{2\pi \log\left(\left(1-\vert \tilde{z}\vert^2\right)^{-1}\right)}}\exp{\left(-\frac{\omega^2}{2 \log\left(\left(1-\vert \tilde{z}\vert^2\right)^{-1}\right)}\right)}d\omega\\
&&=\frac{\exp{\left(\frac{a^2}{\log{\tilde{R}}}\right)}}{\sqrt{2\pi \log\left(\left(1-\vert \tilde{z}\vert^2\right)^{-1}\right)}}e^{-\frac{c_1c_2a^2}{c_1+c_2}}\int_\R e^{-(c_1+c_2)\left(\omega-\frac{ac_1}{c_1+c_2}\right)^2} d\omega\\
&&= \frac{\exp{\left(\frac{a^2}{\log{\tilde{R}}}\right)}}{\sqrt{2\pi \log\left(\left(1-\vert \tilde{z}\vert^2\right)^{-1}\right)}}e^{-\frac{c_1c_2a^2}{c_1+c_2}} \left(\frac{\pi}{c_1+c_2}\right)^{\frac{1}{2}},
\end{eqnarray*}
where
\[
c_1 = \frac{1}{\log{\tilde{R}}+\log\left(1-\vert \tilde{z}\vert^2\right)}, \; c_2 = \frac{1}{2\log\left(\left(1-\vert \tilde{z}\vert^2\right)^{-1}\right)}.
\]
Note that $c_1\lesssim \frac{1}{\log \tilde{R}} \lesssim \frac{1}{\log \e} \ll c_2.$ This implies that
\begin{eqnarray*}
&&\mathbb{E}_{a,b}\left(\exp{\left(\frac{a^2}{\log\tilde{R}}-\frac{\left(a-\Omega_{\e^{\alpha_j}-\zeta}(z)\right)^2}{\left(\log\tilde{R}+\log\left(1-\vert \tilde{z}\vert^2\right)\right)}\right)}\right)\\
&&\leq \left(1+\e^{o(1)}\right) \frac{\exp{\left(\frac{a^2}{\log{\tilde{R}}}\right)}}{\sqrt{2 \log\left(\left(1-\vert \tilde{z}\vert^2\right)^{-1}\right)}}e^{-\frac{c_1c_2a^2}{c_1+c_2}} \left(\frac{1}{c_2}\right)^{\frac{1}{2}}\\
&&\leq \left(1+\e^{o(1)}\right) \exp{\left(\frac{a^2}{\log{\tilde{R}}}\right)} e^{-\frac{c_1c_2a^2}{c_1+c_2}} \lesssim \e^{o(1)} \exp{\left(\frac{a^2}{\log{\tilde{R}}}\right)} e^{-c_1a^2}\\
&&= 1+\e^{o(1)}.
\end{eqnarray*}
Hence
\[
\mathbb{E}\left((M_{a,b}-1)^2\right) \lesssim \e^{o(1)}.
\]
Together with \eqref{Hzbound} we obtain
\[
\mathbb{E}_{a,b}\left(\left.\frac{\Gamma_{\tilde{f}_z}(z)}{\mathrm{area}\left(H_z^\e\right)}\right\vert\Phi_{\e^{\alpha_j-\zeta}}=a, \Psi_{\e^{\alpha_j-\zeta}}=b\right) = \e^{o(1)}.
\]
Let
\[
\mathcal{B}_z^\e := \frac{1}{\mathrm{area}(H_z^\e)}\left(\mathcal{N}_zf(m_z-z) + \sum_{y: v_y=z} \mathcal{N}_zf(m_y-z)\right).
\]
Therefore
\begin{eqnarray}\label{bepsthing}
\mathbb{E}\left( 1_{Z_j}(z) X_w^\e \vert\Phi_{\e^{\alpha_j-\zeta}}(w), \Psi_{\e^{\alpha_j-\zeta}}(w)\right) &=& \mathbb{E}\left(1_{Z_j\cap \mathrm{supp }f}(z)\frac{\Gamma_f(z)}{\mathrm{area}(H_z^\e)}\vert \Phi_{\e^{\alpha_j-\zeta}}(z), \Psi_{\e^{\alpha_j-\zeta}}(z)\right)\nonumber\\
&=& \mathbb{E}\left(\left. 1_{Z_j\cap \mathrm{supp }f}(z)\mathcal{B}_z^\e\right\vert \Phi_{\e^{\alpha_j-\zeta}}(z), \Psi_{\e^{\alpha_j-\zeta}}(z)\right) +\e^{o(1)}.
\end{eqnarray}
Now using \eqref{lowerorder} we have
\begin{eqnarray*}
\vert \mathcal{B}_z^\e \vert &\leq &\frac{\norm{\nabla^2f}_\infty}{\mathrm{area}(H_z^\e)} \left(\vert m_z-z\vert^2 + \sum_{y : v_y=z}\vert m_y-z\vert^2\right)\\
&\leq & \frac{\norm{\nabla^2f}_\infty}{\mathrm{area}(H_z^\e)} (\mathrm{deg}(x_z^\e)+1) (\mathrm{diam}(H_z^\e))^2.
\end{eqnarray*}
Combining with \eqref{bepsthing} this completes the proof.
\end{proof}

Now we will prove Proposition \ref{ondiag}. 
\begin{proof}[Proof of Proposition \ref{ondiag}.]
We have
\begin{eqnarray}\label{resplit}
&&\int\int_{\vert z-w\vert < \e^{\alpha_j-\zeta}} \bar{X}_z^\e\bar{X}_w^\e dzdw\nonumber\\
&&= \int\int_{\vert z-w\vert < \e^{\alpha_j-\zeta}} 1_{Z_j}(z)1_{Z_j}(w) X_z^\e X_w^\e dzdw\nonumber\\
&&- 2 \int\int_{\vert z-w\vert < \e^{\alpha_j-\zeta}} 1_{Z_j}(z)X_z^\e \mathbb{E}\left( 1_{Z_j}(w))\vert\Phi_{\e^{\alpha_j-\zeta}}(w), \Psi_{\e^{\alpha_j -\zeta}}(w)\right) dz dw\nonumber\\
&&+ \int\int_{\vert z-w\vert < \e^{\alpha_j-\zeta}} \mathbb{E}\left( 1_{Z_j}(z) X_z^\e \vert\Phi_{\e^{\alpha_j-\zeta}}(z), \Psi_{\e^{\alpha_j-\zeta}}(z)\right) \mathbb{E}\left( 1_{Z_j}(w) X_w^\e \vert\Phi_{\e^{\alpha_j-\zeta}}(w), \Psi_{\e^{\alpha_{j+1}}}(w)\right) dz dw\nonumber\\
&&=:A_1-2A_2+A_3
\end{eqnarray}
Now we will bound each term individually. For any $z,w \in \mathbb{C},$ let $E_{z,w}$ denote the event that $\vert z-w\vert < \e^{\alpha_j-\zeta},$ and $z,w \in Z_j.$ We have
\begin{equation}\label{A1termfirst}
\int\int_{\vert z-w\vert < \e^{\alpha_j-\zeta}} 1_{Z_j}(z) 1_{Z_j}(w) X_z^\e X_w^\e dzdw = \int_{\mathbb{C}}\int_{\mathbb{C}} X_z^\e X_w^\e 1_{E_{z,w}}dzdw.
\end{equation}
Let $V_\e$ be defined by
\[
V_\e= \{v \in \mathcal{V}\mathcal{G}_\e: H_v^\e \cap \mathrm{Supp}\; f \neq \varnothing\}.
\]
Going back to the discrete sum form, we have that
\begin{equation}\label{A1termsecond}
\vert A_1\vert = \left\vert \int_{\mathbb{C}}\int_{\mathbb{C}} X_z^\e X_w^\e 1_{E_{z,w}}dzdw\right\vert \leq \sum_{z,w \in V_\e} \vert\Gamma_f(z)\vert\vert\Gamma_f(w)\vert 1_{E_{z,w}}.
\end{equation}
Now using Lemma \ref{Gfnotfound} we obtain that 
\begin{equation}\label{eq1}
\sum_{z,w \in V_\e} \vert\Gamma_f(z)\vert\vert\Gamma_f(w)\vert 1_{E_{z,w}} \leq \norm{\nabla f}_\infty^2\sum_{z,w \in V_\e} (\mathrm{deg}(z)+1)(\mathrm{deg}(w)+1) \mathrm{diam}(H_z^\e)\mathrm{diam}(H_w^\e)1_{E_{z,w}}.
\end{equation}
Note that if $E_{z,w}$ holds, then
\[
\mathrm{diam}(H_z^\e)\mathrm{diam}(H_w^\e) \leq \e^{2\alpha_j}.
\]
Hence
\begin{eqnarray}\label{eq2}
&&\sum_{z,w \in V_\e} (\mathrm{deg}(z)+1)(\mathrm{deg}(w)+1) \mathrm{diam}(H_z^\e)\mathrm{diam}(H_w^\e)1_{E_{z,w}}\nonumber\\
&&\leq \e^{2 \alpha_j}\sum_{z,w \in V_\e} (\mathrm{deg}(z)+1)(\mathrm{deg}(w)+1)1_{E_{z,w}}.
\end{eqnarray}
Then by the definition of $Z_j$ we have
\[
\sup_{z \in V_\e}\mathrm{deg}(z) 1_{Z_j}(z) \leq C(\log\e^{-1})^{3}.
\]
This implies by \eqref{Ejdef} that
\begin{eqnarray*}
&&\sum_{z,w \in V_\e\cap \mathrm{Supp}f} (\mathrm{deg}(z)+1)(\mathrm{deg}(w)+1) \mathrm{diam}(H_z^\e)\mathrm{diam}(H_w^\e)1_{E_{z,w}}\\
&&\leq C^2\left(\log\e^{-1}\right)^6 \e^{2 \alpha_j} \sum_{z,w \in V_\e} 1_{E_{z,w}}.
\end{eqnarray*}
Combining this with \eqref{eq1} we obtain
\begin{equation}\label{eq7}
\sum_{z,w \in V_\e} \vert\Gamma_f(z)\vert\vert\Gamma_f(w)\vert 1_{E_{z,w}} \leq \norm{\nabla f}_\infty^2 C^2\left(\log\e^{-1}\right)^6 \e^{2\alpha_j} \sum_{z,w \in V_\e}1_{E_{z,w}}.
\end{equation}
Now we will estimate the sum
\[
\sum_{z,w \in V_\e}1_{E_{z,w}}.
\]
We have
\[
\sum_{z,w \in V_\e\cap Z_j}1_{E_{z,w}} = \sum_{z \in V_\e\cap Z_j} \#\{w\in V_\e\cap Z_j:E_{z,w} \mbox{ holds}\}.
\]
We also have that
\begin{equation}\label{eq3}
\#\{w \in V_\e\cap Z_j:E_{z,w} \mbox{ holds}\} \leq \#\left\{w \in V_\e\cap Z_j:H_w^\e\cap B_{\e^{\alpha_j-\zeta}}(z) \neq \varnothing,\mathrm{diam}\left(\widehat{H}_w^\e\right) \in \left[\e^{\alpha_{j+1}},\e^{\alpha_j}\right]\right\}.
\end{equation}
By \eqref{Ejdef}, if $w \in V_\e,$ there are at most $C (\log\e^{-1})^3$ cells $y$ such that $y\sim w.$ Therefore if $\mathrm{diam}(\widehat{H}_w^\e)\geq \e^{\alpha_{j+1}},$ then there exists a $w$ such that either $y=w$ or $y \sim w$ and
\[
\mathrm{diam}(H_y^\e) \geq \frac{\e^{\alpha_{j+1}}}{C(\log\e^{-1})^3}.
\]
Hence by \eqref{Ejdef}
\begin{eqnarray}\label{eq4}
&&\#\left\{w \in V_\e\cap Z_j : H_w^\e\cap B_{\e^{\alpha_j-\zeta}}(z)\neq \varnothing , \mathrm{diam}\left(\widehat{H}_w^\e\right) \geq \e^{\alpha_{j+1}} \right\}\nonumber\\
&& \leq C (\log \e^{-1})^3 \#\left\{y\in B_{\e^{\alpha_j-\zeta}+\e^{\alpha_j}}(z): \mathrm{diam}\left(H_y^\e\right) \geq \frac{\e^{\alpha_j+\zeta}}{C(\log\e^{-1})^3} \right\}.
\end{eqnarray}
Now, using \eqref{Ejdef} we have that
\begin{eqnarray*}
&&\#\left\{y \in B_{\e^{\alpha_j-\zeta}+\e^{\alpha_j}}(z): \mathrm{diam}\left(H_y^\e\right) \geq \frac{\e^{\alpha_j+\zeta}}{C(\log\e^{-1})^3} \right\}\\
&&\leq \#\left\{y \in B_{2\e^{\alpha_j-\zeta}}(z): \mathrm{area}\left(H_y^\e\right) \geq \frac{\e^{(2\alpha_j+3\zeta)}}{C^2(\log\e^{-1})^6} \right\}\\
&& \leq 4\e^{2\alpha_j-2\zeta} \frac{\e^{-(2\alpha_j+3\zeta)}}{C^{-2}(\log\e^{-1})^{-6}} = 4C^2\e^{-5\zeta}(\log \e^{-1})^6
\end{eqnarray*}
where in the last line we used the fact that the cells intersect only along their boundaries. Putting this together with \eqref{eq3},\eqref{eq4} we obtain
\begin{equation}\label{eq5}
\#\{w\in V_\e\cap Z_j: E_{z,w} \mbox{ holds}\} \leq 4C^3\e^{-5\zeta}(\log \e^{-1})^9
\end{equation}
and so
\begin{equation}\label{Hzbound}
\#\{w\in V_\e\cap Z_j: E_{z,w} \mbox{ holds}\} \leq \e^{-5\zeta+o(1)}.
\end{equation}
Hence
\begin{equation}\label{eq6}
\sum_{z,w \in V_\e}1_{E_{z,w}} \leq \#\left\{z \in V_\e\cap Z_j:\mathrm{diam}\left(\widehat{H}_z^\e\right) \in\left[\e^{\alpha_{j+1}},\e^{\alpha_j}\right]\right\} \e^{-5\zeta+o(1)}.
\end{equation}
Using \eqref{Ejdef} we obtain that
\begin{eqnarray*}
&&\#\left\{z \in V_\e\cap Z_j:\mathrm{diam}\left(\widehat{H}_z^\e\right) \in\left[\e^{\alpha_{j+1}},\e^{\alpha_j}\right]\right\}\\
&&\leq \#\left\{z \in V_\e\cap Z_j:\mathrm{diam}\left(\widehat{H}_z^\e\right) \geq \e^{\alpha_{j+1}} \right\}\\
&&\leq C \log\e^{-1}\#\left\{z \in V_\e\cap Z_j : \mathrm{diam}(H_z^\e) \geq \frac{\e^{\alpha_{j+1}}}{C\log\e^{-1}}\right\}.
\end{eqnarray*}
where in the last line we used the fact that there are at most $(\log \e^{-1})^3+1$ cells in $\widehat{H}_z^\e$, and so if $\mathrm{diam}\left(\widehat{H}_z^\e\right) \geq \e^{\alpha_{j+1}}$ then there is a cell in $\widehat{H}_z^\e$ whose diameter is at least $\frac{\e^{\alpha_{j+1}}}{C\log\e^{-1}}.$ Now using the same argument as for \eqref{eq5} we obtain that
\begin{eqnarray*}
&&\#\left\{z \in V_\e\cap Z_j:\mathrm{diam}\left(\widehat{H}_z^\e\right) \geq \e^{\alpha_{j+1}}\right\}\\
&&\leq C (\log\e^{-1})^3\left(\frac{\e^{\alpha_{j+1}}}{C\log\e^{-1}}\right)^{-2}\e^{-2\zeta} \mathrm{area}(\mathrm{Supp}f)\\
&& \lesssim C^3(\log \e^{-1})^5 \e^{-2\alpha_{j+1}-2\zeta}.
\end{eqnarray*}
Combining this with \eqref{eq6} we obtain
\begin{equation}\label{eq9}
\sum_{z,w \in V_\e\cap Z_j}1_{E_{z,w}} \leq \e^{-2\alpha_{j+1}-7\zeta+o(1)}.
\end{equation}
Plugging this into \eqref{eq7} and then back into \eqref{A1termfirst} and \eqref{A1termsecond} we obtain that
\begin{equation}\label{A1conc}
\vert A_1\vert \leq \e^{o(1)-9\zeta}.
\end{equation}
Now we will bound $A_2.$ By Lemma \ref{smallexp} we have
\begin{eqnarray*}
&&\vert A_2\vert\\ &&\leq \int\int_{\vert z-w\vert < \e^{\alpha_j-\zeta}} 1_{Z_j}(z)\vert X_z^\e\vert \vert\mathbb{E}\left( 1_{Z_j}(z) X_w^\e \vert\Phi_{\e^{\alpha_j-\zeta}}(w), \Psi_{\e^{\alpha_j-\zeta}}(w)\right)\vert dz dw+\e^{o(1)}\\
&&\leq  \int\int_{\vert z-w\vert < \e^{\alpha_j-\zeta}} 1_{Z_j}(z)\vert X_z^\e\vert \norm{\nabla^2 f}_\infty \mathbb{E}\left((\mathrm{deg}(x_w^\e)+1)\frac{\mathrm{diam}(H_w^\e)^2}{\mathrm{area}(H_w^\e)}\vert\Phi_{\e^{\alpha_j-\zeta}}(w), \Psi_{\e^{\alpha_j-\zeta}}(w)\right) dzdw+\e^{o(1)}\\
&&=  \norm{\nabla^2 f}_\infty \int\int_{\vert z-w\vert < \e^{\alpha_j-\zeta}} 1_{Z_j}(z) \vert X_z^\e\vert\mathbb{E}\left(\left. (\mathrm{deg}(x_w^\e)+1)\frac{\mathrm{diam}(H_w^\e)^2}{\mathrm{area}(H_w^\e)} \right\vert \Phi_{\e^{\alpha_j-\zeta}}(w), \Psi_{\e^{\alpha_j-\zeta}}(w)\right) dzdw+\e^{o(1)}.
\end{eqnarray*}
Note that by Lemma \ref{Gfnotfound} we have
\[
\vert X_z^\e\vert \leq (\mathrm{deg}(x_z^\e)+1) \norm{\nabla f}_\infty \frac{\mathrm{diam}(H_z^\e)}{\mathrm{area}(H_z^\e)}.
\]
This implies that
\begin{eqnarray*}
&&\frac{\vert A_2\vert}{\norm{\nabla f}_\infty\norm{\nabla^2 f}_\infty} \\
&&\leq \int\int_{\vert z-w\vert < \e^{\alpha_j-\zeta}}\mathbb{E}\left(\left. 1_{E_{z,w}}(\mathrm{deg}(x_z^\e)+1)(\mathrm{deg}(x_w^\e)+1) \frac{\mathrm{diam}(H_z^\e)}{\mathrm{area}(H_z^\e)}\frac{\mathrm{diam}(H_w^\e)^2}{\mathrm{area}(H_w^\e)} \right\vert \Phi_{\e^{\alpha_j-\zeta}}(w), \Psi_{\e^{\alpha_j-\zeta}}(w)\right)dzdw\\
&&+\e^{o(1)}.
\end{eqnarray*}
Now by definition of $Z_j$ we have that
\begin{eqnarray*}
(\mathrm{deg}(x_z^\e)+1)(\mathrm{deg}(x_w^\e)+1) \frac{\mathrm{diam}(H_z^\e)}{\mathrm{area}(H_z^\e)}\frac{\mathrm{diam}(H_w^\e)^2}{\mathrm{area}(H_w^\e)}1_{E_{z,w}} \lesssim \e^{-2\zeta + o(1)} \frac{1}{\mathrm{diam}(H_z^\e)}.
\end{eqnarray*}
This implies that
\begin{eqnarray*}
\frac{\vert A_2\vert}{\norm{\nabla f}_\infty\norm{\nabla^2 f}_\infty}
& \leq & \e^{-2\zeta +o(1)}\int\int_{\vert z-w\vert < \e^{\alpha_j-\zeta}} \frac{1_{E_{z,w}}}{\mathrm{diam}(H_z^\e)}dzdw+\e^{o(1)}\\
&\leq & \e^{\alpha_j-3\zeta}\e^{o(1)}\int\int_{\vert z-w\vert < \e^{\alpha_j-3\zeta}} \frac{1_{E_{z,w}}\mathrm{diam}(H_z^\e)\mathrm{diam}(H_w^\e)}{\mathrm{area}(H_z^\e)\mathrm{area}(H_w^\e)}dzdw+\e^{o(1)}\\
&= & \e^{o(1)}\e^{\alpha_j-4\zeta}\sum_{z,w \in V_\e\cap Z_j} \mathrm{diam}(H_z^\e)\mathrm{diam}(H_w^\e)1_{E_{z,w}}+\e^{o(1)}\\
&\leq & \e^{o(1)} \e^{\alpha_j-4\zeta} \e^{2\alpha_j-2\zeta}\sum_{z,w \in V_\e\cap Z_j} 1_{E_{z,w}}+\e^{o(1)}.
\end{eqnarray*}
By \eqref{eq9}, we have that
\[
\frac{\vert A_2\vert}{\norm{\nabla f}_\infty\norm{\nabla^2 f}_\infty} \lesssim \e^{-15\zeta + o(1)}
\]
hence
\begin{equation}\label{A2conc}
\vert A_2\vert \lesssim \e^{-15\zeta + o(1)}.
\end{equation}
To bound $A_3$ we note that by Lemma \ref{smallexp} together with \eqref{Ejdef} we have
\begin{eqnarray*}
&&\mathbb{E}\left( \vert 1_{Z_j}(z) X_z^\e\vert  \vert\Phi_{\e^{\alpha_j-\zeta}}(z), \Psi_{\e^{\alpha_j-\zeta}}(z)\right)\\&& \leq \norm{\nabla^2 f}_\infty \mathbb{E}\left(\left.1_{Z_j}(z)(\mathrm{deg}(x_z^\e)+1)\frac{\mathrm{diam}(H_z^\e)^2}{\mathrm{area}(H_z^\e)}\right\vert\Phi_{\e^{\alpha_j-\zeta}}(z), \Psi_{\e^{\alpha_j-\zeta}}(z)\right)+\e^{o(1)}\\
&&\lesssim \e^{o(1)-\zeta},
\end{eqnarray*}
thus
\[
\vert A_3\vert \leq \e^{o(1)-2\zeta}\int \int_{\vert z-w \vert \leq \e^{\alpha_j-\zeta}} 1_{\mathrm{Supp}(f)}(z)1_{\mathrm{Supp}(f)}(w) dz dw+\e^{o(1)} = \e^{o(1)-2\zeta} \vert \mathrm{supp}\; f\vert^2+\e^{o(1)} \lesssim \e^{o(1)-2\zeta}.
\]
Combining this with \eqref{A1conc}, \eqref{A2conc} and \eqref{resplit} yields the result of Proposition \ref{ondiag}.
\end{proof}

\subsection{Off-diagonal integral}\label{offdiagsubsec}
The main result of this section is the following.
\begin{prop}\label{offdiag}
We have
\[
\mathbb{E}\left(\int\int_{\vert z-w\vert \geq \e^{\alpha_j-\zeta}} \bar{X}_z^\e \bar{X}_w^\e dzdw\right) \lesssim \e^{-2\zeta+o(1)}.
\]
\end{prop}
The main ingredient is a long range independence property that $\bar{X}_z^\e$ satisfies. First we rewrite the integral in Proposition \ref{offdiag}. Recall the definitions of $X_z^\e$ and $\bar{X}_z^\e,$ \eqref{Xzeps} and \eqref{barXzeps}. Suppose $z,w \in \mathbb{C}.$ We define $\mathcal{F}_{z,w}$ to be the $\sigma$-algebra generated by $(\Phi,\Psi)\vert_{\mathbb{C}\setminus \left(B_{\frac{\vert z-w\vert}{4}}(z)\cup B_{\frac{\vert z-w\vert}{4}}(w)\right)},$ that is
\begin{equation}\label{sigmaalgdef}
\mathcal{F}_{z,w} = \sigma\left((\Phi,\Psi)\vert_{\mathbb{C}\setminus \left(B_{\vert z-w\vert/4(z)}\cup B_{\vert z-w\vert/4}(w)\right)}\right).
\end{equation}
We will need the following important lemma, which is proven in~\cite{gms-harmonic}.
\begin{lemma}\cite[Lemma A.3]{gms-harmonic}\label{A3harmonic}
Let $B_R(w)$ be a ball containing the unit disk and suppose $h$ is a whole-plane GFF normalized so that its circle average over $B_R(w)$ vanishes. Let $p,p',s$ be such that $p'>p>1,$ $s\in(0,1).$ Then there exist constants $a=a(p,p')>0$ and $b=b(p,p',s)>0$ such that for any $0< \delta \leq a$ and any random variable $X$ determined almost surely by $h_{\delta}(0)$ and $h_1(0),$ we have
\[
\mathbb{E}\left(\left\vert \mathbb{E}(X\vert h_{\mathbb{C}\setminus B_1(0)}) - \mathbb{E}(X\vert h_1(0))\right\vert^p\right) \lesssim \delta^{sp} \mathbb{E}\left(\vert X\vert^p\right) + e^{-b/\delta^{1-s}} \mathbb{E}\left(\vert X\vert^{p'}\right)^{\frac{p}{p'}}.
\]
\end{lemma}

\begin{proof}[Proof of Proposition \ref{offdiag}]
Let $\mathcal{F}_{z,w}$ be as in \eqref{sigmaalgdef}. Then we have
\begin{equation}\label{eq8}
\mathbb{E}\left(\bar{X}_z^\e\bar{X}_w^\e\right) = \mathbb{E}\left(\mathbb{E}\left(\bar{X}_z^\e\bar{X}_w^\e\vert \mathcal{F}_{z,w}\right)\right).
\end{equation}
Note that when conditioning on $\mathcal{F}_{z,w},$ $\bar{X}_\e(z)$ and $\bar{X}_\e(w)$ become independent. Hence
\begin{equation}\label{xepsindepprop}
\mathbb{E}\left(\bar{X}_z^\e\bar{X}_w^\e\vert \mathcal{F}_{z,w}\right) = \mathbb{E}\left(\bar{X}_z^\e\vert \mathcal{F}_{z,w}\right)\mathbb{E}\left(\bar{X}_w^\e\vert \mathcal{F}_{z,w}\right).
\end{equation}
Plugging into \eqref{eq8} and using Cauchy-Schwarz we obtain
\begin{eqnarray*}
&&\mathbb{E}\left(\int\int_{\vert z-w\vert \geq \e^{\alpha_j-\zeta}}\bar{X}_z^\e\bar{X}_w^\e dzdw\right)\\
&&=  \int\int_{\vert z-w\vert \geq \e^{\alpha_j-\zeta}}1_{z,w\in\mathrm{Supp}f}\mathbb{E}\left(\bar{X}_z^\e\bar{X}_w^\e\right)dzdw\\
&&=  \int\int_{\vert z-w\vert \geq \e^{\alpha_j-\zeta}} 1_{z,w\in\mathrm{Supp}f}\mathbb{E}\left(\mathbb{E}\left(\bar{X}_z^\e\vert \mathcal{F}_{z,w}\right)\mathbb{E}\left(\bar{X}_w^\e\vert \mathcal{F}_{z,w}\right)\right) dzdw \\
&&\leq \int\int_{\vert z-w\vert \geq \e^{\alpha_j-\zeta}} 1_{z,w\in\mathrm{Supp}f}\mathbb{E}\left(\left(\mathbb{E}\left(\bar{X}_z^\e\vert \mathcal{F}_{z,w}\right)\right)^2\right)^{\frac{1}{2}}\mathbb{E}\left(\left(\mathbb{E}\left(\bar{X}_w^\e\vert \mathcal{F}_{z,w}\right)\right)^2\right)^{\frac{1}{2}}dzdw.
\end{eqnarray*}
Using Cauchy-Schwarz again we obtain
\begin{eqnarray*}
&&\mathbb{E}\left(\int\int_{\vert z-w\vert \geq \e^{\alpha_j-\zeta}}\bar{X}_z^\e\bar{X}_w^\e dzdw\right)\\
&&\leq \left( \int\int_{\vert z-w\vert \geq \e^{\alpha_j-\zeta}}  1_{z,w\in\mathrm{Supp}f}\mathbb{E}\left(\left(\mathbb{E}\left(\bar{X}_z^\e\vert \mathcal{F}_{z,w}\right)\right)^2\right)dzdw\right)^{\frac{1}{2}}\\
&&\cdot\left( \int\int_{\vert z-w\vert \geq \e^{\alpha_j-\zeta}}  1_{z,w\in\mathrm{Supp}f}\mathbb{E}\left(\left(\mathbb{E}\left(\bar{X}_w^\e\vert \mathcal{F}_{z,w}\right)\right)^2\right)dzdw\right)^{\frac{1}{2}}\\
&&= \left( \int\int_{\vert z-w\vert \geq \e^{\alpha_j-\zeta}}  1_{z,w\in\mathrm{Supp}f}\mathbb{E}\left(\left(\mathbb{E}\left(\bar{X}_z^\e\vert \mathcal{F}_{z,w}\right)\right)^2\right)dzdw\right).
\end{eqnarray*}
Using Lemma \ref{A3harmonic} for $\delta = \frac{\e^{\alpha_{j+1}+\zeta}}{\vert z-w\vert}$ and appropriately rescaling we obtain that
\begin{equation}\label{eq12}
\mathbb{E}\left(\left(\mathbb{E}\left(\bar{X}_z^\e\vert \mathcal{F}_{z,w}\right)\right)^2\right)\leq \frac{\e^{2s \alpha_{j+1}+2s\zeta}}{\vert z-w\vert^{2s}} \mathbb{E}\left(\left(\bar{X}_z^\e\right)^2\right) + e^{-\frac{b\vert z-w\vert^{1-s}}{\e^{(1-s)(\alpha_{j+1}+\zeta)}}} \left(\mathbb{E}\left(\left(\bar{X}_z^\e\right)^3\right)\right)^{\frac{2}{3}}.
\end{equation}
Note that
\begin{eqnarray}\label{cubic}
\mathbb{E}\left(\left(\bar{X}_z^\e\right)^3\right) &=& \mathbb{E}\left(\left(1_{Z_j}(z) X_z^\e - \mathbb{E}\left( 1_{Z_j}(z) X_z^\e \vert\Phi_{\e^{\alpha_j-\zeta}}(z), \Psi_{\e^{\alpha_j-\zeta}}(z)\right)\right)^3\right)\nonumber\\
&\leq &4 \mathbb{E}\left(1_{Z_j}(z)\left(X_z^\e\right)^3\right) + 4 \mathbb{E}\left(\left(\mathbb{E}\left( 1_{Z_j}(z) X_z^\e \vert\Phi_{\e^{\alpha_j-\zeta}}(z), \Psi_{\e^{\alpha_j-\zeta}}(z)\right)\right)^3\right)\nonumber\\
&\leq & 4 \mathbb{E}\left(1_{Z_j}(z)\left(X_z^\e\right)^3\right) + 4 \mathbb{E}\left(\mathbb{E}\left( 1_{Z_j}(z) \left(X_z^\e \right)^3\vert\Phi_{\e^{\alpha_j-\zeta}}(z), \Psi_{\e^{\alpha_j-\zeta}}(z)\right)\right)\nonumber\\
&= & 8 \mathbb{E}\left(1_{Z_j}(z)\left(X_z^\e\right)^3\right).
\end{eqnarray}
We also have by Lemma \ref{Gfnotfound} that
\begin{eqnarray*}
\mathbb{E}\left(1_{Z_j}(z)\left(X_z^\e\right)^3\right) &=& \norm{\nabla f}_\infty^3\mathbb{E}\left(1_{Z_j}(z)\frac{(\mathrm{deg}(x_z^\e)+1)^3\mathrm{diam}(H_z^\e)^3}{\mathrm{area}\left(H_z^\e\right)^3}\right)\\
&\leq & \norm{\nabla f}_\infty^3 \left(\mathbb{E}\left(1_{Z_j}(z)(\mathrm{deg}(x_z^\e)+1)^6\left(\frac{\mathrm{diam}(H_z^\e)^2}{\mathrm{area}\left(H_z^\e\right)}\right)^6\right)\right)^{\frac{1}{2}}\left(\mathbb{E}\left(1_{Z_j}(z)\frac{1}{\mathrm{diam}\left(H_z^\e\right)^6}\right)\right)^{\frac{1}{2}}.
\end{eqnarray*}
Recalling the definition of $Z_j$ in \eqref{Ejdef} we obtain that for some deterministic constant $C$ independent of $\e$ we have
\[
\mathbb{E}\left(1_{Z_j}(z)\left(X_z^\e\right)^3\right) \leq C \left(\mathbb{E}\left(1_{Z_j}(z)\frac{1}{\mathrm{diam}\left(H_z^\e\right)^6}\right)\right)^{\frac{1}{2}}.
\]
Recall that if $Z_j$ holds, then $\mathrm{area}(H_z^\e) \geq \e^\beta,$ and hence $\mathrm{diam} \; (H_z^\e) \geq (\e^\beta/\pi)^{1/2}.$ This implies that
\[
\mathbb{E}\left(1_{Z_j}(z)\left(X_z^\e\right)^3\right) \lesssim \e^{\frac{3}{2}\beta}.
\]
This implies the second term in \eqref{eq12} is negligible. Plugging this into \eqref{eq12} we obtain
\begin{eqnarray}
\mathbb{E}\left(\left(\mathbb{E}\left(\bar{X}_z^\e\vert \mathcal{F}_{z,w}\right)\right)^2\right) &\leq &\frac{\e^{2s \alpha_{j+1}+2s\zeta}}{\vert z-w\vert^{2s}} \mathbb{E}\left(\left(\bar{X}_z^\e\right)^2\right) + e^{-\frac{b\vert z-w\vert^{1-s}}{\e^{(1-s)(\alpha_{j+1}+\zeta)}}} \left(\mathbb{E}\left(\left(\bar{X}_z^\e\right)^3\right)\right)^{\frac{2}{3}}\nonumber\\
&\lesssim & \frac{\e^{2s \alpha_{j+1}+2s\zeta}}{\vert z-w\vert^{2s}} \mathbb{E}\left(\left(\bar{X}_z^\e\right)^2\right)\nonumber
\end{eqnarray}
for appropriate constants $s,b$ and thus
\begin{equation}\label{longrangeindep}
\mathbb{E}\left(\int\int_{\vert z-w\vert \geq \e^{\alpha_j-\zeta}}\bar{X}_z^\e\bar{X}_w^\e dzdw\right) \leq  \mathbb{E}\left(\int\int_{\vert z-w\vert \geq \e^{\alpha_j-\zeta}} 1_{z,w\in\mathrm{Supp}f}\frac{\e^{2s \alpha_{j+1}+2s\zeta}}{\vert z-w\vert^{2s}} \mathbb{E}\left(\left(\bar{X}_z^\e\right)^2\right)dzdw\right)
\end{equation}
By computations similar to \eqref{cubic} one has
\begin{comment}
\begin{equation}\label{eq10}
\mathbb{E}\left((\bar{X}_z^\e)^2\right) \leq 2 \mathbb{E}\left(1_{Z_j}(z)(X_z^\e)^2\right) + 2 \mathbb{E}\left(\mathbb{E}\left(1_{Z_j}(z)X_z^\e \vert \Phi_{\e^{\alpha_j-\zeta}}(z), \Psi_{\e^{\alpha_j-\zeta}}(z)\right)^2\right)
\end{equation}
Using Jensen's inequality, we have
\[
\mathbb{E}\left(\mathbb{E}\left(1_{Z_j}(z)X_z^\e \vert \Phi_{\e^{\alpha_j-\zeta}}(z), \Psi_{\e^{\alpha_j-\zeta}}(z)\right)^2\right) \leq \mathbb{E}\left(\mathbb{E}\left(1_{Z_j}(z)(X_z^\e)^2 \vert \Phi_{\e^{\alpha_j-\zeta}}(z), \Psi_{\e^{\alpha_j-\zeta}}(z)\right)\right) = \mathbb{E}\left(1_{Z_j}(z)(X_z^\e)^2\right),
\]
and so plugging this into \eqref{eq10} we obtain
\end{comment}
\[
\mathbb{E}\left((\bar{X}_z^\e)^2\right) \leq 4 \mathbb{E}\left(1_{Z_j}(z)(X_z^\e)^2\right).
\]
Therefore plugging this into \eqref{longrangeindep} we obtain
\[
\mathbb{E}\left(\left(\mathbb{E}\left(\bar{X}_z^\e\vert \mathcal{F}_{z,w}\right)\right)^2\right) \lesssim \frac{\e^{2s \alpha_{j+1}+2s\zeta}}{\vert z-w\vert^{2s}} \mathbb{E}\left(1_{Z_j}(z)(X_z^\e)^2\right).
\]
Hence
\begin{equation}\label{eq11}
\mathbb{E}\left(\int\int_{\vert z-w\vert \geq \e^{\alpha_j-\zeta}} \bar{X}_z^\e \bar{X}_w^\e dzdw\right) \lesssim \int\int_{\vert z-w\vert \geq \e^{\alpha_j-\zeta}}1_{z,w\in\mathrm{Supp}f}\frac{\e^{2s \alpha_{j+1}+2s\zeta}}{\vert z-w\vert^{2s}}\mathbb{E}\left(1_{Z_j}(z)\left(X_z^\e\right)^2\right)dzdw.
\end{equation}
Now using Lemma \ref{Gfnotfound} we have that
\[
\mathbb{E}\left(1_{Z_j}(z)\left(X_z^\e\right)^2\right) \lesssim \mathbb{E} \left(1_{Z_j}(z)\frac{(\mathrm{deg}(x_z^\e)+1)^2 \mathrm{diam}(H_z^\e)^2}{\mathrm{area}(H_z^\e)^2}\right).
\]
Recalling the definition of $Z_j$ in \eqref{Ejdef} we obtain that
\[
\mathbb{E}\left(1_{Z_j}(z)\left(X_z^\e\right)^2\right) \lesssim \e^{-\zeta}\mathbb{E}\left(\frac{1_{Z_j}(z)}{\mathrm{diam}(H_z^\e)^2}\right) \leq \e^{-2\alpha_j-\zeta}.
\]
Plugging this into \eqref{eq11} we have
\begin{eqnarray*}
\mathbb{E}\left(\int\int_{\vert z-w\vert \geq \e^{\alpha_j-\zeta}} \bar{X}_z^\e \bar{X}_w^\e dzdw\right) &\lesssim &\e^{2 \alpha_{j+1}-2\alpha_j+o(1)} \mathbb{E}\left(\int \int_{\vert z-w\vert \geq \e^{\alpha_j-\zeta}} 1_{z,w\in\mathrm{Supp}f}1_{Z_j}(z) \frac{1}{\vert z-w\vert^2}dwdz\right)\\
&\leq& \vert\log(\e)\vert\e^{2 \alpha_{j+1}-2\alpha_j-2\zeta+o(1)} \mathbb{E}\left(\int_{\mathbb{C}} 1_{Z_j}(z) dz\right).
\end{eqnarray*}
Since $\zeta>0$ was chosen to be arbitrarily small, this completes the proof.
\end{proof}

\subsection{End of proof of Proposition \ref{secondmom}}
In this subsection we finish the proof of Proposition \ref{secondmom}.
\begin{proof}[Proof of Proposition \ref{secondmom}.]
By \eqref{ondiagoffdiag} we have that
\begin{eqnarray*}
\mathbb{E}\left(\left(\int_{\mathbb{C}} \bar{X}_z^\e 1_{Z_j}(z) dz\right)^2\right) &=& \mathbb{E}\left(\int\int_{\vert z-w\vert \leq \e^{\alpha_j-\zeta}} 1_{Z_j}(z) 1_{Z_j}(w)\bar{X}_z^\e \bar{X}_w^\e dzdw\right)\\
&+& \mathbb{E}\left(\int\int_{\vert z-w\vert \geq \e^{\alpha_j-\zeta}} 1_{Z_j}(z) 1_{Z_j}(w)\bar{X}_z^\e\bar{X}_w^\e dzdw\right).
\end{eqnarray*}
The first term is $\lesssim \e^{-15\zeta+o(1)}$ by Proposition \ref{ondiag}, while the second term is $\lesssim \e^{-2\zeta+o(1)}$ by Proposition \ref{offdiag}. This completes the proof.
\end{proof}
\subsection{Proof of Proposition \ref{firstmom}}
In this subsection we prove Proposition \ref{firstmom}.
\begin{proof}[Proof of Proposition \ref{firstmom}.]
By Lemma \ref{smallexp}, we have that
\begin{eqnarray*}
&&\mathbb{E}\left(\left(\int_{\mathbb{C}} \mathbb{E}\left(1_{Z_j(z)}X_z^\e\vert \Phi_{\e^{\alpha_j-\zeta}}(z),\Psi_{\e^{\alpha_j-\zeta}}(z)\right)\right)^2\right)\\
&& \lesssim \mathbb{E}\left(\left(\int_{\mathbb{C}} \norm{\nabla^2 f}_\infty \mathbb{E}\left(1_{Z_j\cap \mathrm{supp } f}(z)\left.(\mathrm{deg}(z)+1)\frac{\mathrm{diam}(H_z^\e)^2}{\mathrm{area}(H_z^\e)}\right\vert\Phi_{\e^{\alpha_j-\zeta}}(z), \Psi_{\e^{\alpha_j-\zeta}}(z)\right)\right)^2\right)+\e^{o(1)}\\
&& \lesssim \mathbb{E}\left(\left(\int_{\mathrm{supp }f} (\log \e^{-1})^3 \e^{-\zeta} dz\right)^2\right)+\e^{o(1)} \lesssim O(\e^{-3\zeta})+\e^{o(1)}
\end{eqnarray*}
where in the last line we used the definition of $Z_j,$ \eqref{Ejdef}.
\end{proof}

\subsection{End of proof of Proposition \ref{mainvar}}
 We will now prove Proposition \ref{singleevent}, the only remaining part in the proof of Proposition \ref{mainvar}.
\begin{proof}[Proof of Proposition \ref{singleevent}]
We have by \eqref{varexpsplit} that
\[
\mathbb{E}\left(\left(\sum_{z \in \mathcal{V}\mathcal{G}_\e} f(\eta(z)) K_{\mathcal{G}_\e}(z)\right)^2 \right) \leq 2 \mathbb{E}\left(\left(\int_{\mathbb{C}} \bar{X}_z^\e dz\right)^2\right) + 2 \mathbb{E} \left(\left(\int_{\mathbb{C}} \mathbb{E}(X_z^\e\vert \Phi_{\e^{\alpha_j-\zeta}}(z),\Psi_{\e^{\alpha_j-\zeta}}(z))dz\right)^2\right).
\]
By Propositions \ref{secondmom} and \ref{firstmom} we know that each term is $\leq \e^{o(1)}.$ This completes the proof.
\end{proof}

\section{Transferring to the Quantum cone field}\label{quantumcone}

In this section we finish the proof of Theorem \ref{main}, that we pass from statements involving the whole plane GFF to statements involving the $\gamma$-quantum cone. The main idea is to use the following fact. Recall that $\Phi$ is a whole plane GFF. Suppose we sample a point $\mathcal{P}$ according to the Lebesgue measure on the unit disk. Let
\[
\Phi^1 := \Phi - \gamma \log \vert \cdot - \mathcal{P}\vert + \gamma \log(\max(\vert \cdot\vert,1)).
\]
Then the laws of $\Phi$ and $\Phi^1$ are mutually absolutely continuous (see Lemma A.10 in \cite{wedges}).

Now we will define three more fields to use as intermediate steps towards the quantum cone. Let
\[
\Phi^2=\Phi^1 - \gamma \log (\max(\vert \cdot\vert,1))
\]
\[
\Phi^3=\Phi^2 - \Phi^2_{\frac{1}{2}}(\mathcal{P})
\]
\[
\Phi^4=\Phi^3\left(\frac{1}{2}\cdot + \mathcal{P}\right)
\]
We will show Theorem \ref{main} step by step.
\begin{lemma}\label{h1ver}
We have Theorem \ref{main} holds if the quantum cone field $\tilde{\Phi}$ is replaced by $\Phi^1.$
\end{lemma}
\begin{proof}
Since $\Phi^1$ and $\Phi$ are mutually absolutely continuous, we can use Proposition \ref{mainbutno} directly to conclude.
\end{proof}
\begin{lemma}\label{h2ver}
We have Theorem \ref{main} holds if the quantum cone $\tilde{\Phi}$ is replaced by $\Phi^2.$
\end{lemma}
\begin{proof}
This comes directly from the fact that $\Phi^1$ and $\Phi^2$ coincide in the unit disk together with Lemma \ref{h1ver}.
\end{proof}
\begin{lemma}\label{h3ver}
We have Theorem \ref{main} holds if the quantum cone $\tilde{\Phi}$ is replaced by $\Phi^3.$
\end{lemma}
\begin{proof}
Let $G_f^{i,\e}$ denote the analogue of $G_f$ when the underlying field is replaced with $\Phi^i.$ Let $S := e^{\gamma h_{\frac{1}{2}}^2(\mathcal{P})}.$ Then recalling Definition \ref{Hvdef}, note that by scaling we have
\[
\frac{G_f^{3,S\e}}{\mathrm{area}(H_z^{S\e})} = \frac{G_f^{2,\e}}{\mathrm{area}(H_z^{\e})}
\]
Therefore by Lemma \ref{h2ver} we only need to convert integrals involving $G_f^{3,S\e}$ to integrals only involving $G_f^{3,\e}.$ Now by Lemma A.2 in \cite{gms-harmonic}, for $\mathfrak{a} \in \R$ and $\mathfrak{z} \in B_{\e^\xi}(0),$ the conditional law of $\Phi^3\vert_{B_{\frac{1}{2}}(\mathcal{P})}$ given $\{\Phi_{\frac{1}{2}}^2(\mathrm{z})=\mathfrak{a}\}\cap\{\mathcal{P}=\mathfrak{z}\}$ is mutually absolutely continuous with respect to the law of $\Phi^3\vert_{B_{\frac{1}{2}}(\mathcal{P})}$ given only $\{\mathcal{P}=\mathfrak{z}\}.$ Following the notation in \cite{gms-harmonic}, if $p>1,$ there is an $r_p=r_p(\rho)$ such that the $p$-th moment of $\mathcal{M}_{\mathfrak{a},\mathfrak{z}}$ is bounded above by a constant depending only on $p$ provided $\mathfrak{a} \in [-r_p,r_p].$ The rest of the proof follows as in step 3 in the proof of Proposition 4.10 in \cite{gms-harmonic}.
\end{proof}

\begin{lemma}\label{h4ver}
We have Theorem \ref{main} holds if the quantum cone $\Phi$ is replaced by $\Phi^4.$
\end{lemma}
\begin{proof}
By $\gamma$-LQG coordinate change, we have that
\[
\frac{G_f^{3,\e}}{\mathrm{area}(H_z^\e)}(\cdot) = C \frac{G_f^{4,2^{-\gamma Q}\e}}{\mathrm{area}(H_z^{2^{-\gamma Q}\e})}(2\cdot + \mathcal{P})
\]
for some deterministic constant $C.$ This constant does not matter in the rest of the proof: one can now follow step 4 in the proof of Proposition 4.10 in \cite{gms-harmonic}.
\end{proof}
Now the final step is to prove Theorem \ref{main}.

\begin{proof}[Proof of Theorem \ref{main}.]
If $\tilde{\Phi}$ is the $\gamma$-quantum cone field, then $\tilde{\Phi}$ and $\Phi^4$ agree in law on the unit disk. Since $\frac{G_f^\e(z)}{\mathrm{area}(H_z^\e)}$ is locally determined by the underlying fields $\tilde{\Phi},\Psi\vert_{B_{\frac{1}{2}}(z)},$ we can replace $\Phi^4$ by $\tilde{\Phi}$ in Lemma \ref{h4ver}. This completes the proof.
\end{proof}
\section{Total Curvature on a mated CRT map Cell}

We let $C$ be a fixed segment of $\tilde{\eta},$
\[
C=\tilde{\eta}([0,1]).
\]
We define $\tilde{\mathcal{G}}_\e(C)$ to be the submap of $\tilde{\mathcal{G}}_\e$ induced by $\mathcal{V}\tilde{\mathcal{G}}_\e \cap [0,1].$ We will look at the total discrete curvature inside $C.$ We recall that the discrete total curvature on a segment $C$ of $\tilde{\eta}$ is defined as the sum of the discrete curvature on every vertex in $C,$
\[
K_{\tilde{\mathcal{G}}_\e}(C) = \sum_{v \in \mathcal{V}\tilde{\mathcal{G}}_\e(C)} K_{\tilde{\mathcal{G}}_\e}(v).
\]
This represents the total curvature in $C.$

Now we will prove Theorem \ref{mainCRTthm}. Our first step will be to rewrite the total curvature in $C$ in a more convenient manner. To this end, we will need the following graph theory lemma about general triangulations with boundary.

\begin{lemma} \cite[Lemma A.7]{bg-lbm}\label{degmagic}
Let $\mathcal{T}$ be a triangulation with boundary, that is, each face has $3$ edges, except for the outer face. We let $\#\mathcal{V}T, \#\mathcal{E}T, \#\mathcal{F}T$ denote the number of vertices, edges, and faces in $T$ respectively. Let $\mathrm{Perim}(T)$ denote the perimeter of the external face, that is, the number of edges on the boundary of the external face. Then we have that
\[
\#\mathcal{E}T = 3 \#\mathcal{V}T + 3 - \mathrm{Perim}(T).
\]
\end{lemma}
We claim the total curvature on $C$ simplifies to an expression only involving degrees on the boundary of $C.$ More precisely, let $\mathrm{deg}_{\mathrm{ext}}^C$ denote the exterior degree, that is the number of edges between a vertex in $\mathcal{V}\tilde{\mathcal{G}}_\e(C)$ and $\mathcal{V}\tilde{\mathcal{G}}_\e \setminus \mathcal{V}\tilde{\mathcal{G}}_\e(C).$ Then we have the following.
\begin{prop}\label{cancellations}
We have
\[
K_{\tilde{\mathcal{G}}_\e}(C) = \frac{\pi}{3}\left(2\mathrm{Perim}(C) -6 - \sum_{v \in \mathcal{V}\tilde{\mathcal{G}}_\e(C)} \mathrm{deg}_{\mathrm{ext}}^C(v)\right).
\]
\end{prop}
\begin{proof}
We have that
\begin{eqnarray*}
\sum_{v \in \mathcal{V}\tilde{\mathcal{G}}_\e(C)} K_{\tilde{\mathcal{G}}_\e}(v) &=& \frac{\pi}{3}\sum_{v \in \mathcal{V}\tilde{\mathcal{G}}_\e(C)} (6-\mathrm{deg}(v)) =2\pi \# \mathcal{V}\tilde{\mathcal{G}}_\e(C) - \frac{\pi}{3}\sum_{v \in \mathcal{V}\tilde{\mathcal{G}}_\e(C)} \mathrm{deg}(v).
\end{eqnarray*}
Recall that $\mathrm{deg}_{\mathrm{ext}}^C(v)$ denotes the number of edges connecting $v$ to a point outside of $C,$ and similarly let $\mathrm{deg}_{\mathrm{int}}^C(v)$ denote the number of edges connecting $v$ with another vertex in $C.$ Then note that
\[
\sum_{v \in \mathcal{V}\tilde{\mathcal{G}}_\e(C)} \mathrm{deg}_{\mathrm{int}}^C(v) = 2\#\mathcal{E}C
\]
and thus
\[
\sum_{v \in \mathcal{V}\tilde{\mathcal{G}}_\e(C)} \mathrm{deg}(v) = \sum_{v \in \mathcal{V}\tilde{\mathcal{G}}_\e(C)} \mathrm{deg}_{\mathrm{ext}}^C(v) + \sum_{v \in \mathcal{V}\tilde{\mathcal{G}}_\e(C)} \mathrm{deg}_{\mathrm{int}}^C(v) = \sum_{v \in \mathcal{V}\tilde{\mathcal{G}}_\e(C)} \mathrm{deg}_{\mathrm{ext}}^C(v) + 2 \#\mathcal{E}C.
\]
This implies that
\begin{eqnarray*}
\sum_{v \in \mathcal{V}\tilde{\mathcal{G}}_\e(C)} K_{\tilde{\mathcal{G}}_\e}(v) &=& \frac{\pi}{3}\left(6 \# \mathcal{V}\tilde{\mathcal{G}}_\e(C) - \left( \sum_{v \in \mathcal{V}\tilde{\mathcal{G}}_\e(C)} \mathrm{deg}_{\mathrm{ext}}^C(v) + 2 \# \mathcal{E}C \right)\right)\\
&=& \frac{\pi}{3}\left(6 \# \mathcal{V}\tilde{\mathcal{G}}_\e(C) - 2 \# \mathcal{E}C - \sum_{v \in \mathcal{V}\tilde{\mathcal{G}}_\e(C)}\mathrm{deg}_{\mathrm{ext}}^C(v)\right).
\end{eqnarray*}
Now using Lemma \ref{degmagic}, we obtain
\begin{eqnarray*}
\sum_{v \in \mathcal{V}\tilde{\mathcal{G}}_\e(C)} K_{\tilde{\mathcal{G}}_\e}(v) &=& \frac{\pi}{3}\left(6 \# \mathcal{V}\tilde{\mathcal{G}}_\e(C) - \sum_{v \in \mathcal{V}\tilde{\mathcal{G}}_\e(C)}\mathrm{deg}_{\mathrm{ext}}^C(v) -2 \left(3 \# \mathcal{V}\tilde{\mathcal{G}}_\e(C) +3-\mathrm{Perim}(C)\right)\right)\\
&=& \frac{\pi}{3}\left(2 \mathrm{Perim}(C) -6 - \sum_{v \in \mathcal{V}\tilde{\mathcal{G}}_\e(C)} \mathrm{deg}_{\mathrm{ext}}^C(v)\right).
\end{eqnarray*}
This completes the proof.
\end{proof}

For every vertex $v$ on the boundary of $C,$ let
\[
K_{\mathrm{ext}}^C(v):=\frac{2\pi}{3}- \frac{\pi}{3}\mathrm{deg}_{\mathrm{ext}}^C(v).
\]
Then by Proposition \ref{cancellations}, we have
\begin{equation}\label{eqgeodcurv}
K_{\tilde{\mathcal{G}}_\e}(C) = \sum_{v \in \p_{\tilde{\mathcal{G}}_\e} C} K_{\mathrm{ext}}^C(v) - \frac{\pi}{2}
\end{equation}
where $\p_{\tilde{\mathcal{G}}_\e} C:= \{v \in \tilde{\mathcal{G}}_\e(C): \mathrm{deg}_{\mathrm{ext}}(v)>0\}.$ We can further split this sum. Let $C_{PL}, C_{PR}, C_{FL}, C_{FR}$ denote the past-left, past-right, future-left, future-right components of the boundary of $C,$ and let  $S_{PL}, S_{PR}, S_{FL}, S_{FR}$ denote the analogous components of $\p_{\tilde{\mathcal{G}}_\e}C$ (see Figure \ref{split}). More precisely, let
\[
C_{PL}= \{\mbox{left boundary of }\tilde{\eta}([0,1])\} \cap \{\mbox{left boundary of }\tilde{\eta}((-\infty,0])\},
\]
\[
C_{PR}= \{\mbox{right boundary of }\tilde{\eta}([0,1])\} \cap \{\mbox{right boundary of }\tilde{\eta}((-\infty,0])\},
\]
\[
C_{FL}= \{\mbox{left boundary of }\tilde{\eta}([0,1])\} \cap \{\mbox{left boundary of }\tilde{\eta}([1,\infty))\},
\]
\[
C_{FR}= \{\mbox{right boundary of }\tilde{\eta}([0,1])\} \cap \{\mbox{right boundary of }\tilde{\eta}([1,\infty))\},
\]
and let $S_{PL},S_{PR},S_{FL},S_{FR}$ be the sets of cells sharing a nontrivial arc with the corresponding piece of the boundary of $C,$ with vertices whose corresponding cells share a nontrivial boundary arc with $C_{PL},C_{PR}$ (resp. $C_{FL},C_{FR}$) are assigned to one of the sets $S_{PL}, S_{PR}$ (resp. $S_{FL}, S_{FR}$) arbitrarily.
\begin{figure}
\begin{center}
\includegraphics[width=0.4\textwidth]{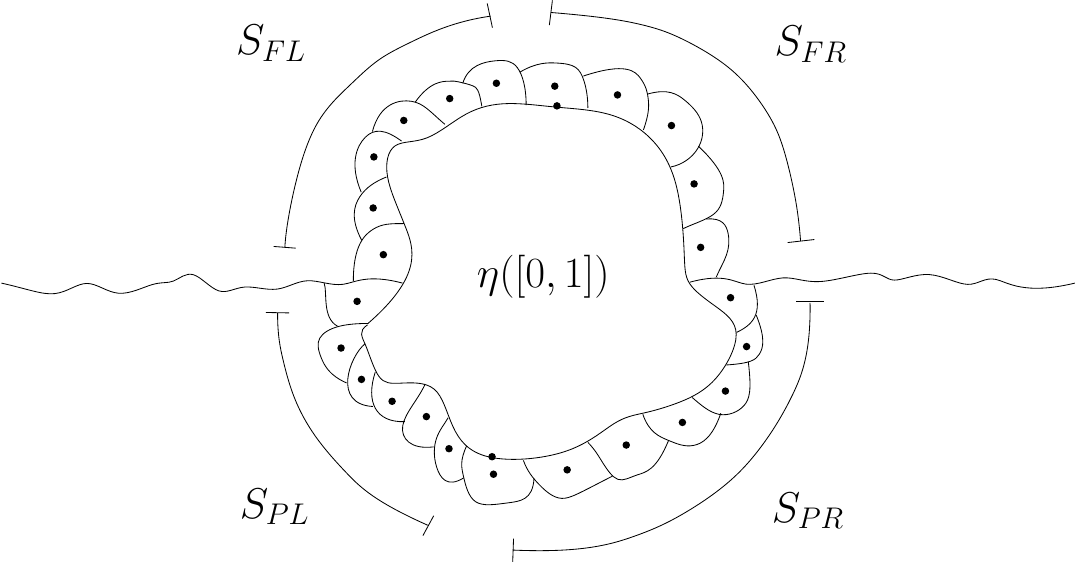}  
\caption{\label{split} Splitting the boundary of a space-filling SLE segment.}
\end{center}
\end{figure}
Then we claim the following.
\begin{prop}\label{splitgeodcurv}
There exist independent Brownian motions $B_{FL},B_{FR},B_{PL},B_{PR}$ that are also independent of $\tilde{\Phi},\tilde{\eta}$ and $\Lambda,$ and a deterministic constant $\alpha>0$ such that we have the following convergences in law:
\begin{equation}\label{PLconv}
\e^{\frac{1}{4}}\sum_{v \in S_{PL}} K_{\mathrm{ext}}^C(v) \to \alpha B_{PL}(\nu_{\tilde{\Phi}}(C_{PL})),
\end{equation}
\begin{equation}\label{PRconv}
\e^{\frac{1}{4}}\sum_{v \in S_{PR}} K_{\mathrm{ext}}^C(v) \to \alpha B_{PR}(\nu_{\tilde{\Phi}}(C_{PR})),
\end{equation}
\begin{equation}\label{FLconv}
\e^{\frac{1}{4}}\sum_{v \in S_{FL}} K_{\mathrm{ext}}^C(v) \to \alpha B_{FL}(\nu_{\tilde{\Phi}}(C_{FL})),
\end{equation}
\begin{equation}\label{FRconv}
\e^{\frac{1}{4}}\sum_{v \in S_{FR}} K_{\mathrm{ext}}^C(v) \to \alpha B_{FR}(\nu_{\tilde{\Phi}}(C_{FR})),
\end{equation}
where denotes the $\tilde{\mathcal{G}}_\e$-length, and $\nu_{\tilde{\Phi}}$ denotes the LQG length.
\end{prop}
Assuming this proposition we prove Theorem \ref{mainCRTthm}.
\begin{proof}[Proof of Theorem \ref{mainCRTthm}.]
Adding up \eqref{PLconv}, \eqref{PRconv}, \eqref{FLconv}, \eqref{FRconv} we obtain that
\begin{equation}
\e^{-\frac{1}{4}}\sum_{v \in \p_{\tilde{\mathcal{G}}_\e}C}K_{\mathrm{ext}}^C(v)\to \alpha\left(B_{PL}(\nu_{\tilde{\Phi}}(C_{PL})) + B_{PR}(\nu_{\tilde{\Phi}}(C_{PR})) + B_{FL}(\nu_{\tilde{\Phi}}(C_{FL})) + B_{FR}(\nu_{\tilde{\Phi}}(C_{FR}))\right)
\end{equation}
Using \eqref{eqgeodcurv} we obtain that
\begin{eqnarray*}
&&\e^{\frac{1}{4}} K_{\tilde{\mathcal{G}}_\e}(C)\to \alpha\left(B_{PL}(\nu_{\tilde{\Phi}}(C_{PL})) + B_{PR}(\nu_{\tilde{\Phi}}(C_{PR})) + B_{FL}(\nu_{\tilde{\Phi}}(C_{FL})) + B_{FR}(\nu_{\tilde{\Phi}}(C_{FR}))\right).
\end{eqnarray*}
Now recalling the fact that the laws of $B_{PL}, B_{PR},B_{FL},B_{FR}$ are independent Brownian motions by Proposition \ref{splitgeodcurv}, we have that
\[
B_{PL}(\nu_{\tilde{\Phi}}(C_{PL})) + B_{PR}(\nu_{\tilde{\Phi}}(C_{PR})) + B_{FL}(\nu_{\tilde{\Phi}}(C_{FL})) + B_{FR}(\nu_{\tilde{\Phi}}(C_{FR}))
\]
and the random variable $B$ in the statement of Theorem \ref{mainCRTthm} agree in law. This completes the proof.
\end{proof}
Now we will focus on proving Proposition \ref{splitgeodcurv}, more specifically, we will focus on proving \eqref{FRconv}, since the other three convergences are proved in the same way. Let $\beta:[0,\infty) \to \mathbb{C}$ be a parametrization of the right boundary of $\tilde{\eta}([0,\infty)),$ and let $S_t$ be the set of vertices corresponding to cells adjacent to $\beta([0,t])$ contained in $\tilde{\eta}([0,\infty)).$ We will now consider the cells intersecting $\beta([0,\infty))$ instead of only the right boundary of $\tilde{\eta}([0,1]).$ The reason for this is that we want to use the central limit theorem, and so it is more convenient to work up to infinity. First, we will need to define a few important quantities. For any positive integer $k,$ let
\[
\tau_k^+ := \inf\{t>0: R_t \leq -k\},
\]
and similarly
\[
\tau_k^- := \sup\{t<0: R_{t} \geq k\}.
\]

Let $\{C_i^+\}$ be the cells intersecting $\beta([0,\infty))$ contained in $\tilde{\eta}([0,\infty))$ and similarly let $\{C_i^-\}$ denote the cells intersecting $\beta([0,\infty))$ that are contained in $\tilde{\eta}((-\infty,0]),$ where $\{C_i^\pm\}$ are ordered as they are hit by $\beta.$ For each cell $C_i^\pm,$ we define the points $a_i^\pm$ such that $\left[ a_i^\pm,a_{i+1}^\pm\right]$ is the only non singleton component of $C_i^\pm \cap \beta([0,\infty))$ (we remark that when $\kappa \geq 8$ we have $C_i^\pm \cap \beta([0,\infty)) = \beta\left(\left[ a_i^\pm,a_{i+1}^\pm\right]\right)$).
Define $N_k$ by
\begin{equation}\label{Nkdef}
N_k^+ := \#\{i:a_i^+ \in \tilde{\eta}([\tau_{k-1}^+,\tau_k^+])\}, \; N_k^- := \#\{i:a_i^- \in \tilde{\eta}([\tau_k^-,\tau_{k-1}^-])\}.
\end{equation}
We will prove a few properties about $N_k^\pm,$ which will be needed in the proof of Theorem \ref{mainCRTthm} and Proposition \ref{splitgeodcurv}.
\begin{lemma}\label{indepNK}
$N_i^+$ and $N_i^-$ are determined by
\[
\left((R_t-R_{\tau_i^+})\vert_{[\tau_{i-1}^+,\tau_i^+]}, \Lambda\vert_{[\tau_{i-1}^+,\tau_i^+]}-\tau_{i-1}^+\right),
\]
\[
\left((R_t-R_{\tau_i^-})\vert_{[\tau_i^-,\tau_{i-1}^-]}, \Lambda\vert_{[\tau_i^-,\tau_{i-1}^-]}-\tau_{i-1}^-\right)
\]
respectively. Moreover, $\{N_k^+\}$ and $\{N_k^-\}$ are each i.i.d. families.
\end{lemma}
\begin{proof}
We will show the first statement only for $N_k^+,$ the argument for $N_k^-$ being identical. For the first statement, note that $N_k^+$ is the number of intervals of the form $[y_{j-1},y_j]$ containing a running minimum of $R$ contained in $[\tau_{k-1}^+,\tau_k^+].$ This in turn only depends on $ (R-R_{\tau_k^+})\vert_{[\tau_{k-1}^+,\tau_k^+]}, $ together with $\Lambda\vert_{[\tau_{k-1}^+,\tau_k^+]}.$ The second statement follows from the fact that
\[
\begin{matrix}
\left( (R_t-R_{\tau_i^+})\vert_{[\tau_{i-1}^+,\tau_i^+]}, \Lambda\vert_{[\tau_{i-1}^+,\tau_i^+]\cup [\tau_i^-,\tau^-_{i-1}]}\right),\\
\left((R_t-R_{\tau_j^+})\vert_{[\tau_{j-1}^+,\tau_j^+]}, \Lambda\vert_{[\tau_{j-1}^+,\tau_j^+]\cup [\tau_j^-,\tau^-_{j-1}]}\right)
\end{matrix}
\]
are independent for distinct $i,j$, and they determine $N_i^\pm,N_j^\pm$ respectively.
\end{proof}
\begin{lemma}\label{Nklemma}
Suppose that $\e=1,$ and let $N_k^\pm$ be defined as in \eqref{Nkdef}. Then
\[
\mathrm{Var}(N_k^\pm) < \infty.
\]
\end{lemma}
\begin{proof}
Note that it suffices to show it for $k = 1.$ Without loss of generality it suffices to show that $\mathrm{Var}(N_1^+) < \infty.$ First let $M$ denote the numbers of unit length intervals required to cover $\{t \in [0,\tau_1^+]: R_t=\inf_{0\leq s\leq t} R_s\}.$ We claim there exist constants $C_1,C_2$ such that
\[
\mathbb{P}\left(M > A\right) \leq C_1 e^{-C_2 A}.
\]
We will define a sequence of stopping times for the Brownian motion $R_t.$ Let $s_0=r_0=0,$ and suppose that $r_j,s_j$ have been defined for $0\leq j \leq n-1.$ We let
\[
s_n = \min \{t : t \geq r_{n-1},\; \inf_{0\leq x \leq t} R_x = R_t\},
\]
and
\[
r_n = s_n + 1.
\]
Let $Y_0= \max\{n:s_n \leq \tau_1^+\}.$ Then since the intervals $\{[s_n,r_n]\}_{n=1}^{Y_0}$ cover $\{t \in [0,\tau_1^+]: R_t=\inf_{0\leq s\leq t} R_s\},$ we have that $M \leq Y_0.$ Since $Y_0 \leq \tau_1^+$ and $\tau_1^+$ has an exponential tail, this proves the claim. Note in particular that this implies $\mathbb{E}(M^2)<\infty.$ Let $I_1, \ldots , I_M$ be the intervals required in the definition of $M.$ Note that conditioning on $M,$ the collection $\{\vert\Lambda\cap I_j\vert\}_{j=1}^M$ is i.i.d., where $\vert\Lambda \cap I_j\vert$ has a Poisson distribution with parameter $1.$ Therefore 
\[
N_1^+ \leq \sum_{j=1}^M \# \{j: [y_{j-1},y_j]\cap I_j \neq \varnothing\}\leq \left(\sum_{j=1}^M \vert \Lambda \cap I_j\vert\right)+M.
\]
Hence, for some absolute constant $C$ we have
\begin{eqnarray*}
\mathrm{Var}(N_1^+) \leq \mathbb{E}((N_1^+)^2) &\leq& \mathbb{E}\left(\left(\sum_{j=1}^M \vert \Lambda \cap I_j\vert+M\right)^2\right)\\
&\leq& 2 \mathbb{E}\left(\mathbb{E}\left(\left.\left(\sum_{j=1}^M \vert \Lambda \cap I_j\vert\right)^2 \right\vert M\right)\right)+2\mathbb{E}(M^2) \leq \mathbb{E}\left(CM^2 \right) < \infty,
\end{eqnarray*}
where in the second to last inequality we used Cauchy-Schwarz.
\end{proof}

For any $t \in \R$ let $D_t$ denote the mated CRT map cell containing $\beta(t),$ and let $\bar{v}_t$ denote the vertex corresponding to the cell $D_t.$ For any vertex $v \in [0,\infty),$ we also let $\mathrm{deg}_{\mathrm{ext}}^{\tilde{\eta}([0,\infty))}(v)$ be the number of edges between $v$ and a vertex in $(-\infty,0],$ and define $K_{\mathrm{ext}}^{\tilde{\eta}((0,\infty))}$ by
\[
K_{\mathrm{ext}}^{\tilde{\eta}((0,\infty))}(v) = 2- \mathrm{deg}_{\mathrm{ext}}^{\tilde{\eta}([0,\infty))}(v).
\]
We will also need the following lemma.
\begin{lemma}\label{NKidentity}
Recall that $\beta:[0,1] \to \mathbb{C}$ is a unit $\nu_{\tilde{\Phi}}$ length parametrization of the right boundary of $\tilde{\eta}([0,\infty)),$ and let $S_t$ be the set of vertices corresponding to cells adjacent to $\beta([0,t])$ contained in $\tilde{\eta}([0,\infty)).$ Then
\[
\left\vert\sum_{k \leq t} \left(N_k^+ - N_k^-\right)  -  \sum_{v \in S_t} K_{\mathrm{ext}}^{\tilde{\eta}((0,\infty))}(v)\right\vert \leq \mathrm{deg}(\bar{v}_t) + 2.
\]
Analogous statements hold replacing $C_{FR}$ by $C_{FL},$ $C_{PR},$ $C_{PL}.$
\end{lemma}
\begin{proof}
We claim that
\begin{equation}\label{bipartite}
\left\vert \sum_{v \in S_t} \mathrm{deg}_{\mathrm{ext}}^{\tilde{\eta}([0,\infty))}(v) - \sum_{k\leq t}\left(N_k^+ + N_k^-\right)\right\vert \leq \mathrm{deg}_{\mathrm{ext}}^{\tilde{\eta}([0,\infty)}(\bar{v}_t).
\end{equation}
Indeed, we can rewrite the LHS as
\[
\sum_{v \in S_t} \mathrm{deg}_{\mathrm{ext}}^{\tilde{\eta}([0,\infty))}(v) = \sum_{C_1\subset \tilde{\eta}([0,\infty)),C_2\subset \tilde{\eta}((-\infty,0])} 1_{C_1, C_2 \mbox{ are adjacent}},
\]
where the sum is taken over cells $C_1$ adjacent to $\beta([0,t))$ contained in $\tilde{\eta}([0,\infty))$ and cells $C_2$ adjacent to $\beta([0,t))$ contained in $\tilde{\eta}((-\infty,0]).$ Note that each pair of such adjacent cells $C_1,C_2$ corresponds to their shared boundary arc which is of the form $(a,b),$ where $a,b \in \{a_i^\pm\}_{i \geq 0} \cap [0,t]$ (see Figure \ref{NkNk}). However, given that there could be intervals $(a,b)$ of this form not corresponding to a pair of adjacent cells (if there is an $i$ such that $a_{i_0}^+ = \max\{a_i^{\pm}\}\cap [0,t],$ then $N_k^++N_k^-$ counts the edges incident to $C_{i_0}^+$) we obtain \eqref{bipartite}.
\begin{figure}
\begin{center}
\includegraphics[width=0.4\textwidth]{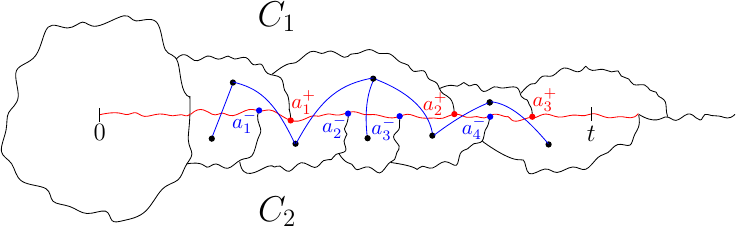}  
\caption{\label{NkNk} correspondence between pairs of adjacent cells (represented by the edges in blue) and boundary arcs (drawn in red). The points $a_i^+$ are indicated in red, while the points $a_i^-$ are indicated in blue.}
\end{center}
\end{figure}

Now, combining \eqref{bipartite} with the fact that
\[
\left\vert\sum_{v \in S_t} 2 - 2 \sum_{k\leq t} N_k^+\right\vert \leq 2,
\]
we obtain that
\begin{eqnarray*}
\left\vert\sum_{v \in S_t} K_{\mathrm{ext}}^{\tilde{\eta}([0,\infty))}(v) - \sum_{k\leq t} \left(N_k^+-N_k^-\right)\right\vert &\leq& \left\vert \sum_{v \in S_t} \mathrm{deg}_{\mathrm{ext}}^{\tilde{\eta}([0,\infty))}(v) - \sum_{k\leq t}\left(N_k^+ + N_k^-\right)\right\vert\\
&+& \left\vert\sum_{v \in S_t} 2 - 2 \sum_{k\leq t} N_k^+\right\vert\\
&\leq & \mathrm{deg}_{\mathrm{ext}}^{\tilde{\eta}([0,\infty))}(\bar{v}_t) + 2.
\end{eqnarray*}
\end{proof}

Before we prove Proposition \ref{splitgeodcurv}, we will need one more lemma. Let
\[
X_t^\e=\sum_{v \in S_t}K_{\mathrm{ext}}^{\tilde{\eta}([0,\infty))}(v).
\]
We have the following lemma.
\begin{lemma}\label{Xepsconvergence}
There exists an absolute constant $\alpha$ such that $\e^{\frac{1}{4}}X_t^\e \overset{d}{\to} \alpha B(t)$ with respect to the local uniform topology on $[0,\infty),$ where $B$ is a Brownian motion. 
\end{lemma}
\begin{proof}
Let
\[
Y_t^\e = \frac{\pi}{3}\sum_{k\leq t} \left(N_k^+-N_k^-\right). 
\]
We claim that there is an absolute constant $\alpha$ such that
\begin{equation}\label{Yclaim}
\e^{\frac{1}{4}} Y_t^\e \to \alpha B(t)
\end{equation}
in distribution as $\e \to 0,$ with respect to the local uniform topology. By rescaling, to show this claim it suffices to show that there is an absolute constant $\alpha$ such that if $\e=1,$ then
\begin{equation}\label{Yclaim2}
\frac{Y_{nt}^1}{n^{\frac{1}{2}}} \overset{d}{\to} \alpha B(t)
\end{equation}
as $n \to \infty,$ uniformly over compact sets, where $B$ is a Brownian motion. Indeed, if this convergence holds then taking $n=\e^{-\frac{1}{2}}$ we obtain 
\[
\e^{\frac{1}{4}} Y_{\left(\e^{-\frac{1}{2}}t\right)}^1 \overset{d}{\to} \alpha B(t),
\]
and now since $Y_{\e^{-\frac{1}{2}}t}^1$ has the same law as $Y_t^\e$ we conclude
\[
\e^{\frac{1}{4}} Y_t^\e \overset{d}{\to} \alpha B(t).
\]
Suppose momentarily that we have proven \eqref{Yclaim2} and thus \eqref{Yclaim}. Recalling Lemma \ref{NKidentity}, note that
\[
\left\vert X_t^\e - \frac{\pi}{3}\sum_{k \leq t} \left(N_k^+ - N_k^-\right)\right\vert\leq \mathrm{deg}(\bar{v}_t) + 2.
\]
Now using Lemma \ref{degunifboundnot}, we have that with probability tending to $1$ as $\e \to 0,$ we have 
\[
\left\vert X_t^\e - \frac{\pi}{3}\sum_{k \leq t} \left(N_k^+ - N_k^-\right)\right\vert\lesssim \left(\log \e^{-1}\right)^{-3}.
\]
Combining this with \eqref{Yclaim} we obtain
\[
\e^{\frac{1}{4}}X_t^\e = \e^{\frac{1}{4}}Y_t^\e + o_\e(1) \overset{d}{\to} \alpha B_{PL}(t).
\]
It remains to prove \eqref{Yclaim2}. From now on we will assume that $\e=1.$ Note that by symmetry we have that
\[
\mathbb{E}(N_k^+-N_k^-)=0.
\]
Therefore using Donsker's theorem together with Lemma \ref{indepNK} and \ref{Nklemma}, we obtain that
\begin{equation}\label{unindep}
\frac{\pi}{3}\sum_{k\leq nt} (N_k^+-N_k^-)= \frac{\pi}{3}\sum_{k\leq nt} (N_k^+-N_k^--\mathbb{E}(N_k^+-N_k^-)) \to \alpha B_{PL}(t),
\end{equation}
where
\[
\alpha = \frac{\pi}{3} \sqrt{\mathrm{Var}\left(N_k^+ - N_k^-\right)}.
\]
\end{proof}

Now we are ready to prove Proposition \ref{splitgeodcurv}.

\begin{proof}[Proof of Proposition \ref{splitgeodcurv}]
By taking taking $t = \nu_{\tilde{\Phi}} (S_{PR})$ and using that $S_{\nu_{\tilde{\Phi}}(S_{PR})} = S_{PR}$ in Lemma \ref{Xepsconvergence} we obtain \eqref{PLconv}. The convergences \eqref{PRconv}, \eqref{FLconv},\eqref{FRconv} are obtained in the same way. The final step is to show that the Brownian motions $B_{PL},B_{PR},B_{FL},B_{FR}$ are all independent, and independent of $\tilde{\Phi}, \tilde{\eta}$ (equivalently, from $(L,R)$). To show that the limiting Brownian motion $B_{PR}$ is independent from $L,R$, the basic idea is to approximate each of $L$, $R$, and $X^\ep$ (as defined just above Lemma~\ref{Xepsconvergence}) by processes which are exactly independent from each other. This can be done since $X^\ep$ depends only on the behavior of $L$ and $R$ in a small neighborhood of the set of times when $R$ attains a record minimum run forward or backward from time 0.

To carry this out precisely, let $\delta>0,$ and let $w^+_0=0.$ We define $w_k^+$ by
\[
w_k^+ = \inf\left\{t \geq w_{k-1}^++\delta: R_t = \inf_{0 \leq s\leq t} R_s\right\}. 
\]
We recall that $\Lambda^\e = \{y_k\}.$ We let
\[
\tilde{S}_t := \{y_k \in S_t: [y_{k-1},y_k] \subseteq [w_j^+, w_j^++\delta] \mbox{ for some }j\} 
\]
and we define
\[
\tilde{X}_t^\e := \sum_{v \in \tilde{S}_t} \left(2-\widetilde{\mathrm{deg}}_{\mathrm{ext}}(v)\right),
\]
where $\widetilde{\mathrm{deg}}_{\mathrm{ext}}(v)$ is defined as
\[
\widetilde{\mathrm{deg}}_{\mathrm{ext}}(v) :=\# \{i: \beta([a_i^-,a_{i+1}^-]), H_v^\e \mbox{ share a nontrivial boundary arc}, [a_i^-,a_{i+1}^-] \subseteq [w_j^+,w_j^++\delta] \mbox{ for some }j\},
\]
that is, $\widetilde{\mathrm{deg}}_{\mathrm{ext}}(v)$ is the number of cells adjacent to both $\beta([0,\infty))$ and the cell $H_v^\e$ corresponding to $v,$ such that its intersection with $\beta([0,\infty))$ is contained in an interval of the form $[w_j^+,w_j^++\delta].$

We claim that $\tilde{X}^\e_t\vert_{0 \leq t\leq T}$ is determined by
\begin{equation}\label{info}
\begin{matrix}
\left\{(R_t-R_{w_k^+})\vert_{t \in [w_k^+,w_k^++\delta]}, \left(\Lambda_\e\cap [w_k^+,w_k^++\delta]\right)-w_k^+\right\}_{k: 0\leq w_k^+ \leq T}\\
\left\{(R_t-R_{w_k^-})\vert_{t \in [w_k^-,w_k^-+\delta]}, \left(\Lambda_\e\cap [w_k^-,w_k^-+\delta]\right)-w_k^-\right\}_{k: 0\leq w_k^- \leq T}
\end{matrix}
\end{equation}
with probability going to $1$ as $\e \to 0.$ More precisely, there is an event $E_{\e,\delta}$ such that $X^\e_t\vert_{0 \leq t\leq T} 1_{E_{\e,\delta}}$ is determined by \eqref{info} and $\mathbb{P}(E_{\e,\delta}) \to 1$ as $\e \to 0.$ Indeed, $\{\tilde{X}_t^\e\}_{0\leq t \leq T}$ is determined by $\widetilde{\mathrm{deg}}_{\mathrm{ext}}(v)$ for $v \in \tilde{S}_T.$ Recall that $\Lambda^\e=\{y_k\}_{k \in \mathbb{Z}}.$ Note that if for all $k$ such that $0 \leq y_k \leq T$ we have $\vert y_k-y_{k-1}\vert \leq \delta,$ then
\begin{equation}\label{containment}
\left\{t:R_t = \inf_{0\leq s\leq t} R_s\right\} \subseteq \bigcup_{k: 0 \leq w_k^+ \leq T} [w_k^+,w_k^++\delta].
\end{equation}
Let $E_{\e,\delta}$ be the event that 
\[
\sup_{k: k \geq 0, y_k \leq T} \vert y_k-y_{k-1}\vert \leq \delta.
\]
Then
\begin{comment}
\begin{eqnarray*}
\mathbb{P}(E_{\e,\delta}) &=& 1-e^{-\e^{-1}T}- \sum_{k=1}^\infty \mathbb{P}(\vert \{i:i \geq 0, y_i \leq T\} = k\vert) \left(\int_\delta ^\infty \e^{-1} ke^{-\e^{-1}t}dt\right)\\
&=& 1-e^{-\e^{-1}T}-\sum_{k=1}^\infty \frac{\left(\e^{-1}T\right)^ke^{-\e^{-1}T}}{k!} k e^{-\e^{-1}\delta}\\
&=& 1-e^{-\e^{-1}T}-\e^{-1}Te^{\e^{-1}\delta} \sum_{k=0}^\infty \frac{\left(\e^{-1}T\right)^ke^{-\e^{-1}T}}{k!}\\
&=& 1-e^{-\e^{-1}T}- \e^{-1}Te^{\e^{-1}\delta} \to 0
\end{eqnarray*}
as $\e \to 0.$ Hence
\end{comment}
$\mathbb{P}(E_{\e,\delta})\to 1$ as $\e \to 0.$ Moreover, if $E_{\e,\delta}$ holds, then $\{\widetilde{\mathrm{deg}}_{\mathrm{ext}}(v)\}_{v\in S_T}$ is determined by $R$'s running minima in $[w_k^+,w_k^++\delta]\cup [w_k^-,w_k^-+\delta]$ and $\left(\Lambda_\e\cap [w_k^+,w_k^++\delta]\right)-w_k^+,$ $\left(\Lambda_\e\cap [w_k^-,w_k^-+\delta]\right)-w_k^-$ and thus by \eqref{containment}, $\tilde{X}_t^\e\vert_{0\leq t\leq T}$ is determined by \eqref{info}, if $E_{\e,\delta}$ holds. Now let 
\[
Z_t := (R_t,L_t).
\]
We claim that
\[
\left\vert \sum_{k : w_k^+ \leq T} (Z_{w_k^++\delta}-Z_{w_k^+})\right\vert \approx \delta^{\frac{1}{4}}.
\]
Indeed, since $\{w_k^+\}$ is a collection of stopping times, if we let $\bar{K}(t)$ be the random variable defined by
\[
\bar{K}(t) := \#\{k: w_k^+ \leq t\},
\]
then we have
\[
\sum_{k:w_k^+ \leq T} (Z_{w_k^++\delta}-Z_{w_k^+})
\]
is stochastically dominated by $B_{\bar{K}(T)\delta},$ where $B$ is a two dimensional correlated Brownian motion independent of $\bar{K}(T).$ Therefore
\[
\sup_{0 \leq t \leq T} \left\vert \sum_{k:w_k^+ \leq t} (Z_{w_k^++\delta}-Z_{w_k^+})\right\vert \overset{d}{=} \sup_{0\leq t\leq T} \vert B_{\bar{K}(t)\delta}\vert.
\]
Note that
\begin{eqnarray*}
\mathbb{P}(\vert B_{\bar{K}(t)\delta}\vert>C\delta^{\frac{1}{4}}) &=& \sum_{k\geq 0} \mathbb{P}(\vert B_{k\delta}\vert> C\delta^{\frac{1}{4}})\mathbb{P}(\bar{K}(t)=k)\\
&=& \sum_{k \geq 0}\mathbb{P}\left(\vert B_1\vert > C\delta^{-\frac{1}{4}}k^{-\frac{1}{2}}\right)\mathbb{P}(\bar{K}(t)=k)\\
&\leq& \sum_{k \leq C\delta^{-\frac{1}{2}}} \mathbb{P}\left(\vert B_1\vert > C^{\frac{1}{2}}\right) + \mathbb{P}(\delta^{\frac{1}{2}}\bar{K}(t)> C) = o_\delta(1)
\end{eqnarray*}
where $o_\delta(1) \to 0$ as $\delta \to 0$ and where $C=C(\delta)$ is chosen so that $C(\delta)\to \infty.$ Hence if we let
\[
\tilde{Z}_t^\delta := \sum_{k=0}^{\bar{K}(t)} \left(Z_{w_{k+1}^+}-Z_{w_k^++\delta}\right)
\]
then $\tilde{R}_t^\delta \to R_t$ in law as $\delta \to 0,$ since with probability going to $1$ as $\delta\to0,$ we have
\begin{equation}\label{Zcloseness}
\sup_{0 \leq t \leq T}\left\vert Z_t - \tilde{Z}_t^\delta\right\vert \leq o_\delta(1).
\end{equation}
Finally, note that
\begin{equation}\label{comparisonXeps}
\vert \tilde{X}_t^\e-X_t^\e\vert \lesssim \delta^{-\frac{1}{2}} (\log \e^{-1})^3
\end{equation}
with probability going to $1$ as $\e \to0.$ Indeed, since every cell adjacent to $\beta([0,\infty))$ that isn't in $\{H_{w_k^+}^\e\}$ is contained in an interval of the form $[w_j^+,w_j^++\delta]$ (see Figure \ref{hopefullylast}), we can write

\begin{figure}[ht!]
\begin{center}
\includegraphics[width=0.75\textwidth]{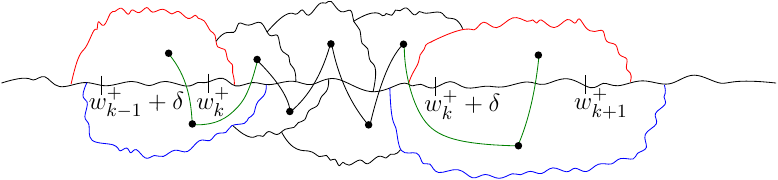} 
\caption{\label{hopefullylast} The red cells and blue are counted in $X_t^\e,$ but not in $\tilde{X}_t^\e.$ The green edges are those ``lost" when computing $\tilde{X}_t^\e.$
}
\end{center}
\vspace{-3ex}
\end{figure}

\begin{eqnarray}\label{eq15}
\vert X_t^\e - \tilde{X}_t^\e\vert &\leq &\sum_{k=0}^{\bar{K}(t)} \mathrm{deg}_{\mathrm{ext}}(\bar{v}_{w_k^+}) + \sum_{v \in \tilde{S}(t)} \vert \widetilde{\mathrm{deg}}_{\mathrm{ext}}(v) - \mathrm{deg}_\mathrm{ext}(v)\vert\nonumber\\
&\lesssim& \sum_{v \in \tilde{S}(t)} \vert \widetilde{\mathrm{deg}}_{\mathrm{ext}}(v) - \mathrm{deg}_\mathrm{ext}(v)\vert + \bar{K}(t) (\log \e^{-1})^3\nonumber\\
&\lesssim& \sum_{v \in \tilde{S}(t)} \vert \widetilde{\mathrm{deg}}_{\mathrm{ext}}(v) - \mathrm{deg}_\mathrm{ext}(v)\vert+ \delta^{-\frac{1}{2}} (\log \e^{-1})^3,
\end{eqnarray}
where we used Lemma \ref{degunifbound} and the fact that $\bar{K}(t)\lesssim \delta^{-\frac{1}{2}}$ due to the fact that the set of running minima of a Brownian motion has Minkowski dimension $\frac{1}{2}.$ Now note that
\begin{eqnarray*}
&&\vert \widetilde{\mathrm{deg}}_{\mathrm{ext}}(v) - \mathrm{deg}_\mathrm{ext}(v)\vert\\
&&\leq \#\{i: \beta([a_i^-,a_{i+1}^-]), H_v^\e \mbox{ share a nontrivial boundary arc}, \nexists j: [a_i^-,a_{i+1}^-] \subseteq [w_j^+,w_j^++\delta]\},
\end{eqnarray*}
and so if we let $S_0(t)$ denote the set of vertices corresponding to cells adjacent to $\beta([0,t))$ contained in $\tilde{\eta}((-\infty,0]),$ and $\mathrm{deg}_0(v)$ denote the number of cells contained in $\tilde{\eta}([0,\infty))$ adjacent to both $\beta([0,\infty))$ and the cell corresponding to $v,$ proceeding as in \eqref{eq15} we have
\begin{eqnarray*}
\vert X_t^\e - \tilde{X}_t^\e\vert \lesssim \sum_{v \in S_0(t)} \mathrm{deg}_0(v)+\delta^{-\frac{1}{2}}(\log \e^{-1})^3\lesssim \delta^{-\frac{1}{2}}(\log \e^{-1})^3.
\end{eqnarray*}
\begin{comment}
Repeating the same argument with $L$ instead, defining $w_k^-$ and $\tilde{L}_t$ analogously, and also repeating the argument for $t<0$ we see that if we let
\[
\tilde{L}_t^\delta := \sum_{k=0}^{\bar{K}(t)} \left(L_{w_{k+1}^-}-L_{w_k^-+\delta}\right),
\]
then $\tilde{L}_t^\delta \to L_t$ in probability.
\end{comment}
Moreover, note that $\tilde{Z}_t^\delta$ depends only on the collection $\left(Z_t - Z_{w_k^+}\right)\vert_{t \in [w_k^+ , w_k^++\delta]}$ for $k \geq 0.$ Therefore combining \eqref{Zcloseness} together with Lemma \ref{Xepsconvergence} we obtain the result.

We have shown that $(L,R, \ep^{1/4} X^\ep)$ converges in distribution as $\ep \to 0$ to $(L,R,\alpha B_{PR})$, where $B_{PR}$ is a Brownian motion independent from $(L,R),$ and where $\alpha$ is as in Lemma \ref{Xepsconvergence}. It remains to prove a joint convergence statement for processes associated with four different boundary arcs.

Write $A^\ep_{PR} = X^\ep$. Define $A^\ep_{PL}$ analogously to $X^\ep$ but with the roles of $L$ and $R$ interchanged. Define $A^\ep_{FR}$ and $A^\ep_{FL}$ analogously to $A^\ep_{PR}$ and $A^\ep_{PL}$ but with $\tilde{\eta}((-\infty,0])$ and $\tilde{\eta}([0,\infty))$ replaced by $\tilde{\eta}([1,\infty))$ and $\tilde{\eta}((-\infty,1])$, respectively. We claim that 
\begin{equation} \label{eqn:bm-joint-limit}
(L,R,\ep^{1/4}  A_{PR}^\ep, \ep^{1/4} A_{PL}^\ep,\ep^{1/4} A_{FR}^\ep,\ep^{1/4}  A_{FL}^\ep) \overset{d}{\to}
(L,R,\alpha B_{PR} , \alpha B_{PL} , \alpha B_{FR} , \alpha B_{FL} )
\end{equation}
where the conditional law of $(B_{PR} , B_{PL} , B_{FR} , B_{FL} )$ given $(L,R)$ is described as follows.
\begin{itemize}
\item 
The conditional law of $(B_{PR}, B_{FR})$ is that of a pair of Brownian motions coupled together so that that $B_{PR}|_{[0,\nu_{\tilde{\Phi}}(C_{PR})]}$ and $B_{FR}|_{[0,\nu_{\tilde{\Phi}}(C_{FR})]}$ are independent and $B_{PR}(t + \nu_{\tilde{\Phi}}(C_{PR})) - B_{PR}(\nu_{\tilde{\Phi}}(C_{PR})) = B_{PR}(t + \nu_{\tilde{\Phi}}(C_{FR})) - B_{FR}(\nu_{\tilde{\Phi}}(C_{FR}))$ for each $t\geq 0$. 
\item 
The conditional law of $(B_{PL}, B_{FL})$ is described in the same way but with $L$ in place of $R$ throughout.
\item
$(B_{PR}, B_{FR})$  and $(B_{PL}, B_{FL})$ are conditionally independent. 
\end{itemize}

We will not provide a fully detailed proof of~\eqref{eqn:bm-joint-limit} since the argument uses very similar ideas to the proof of the convergence $(L,R,X^\ep) \overset{d}{\to} (L,R,\alpha B_{PR})$. The idea of the proof is as follows. 
Let $\delta > 0$. Let $w_{k,PR}^\pm = w_k^\pm$ for $k\geq 0$ and let $\wt A_{PR}^\ep(t) = \wt X_t^\ep$, defined as above. Define $w_{k,FR}^\pm$, $w_{k,PL}^\pm$, $w_{k,FL}^\pm$ and $\wt A_{FR}^\ep $, $\wt A_{PL}^\ep $, $\wt A_{FL}^\ep $  in an analogous manner. 

We claim that if $\delta > 0$ is small, it holds with high probability that the number of pairs of indices $j,k$ such that the following is true is of order $o(\delta^{-1})$: the intervals $[w_{j,PR}^\pm  , w_{j,PR}^\pm + \delta]$ and $[w_{k,PL}^\pm  , w_{k,PL}^\pm + \delta] $ have non-empty intersection and the images under $\tilde{\eta}$ of each of these intervals intersects $\tilde{\eta}([0,1])$ . Moreover the same holds with $PR, PL$ replaced by any two distinct elements of $\{PR, FR, PL, FL\}$. Indeed, in the case when $\kappa \geq 8$ ($\gamma \leq \sqrt 2$), this follows from the continuity of $\tilde{\eta}$ and the fact that the intersection of the left and right boundaries of $\tilde{\eta}([0,\infty))$ is a.s.\ the single point $\{0\}$, and similarly for $\tilde{\eta}([1,\infty))$. When $\kappa \in (4,8)$, the left and right boundaries of $\tilde{\eta}([0,\infty))$ a.s.\ intersect in an uncountable set. But, the set of times when $\tilde{\eta}$ hits this set has the law of the range of a stable subordinator of index $1-\kappa/8 < 1/2$ (see \cite[Lemma 3.5]{gm-characterizations}), so set of times has Minkowski dimension strictly smaller than $1/2$. 

Using the claim in the previous paragraph and similar arguments as in the proof of the convergence $(L,R,X^\ep) \overset{d}{\to} (L,R,\alpha B_{PR})$, we can define further approximations $\rng L^\delta , \rng R^\delta$, $\rng A^\ep_{PR} $, $\rng A_{FR}^\ep $, $\rng A_{PL}^\ep $, $\rng A_{FL}^\ep $ (depending on $\ep$ and $\delta$) which exactly satisfy a version of the independence properties stated just after~\eqref{eqn:bm-joint-limit} and which are close to $L, R, A_{PR}^\ep , A_{FR}^\ep , A_{PL}^\ep, A_{PR}^\ep$, respectively, with high probability when $\ep \to 0$ and then $\delta \to 0$ (in similar sense to \eqref{Zcloseness}, \eqref{comparisonXeps}). 
This completes the proof of~\eqref{eqn:bm-joint-limit}, which in turn implies the proposition statement.
\end{proof}

\begin{comment}

Let
\[
A_{PL}^\e(t) := \e^{\frac{1}{4}}\tilde{X}_t^\e,
\]
and define $A_{PR}^\e(t),A_{FL}^\e(t), A_{FR}^\e(t)$ analogously. Then by Lemma \ref{Xepsconvergence} , we have that
\[
(L,R,\Lambda_\e,A_{PL}^\e(t),A_{PR}^\e(t),A_{FL}^\e(t),A_{FR}^\e(t))
\]
converges in law to 
\[
(L,R,\Lambda_\e,B_{PL}(t),B_{PR}(t),B_{FL}(t),B_{FR}(t)).
\] 
Using the fact that $\tilde{R}_t^\delta$ and $\tilde{L}_t^\delta$ depend only on $R_t\vert_{t \in \bigcup_{k \geq 0}[w_k^++\delta , w_{k+1}^+]}$ and $L_t\vert_{t \in \bigcup_{k \geq 0}[w_k^\ell+\delta , w_{k+1}^\ell]}$ respectively, we see that for small enough $\e = \e(\delta),$ we have $(\tilde{L}_t^\delta,\tilde{R}_t^\delta,\Lambda_\e,A_{PL}^{\e(\delta)}(t),A_{PR}^{\e(\delta)}(t),A_{FL}^{\e(\delta)}(t),A_{FR}^{\e(\delta)}(t))$ converges to $(L,R,\Lambda_\e,B_{PL},B_{PR},B_{FL},B_{FR})$ in law. Since for any $\delta,$ in the first expression $A_{PL}^{\e(\delta)},$ $A_{PR}^{\e(\delta)},$ $A_{FL}^{\e(\delta)},$ $A_{FR}^{\e(\delta)}$ are independent, so are $B_{PL},B_{PR},B_{FL},B_{FR}.$ This completes the proof.

\end{comment}

\bibliography{ref}
\bibliographystyle{hmralphaabbrv}
\end{document}